\documentclass[]{article}

\usepackage[margin=1in]{geometry}
\usepackage{amssymb,amsmath,amsthm}
\usepackage{bbm}
\usepackage[shortlabels]{enumitem}
\usepackage{graphicx}
\usepackage{xcolor}
\usepackage{hyperref}
\usepackage{authblk}
\usepackage{hyperref}
\hypersetup{
    colorlinks,
    linkcolor={blue!20!black},
    citecolor={teal!80!black},
    urlcolor={black}
}

\allowdisplaybreaks

\theoremstyle{plain}
\newtheorem{thm}{Theorem}[section]
\newtheorem{cor}[thm]{Corollary}
\newtheorem{lem}[thm]{Lemma}
\newtheorem{prop}[thm]{Proposition}
 
\theoremstyle{definition}
\newtheorem{defn}[thm]{Definition}
\newtheorem{rem}[thm]{Remark}
\newtheorem*{rem*}{Remark}
\newtheorem*{acks}{Acknowledgements}

\numberwithin{equation}{section}

\newcommand{\C}{\mathbb{C}}
\newcommand{\R}{\mathbb{R}}
\newcommand{\Q}{\mathbb{Q}}
\newcommand{\Z}{\mathbb{Z}}
\newcommand{\N}{\mathbb{N}}
\newcommand{\Exp}{\mathbb{E}}
\newcommand{\Prob}{\mathbb{P}}
\newcommand{\bd}{\partial}
\newcommand{\ud}{\mathbb{D}}
\newcommand{\ep}{\varepsilon}

\title{The Minkowski content measure for the Liouville quantum gravity metric}
\date{}
\author{Ewain Gwynne\thanks{\protect\url{ewain@uchicago.edu}}\; }
\author{\;Jinwoo Sung\thanks{\protect\url{jsung@math.uchicago.edu}}}
\affil{University of Chicago}

\begin{document}

\maketitle

\begin{abstract}
    A Liouville quantum gravity (LQG) surface is a natural random two-dimensional surface, initially formulated as a random measure space and later as a random metric space. We show that the LQG measure can be recovered as the Minkowski measure with respect to the LQG metric, answering a question of Gwynne and Miller \cite{gm-uniqueness}. As a consequence, we prove that the metric structure of a $\gamma$-LQG surface determines its conformal structure for every $\gamma \in (0,2)$. Our primary tool is the continuum mating-of-trees theory for space-filling SLE. In the course of our proof, we also establish a H\"older continuity result for space-filling SLE with respect to the LQG metric.
\end{abstract}

\section{Introduction}

\subsection{Overview}

Fix $\gamma \in (0,2)$ and let $U \subset \C$ be an open domain. Let $h$ be the Gaussian free field (GFF) on $U$ or a minor variant thereof. The $\gamma$-Liouville quantum gravity (LQG) surface described by $(U,h)$ is formally the two-dimensional Riemannian manifold with metric tensor 
\begin{equation}\label{eqn:lqg-def-formal} 
    e^{\gamma h} (dx^2 + dy^2)
\end{equation} 
where $dx^2 + dy^2$ is the Euclidean metric tensor. LQG was introduced by Polyakov \cite{polyakov-qg1} in the physics literature as a canonical model of two-dimensional random geometry. Since then, LQG has been identified as the scaling limit of various types of random planar maps, as described in surveys \cite{gwynne-ams-survey,ghs-mating-survey,miller-icm,sheffield-icm}. 

The expression \eqref{eqn:lqg-def-formal} does not make literal sense because the GFF $h$ does not admit pointwise values; it is defined only as a random distribution. Nevertheless, it is possible to rigorously construct the area measure and distance function corresponding to~\eqref{eqn:lqg-def-formal} through regularization and renormalization techniques.

The $\gamma$-LQG volume measure $\mu_h$ is the Gaussian multiplicative chaos measure associated with the GFF $h$, which can be constructed via various regularization methods \cite{aru-gmc-survey,berestycki-gmt-elementary,shef-kpz,kahane,rhodes-vargas-review}. In the context of LQG, it was first defined by Duplantier and Sheffield \cite{shef-kpz} to be the almost sure weak limit
\begin{equation} 
    \mu_h:= \lim_{\ep \to 0} \ep^{\gamma^2/2} e^{\gamma h_\ep(z)}\,d^2z, 
\end{equation}
where $h_\ep(z)$ is the average of $h$ on the circle $\bd B_\ep(z)$ and $d^2z$ is the Lebesgue measure on $U$. Note that for a smooth function $f:U \to \R$, the volume form associated with the Riemannian metric tensor $e^f(dx^2 + dy^2)$ is $e^f d^2z$. With probability one, $\mu_h$ is mutually singular with respect to the Lebesgue measure but has no point masses and assigns a positive mass to any open set. The circle average approximation can be replaced by other mollification methods for the GFF; see \cite{aru-gmc-survey,berestycki-gmt-elementary,kahane,rhodes-vargas-review} for further details. Notably, let $p_{\ep^2/2}(z) = \frac{1}{\pi \ep^2} \exp(-|z|^2/\ep^2)$ be the heat kernel at time $\ep^2/2$ and let 
\begin{equation}
    h_\ep^*(z) := (h * p_{\ep^2/2})(z) = \int_{U} h(w) p_{\ep^2/2}(z-w) \,d^2w.
\end{equation}
Then $\lim_{\ep \to 0}\ep^{\gamma^2/2}  e^{\gamma h_\ep^*(z)} d^2z = \mu_h$ in probability with respect to the topology of weak convergence of measures.\footnote{Most works on Gaussian multiplicative chaos require GFF to be regularized using compactly supported mollifiers. The heat kernel mollification of GFF is considered in \cite{rhodes-vargas-review} (which calls it ``white noise decomposition"), where it is shown that $\ep^{\gamma^2/2} e^{\gamma h_\ep^*(z)}d^2z$ converges to $\mu_h$ in law. One can extend this to convergence in probability by applying existing methods, for instance by checking \cite[Lemma~3.5]{berestycki-gmt-elementary} using the covariance formula $\mathrm{Cov}(h_{\ep}^*(z), h_r^*(w)) = \pi \int_{(\ep^2+r^2)/2}^\infty p_t(z-w)\,dt$. }

More recently, the $\gamma$-LQG metric $D_h$ was constructed as the scaling limit of Liouville first passage percolation (LFPP) in the case that $U = \C$ and $h$ is a whole-plane GFF. For $z,w\in \C$ and $\ep>0$, the $\ep$-LFPP metric is defined by 
\begin{equation}\label{eqn:LFPP-metric}
    D_h^\ep(z,w) := \inf_{P: z\to w} \int_0^1 e^{\xi h_\ep^*(P(t))} |P'(t)|\,dt   ,
\end{equation}
where the infimum is taken over all piecewise continuously differentiable paths $P$ from $z$ to $w$. Note that for a smooth function $f:\C \to \R$, the distance between $z$ and $w$ corresponding to the Riemannian metric tensor $e^f(dx^2 + dy^2)$ is $\inf_{P:z\to w} \int_0^1 e^{f(P(t))/2}|P'(t)|\,dt$. Here, 
\begin{equation}
    \xi := \frac{\gamma}{d_\gamma},
\end{equation}
where $d_\gamma>2$ is the dimension of $\gamma$-LQG. The constant $d_\gamma$ is obtained a priori as a scaling exponent corresponding to various approximations of the $\gamma$-LQG metric space \cite{dg-lqg-dim,dzz-heat-kernel}. Ding, Dub\'edat, Dunlap, and Falconet \cite{dddf-lfpp} showed that for a suitable choice of deterministic scaling constants $\{\mathfrak a_\ep\}_{\ep >0}$, the family of rescaled LFPP metrics $\{\mathfrak a_\ep^{-1} D_h^\ep\}_{\ep \in (0,1)}$ is tight. Building furthermore on \cite{lqg-metric-estimates,gm-confluence,local-metrics}, Gwynne and Miller \cite{gm-uniqueness} showed that the subsequential limit is unique and defined $\gamma$-LQG metric as the limit 
\begin{equation}\label{eqn:lqg-metric-def}
    D_h:= \lim_{\ep \to 0} \mathfrak a_\ep^{-1} D_h^\ep
\end{equation} 
in probability with respect to the local uniform topology on $\C \times \C$. In particular, $D_h$ almost surely induces the Euclidean topology. A posteriori, $d_\gamma$ was identified as both the Hausdorff dimension \cite{gp-kpz} and the Minkowski dimension \cite{afs-metric-ball} of the $\gamma$-LQG metric space $(\C, D_h)$. 

It is natural to ask how the measure $\mu_h$ and the metric $D_h$ are related. In this paper, we show that $\mu_h$ is almost surely equal to the Minkowski content measure with respect to $D_h$ (Theorem~\ref{thm:main-thm}). This answers~\cite[Problem 7.10]{gm-uniqueness}. Our result can be viewed as an LQG analog of \cite{lawler-rezai-nat}, which constructed the Minkowski content for Schramm--Loewner evolution (SLE) curves and showed that it is equivalent to the so-called \emph{natural parameterization} of SLE~\cite{lawler-shef-nat}.

A particular consequence of our result is that $D_h$ almost surely determines $\mu_h$. It was shown in \cite[Theorem~1.3]{afs-metric-ball} that the pointed metric measure space $(\C, 0, D_h, \mu_h)$ almost surely determines $h$ up to rotation and scaling. Therefore, our result shows that the pointed metric space $(\C, 0, D_h)$ almost surely determines $h$ modulo rotation and scaling (Corollary~\ref{cor:forgetting-parameterization}).

The primary tool in our proofs is the \emph{mating-of-trees} theorem of Duplantier, Miller, and Sheffield~\cite{dms-lqg-mating}. This theorem says that the left/right boundary length process for a space-filling SLE curve $\eta$ on an LQG surface is a correlated two-dimensional Brownian motion. Roughly speaking, this result is useful for two reasons. First, it gives a source of exact independence since the LQG surfaces traced by the curve $\eta$ during disjoint time intervals are independent. In particular, this leads to a short proof of a lower bound for the number of LQG metric balls needed to cover a given set (Proposition~\ref{prop:covering-number-lower-bound}) without needing a separate two-point estimate. Second, the mating-of-trees theory provides a convenient way of decomposing space into regions of equal LQG mass, namely the segments $\eta([x-\ep,x])$ for $x\in\ep\mathbb Z$, where $\eta$ is parameterized by $\mu_h$-mass. See Section~\ref{sec:proof-strategy} for more details on the proof method.

The mating-of-trees theory has many applications in the study of random conformal geometry, including LQG and SLE; see~\cite{ghs-map-dist} for a survey of these applications. This is the first paper to use this theory to prove properties of the LQG metric for general $\gamma\in (0,2)$. We expect that there will be more applications of the mating-of-trees theory to the LQG metric in the future. In the course of our proof, we obtain some estimates for space-filling SLE on an LQG surface, which are of independent interest. We especially highlight the H\"older continuity result (Theorem~\ref{thm:sle-holder}), which is already used in the paper~\cite{bgs-meander} to prove a result about random permutations.

Previously, Le Gall \cite{legall-measure} showed that the volume measure of the Brownian sphere is equal to a constant multiple of the Hausdorff measure associated with the gauge function $r^4 \log \log (1/r)$. This implies that $D_h$ almost surely determines $\mu_h$ for $\gamma = \sqrt{8/3}$, due to the equivalence between the Brownian sphere and the $\sqrt{8/3}$-LQG sphere established by Miller and Sheffield \cite{lqg-tbm1,lqg-tbm2,lqg-tbm3}. See also \cite[Corollary~1.4]{gm-uniqueness} for the fact that the Miller--Sheffield metric agrees with the limit of LFPP. In a sense, our result is a generalization of Le Gall's result to general $\gamma\in (0,2)$, but with Minkowski content instead of Hausdorff measure. Our proof and Le Gall's have some superficial similarities, but the main techniques are fundamentally different.

\begin{acks}
We thank the anonymous referee for useful remarks on an earlier version of this work and Greg Lawler for helpful discussions.
E.G.\ was partially supported by a Clay research fellowship. J.S.\ was partially supported by a scholarship from Kwanjeong Educational Foundation.
\end{acks}

\subsection{Main results}

We say that a random distribution $h$ on $\C$ is a \textit{whole-plane GFF plus a continuous function} if there is a coupling of $h$ with a random continuous function $f: \C \to \R$ such that $h - f$ has the law of a whole-plane GFF (see Section~\ref{sec:whole-plane-gff} for the definition of the whole-plane GFF). Our main theorem states that for each $\gamma \in (0,2)$, the Minkowski content measure (see Definition~\ref{def:minkowski-content-measure}) exists on the $\gamma$-LQG metric space $(\C,D_h)$ and is equal to the corresponding $\gamma$-LQG measure $\mu_h$.

\begin{thm}\label{thm:main-thm}
    Fix $\gamma \in (0,2)$. There exists a deterministic sequence $\{\mathfrak b_\ep\}_{\ep>0}$ such that the following is true.
    
    Let $h$ be a whole-plane GFF plus a continuous function and $D_h$ be the corresponding $\gamma$-LQG metric. Suppose $A\subset \C$ is either:
    \begin{enumerate}[(i)]
    	\item a deterministic bounded Borel set; or
		\item a random compact set, measurable with respect to the Borel $\sigma$-algebra induced by the Hausdorff distance on $\C$, that is coupled with the random distribution $h$
    \end{enumerate}
	such that $\mu_h(\bd A) = 0$ almost surely.
    Let $N_\ep(A;D_h)$ be the minimum number of $D_h$-balls of radius $\ep>0$ required to cover $A$. Then,
    \begin{equation}\label{eqn:main-theorem}
        \lim_{\ep \to 0} \mathfrak b_\ep^{-1} N_\ep(A;D_h) = \mu_h(A) \quad \text{in probability.}
    \end{equation}   
\end{thm}

We do not exclude the possibility that the limit \eqref{eqn:main-theorem} may hold almost surely. However, our methods are insufficient to prove how fast the sequence $\mathfrak b_\ep^{-1} N_\ep(A;D_h)$ converges. We expect that stronger LQG metric estimates than those currently available are necessary to prove quantitative bounds required for almost sure convergence.

A corollary of Theorem~\ref{thm:main-thm} is that the metric space structure of a $\gamma$-LQG surface is sufficient information to determine not only its metric measure space structure but also its conformal structure, in the sense that it almost surely determines the field $h$.

\begin{cor}\label{cor:forgetting-parameterization}
    Let $\gamma \in (0,2)$. Suppose $h$ is a whole-plane GFF normalized to have a mean zero on the unit circle. Let $D_h$ and $\mu_h$ be the corresponding $\gamma$-LQG metric and measure, respectively. Then the random pointed metric space $(\C, 0, D_h)$ almost surely determines the random pointed metric measure space $(\C, 0, D_h, \mu_h)$, and moreover the field $h$ up to rotation and scaling of the complex plane.
\end{cor}

By considering $(\C, 0, D_h)$ as a random pointed metric space, we forget the parameterization of $D_h$ in the complex plane. More precisely, we consider it as a random element in the space of isometry classes of complete and locally compact length spaces endowed with the local Gromov--Hausdorff topology (as defined by Gromov in \cite{gromov-metric}). Similarly, the random pointed metric measure space $(\C, 0, D_h, \mu_h)$ in Corollary~\ref{cor:forgetting-parameterization} is measurable with respect to the local Gromov--Hausdorff--Prokhorov topology \cite{adh-ghp-distance}. It was known previously that the random pointed metric measure space $(\C, 0, D_h, \mu_h)$ almost surely determines the field $h$ up to rotation and scaling. For $\gamma = \sqrt{8/3}$, this fact was first established in \cite{lqg-tbm3} as a key component of the equivalence between the Brownian sphere and the $\sqrt{8/3}$-LQG sphere. An explicit way of reconstructing $h$ from the pointed metric measure space $(\C, 0, D_h, \mu_h)$ was given in \cite{gms-poisson-voronoi}, which was extended to all $\gamma \in (0,2)$ in \cite{afs-metric-ball}.

Our choice of scaling constants $\{\mathfrak b_\ep\}_{\ep>0}$ has the following description which depends only on $\gamma \in (0,2)$. Let $(\C, h^\gamma, 0,\infty)$ be a $\gamma$-quantum cone under the circle average embedding (see Section~\ref{sec:quantum-cone} for definitions). Let $\eta$ be a whole-plane space-filling SLE$_{\kappa'}$ curve from $\infty$ to $\infty$ with $\kappa' = 16/\gamma^2$ that is sampled independently of $h^\gamma$. Then parameterize $\eta$ by the $\gamma$-LQG measure: that is, $\eta(0) = 0$ and $\mu_h(\eta[s,t]) = t-s$ for all real $s\leq t$. (We review this setup further in Section~\ref{sec:mating-of-trees}.) Let $D_{h^\gamma}$ be the $\gamma$-LQG metric associated with $h^\gamma$. For $\ep>0$, define 
\begin{equation}\label{eqn:scaling-constant}
    \mathfrak b_\ep := \Exp[N_\ep(\eta[0,1]; D_{h^\gamma})].
\end{equation}
While we do not have an exact formula for $\mathfrak b_\ep$, the following properties justify calling the limit in \eqref{eqn:main-theorem} the $d_\gamma$-Minkowski content of $A$ with respect to the metric $D_h$ (also see Section~\ref{sec:minkowski-content-def}).

\begin{prop}\label{prop:scaling-constant-properties}
    Let $\mathfrak b_\ep$ be as in \eqref{eqn:scaling-constant}. There exist constants $0 < c_1 < c_2 < \infty$ such that for all $\ep \in (0,1)$,
    \begin{equation}\label{eqn:scaling-constant-bounds}
        c_1 \ep^{-d_\gamma} \leq \mathfrak b_\ep \leq c_2 \ep^{-d_\gamma}.
    \end{equation}
    Moreover, the function $\ep \mapsto \mathfrak b_\ep$ is regularly varying with index $-d_\gamma$; i.e., for every $r>0$ we have 
    \begin{equation}\label{eqn:scaling-constant-regularly-varying}
        \lim_{\ep \to 0} \frac{\mathfrak b_{r\ep}}{\mathfrak b_\ep}= r^{-d_\gamma}.
    \end{equation}
\end{prop}

As a byproduct of our proof of Theorem~\ref{thm:main-thm}, we obtain a H\"older continuity result for the space-filling SLE curve on an LQG metric space.

\begin{thm} \label{thm:sle-holder}
    Let $\gamma \in (0,2)$ and $\kappa' \in (4,\infty)$ be constants, which do not necessarily satisfy $\kappa' = 16/\gamma^2$. Let $h^\gamma$ be the field of a $\gamma$-quantum cone under the circle average embedding, and let $\eta$ be a whole-plane space-filling SLE$_{\kappa'}$ from $\infty$ to $\infty$ which is sampled independently of $h^\gamma$ and then parameterized by $\mu_{h^\gamma}$. Almost surely, $\eta$ on the metric space $(\C, D_{h^\gamma})$ is locally H\"older continuous with any exponent less than $1/d_\gamma$ and is not locally H\"older continuous with any exponent greater than $1/d_\gamma$.
\end{thm}

Theorem~\ref{thm:sle-holder} is used in~\cite[Section 4]{bgs-meander} to show that the dimensions of the supports of certain random permutons defined in terms of SLE-decorated LQG are almost surely equal to one. Since LQG decorated by space-filling SLE is related to many other mathematical objects~\cite{ghs-mating-survey}, we expect that the theorem will have more applications in the future.

\subsection{Notations}\label{sec:notations}
We use the following notations throughout the paper.
\begin{itemize}
\item The constant $\gamma \in (0,2)$ is fixed, and we do not consider multiple values of $\gamma$ simultaneously. When we do not specify the value of $\gamma$ (e.g., in expressions such as ``LQG metric ball" or ``quantum dimension"), we refer to the corresponding $\gamma$-LQG quantities.\vspace{4pt}

\item We write $d_\gamma$ for the dimension of $\gamma$-LQG. We also use the $\gamma$-dependent constants
\begin{equation}\label{eqn:Q-and-xi-def}
    Q = \frac{\gamma}{2} + \frac{2}{\gamma} \quad \text{and} \quad \xi = \frac{\gamma}{d_\gamma}.
\end{equation}

\item Given a random distribution $h$ on $\C$, we denote the associated $\gamma$-LQG metric as $\mu_h$ and the $\gamma$-LQG measure as $D_h$. \vspace{4pt}

\item We denote by $B_r(z)$ the Euclidean ball of radius $r$ centered at $z$. $\mathcal B_r(z;D_h):=\{w \in \C: D_h(z,w) < r\}$ is the $D_h$-metric ball of radius $r$ centered at $z$. For a set $A \subset \C$, we denote by $B_r(A) := \bigcup_{z \in A} B_r(z)$ and $\mathcal B_r(A;D_h) := \bigcup_{z\in A} \mathcal B_r(z;D_h)$ its $r$-neighborhoods in the Euclidean and LQG metrics, respectively.\vspace{4pt}

\item We say that an event $E_\ep$ indexed by $\ep>0$ occurs with \textit{superpolynomially high probability} if, for each $p>0$, we have $\Prob(E_\ep) \geq 1- \ep^p$ for sufficiently small $\ep > 0$.\vspace{4pt}

\item Let $X$ and $Y$ be random variables coupled on a probability space, taking values in measurable spaces $\mathcal X$ and $\mathcal Y$, respectively. We say that $Y$ is \textit{almost surely determined by} $X$ if there is a measurable function $F:\mathcal X \to \mathcal Y$ such that $Y = F(X)$ almost surely.
\end{itemize}

\subsection{Outline}
We review the necessary preliminaries in Section~\ref{sec:preliminaries}. With these preliminaries in hand, we present a detailed overview of our proof of the main results in Section~\ref{sec:proof-strategy}. Section~\ref{sec:tightness} is dedicated to showing the tightness of the Minkowski content approximations. In Section~\ref{sec:minkowski-content-sle-segments}, we prove a key stepping stone towards the main results:\ the convergence in \eqref{eqn:main-theorem} of the normalized covering number $\mathfrak b_\varepsilon^{-1} N_\varepsilon(A;D_h)$ to the $\gamma$-LQG area $\mu_h(A)$ when $A$ is a space-filling SLE segment of fixed LQG area. We extend this convergence to general sets $A$ and fields $h$ and complete the proofs of the main results in Section~\ref{sec:generalization}.

\section{Preliminaries}\label{sec:preliminaries}

We review a few preliminaries, including the Minkowski content measure on a metric space (Section~\ref{sec:minkowski-content-def}), the axiomatic characterization of the LQG metric (Section~\ref{sec:lqg-metric-axioms}), and the continuum mating-of-trees theory (Section~\ref{sec:mating-of-trees}). 

We also prove some extensions of known LQG results. In Section~\ref{sec:quantum-surfaces}, we show that the conformal coordinate change rule for the LQG metric~\cite{gm-coord-change} extends to certain random scalings and translations. In Section~\ref{sec:mating-of-trees}, we use these results to prove that the mating-of-trees theorem is, in a certain precise sense, compatible with the LQG metric. 

In Section~\ref{sec:proof-strategy}, we give an overview of the proof of Theorem~\ref{thm:main-thm} and a comparison with Le Gall's proofs in \cite{legall-measure}.

\subsection{Minkowski content measure} \label{sec:minkowski-content-def}

Let $(X,d)$ be a metric space. Given a set $A \subset X$, let $N_\ep(A;d)$ be the minimum number of metric balls with radius $\ep>0$ required to cover $A$. The \textit{Minkowski dimension} of a set $A$ is defined as
\begin{equation}\label{eqn:minkowski-dimension-def}
    \dim_{\mathrm M}(A;d) = \lim_{\ep \to 0} \frac{\log N_\ep(A;d)}{\log \ep^{-1}}
\end{equation}
if the limit exists. There are several equivalent descriptions of the Minkowski dimension. For instance, we can replace the covering number $N_\ep(A;d)$ with the packing number $N^{\mathrm {pack}}_\ep(A;d)$, which is the maximum possible number of disjoint metric balls with radius $\ep$ whose centers all lie in $A$. These two definitions are equivalent because $N^{\mathrm {pack}}_\ep(A;d) \leq N_\ep(A;d) \leq N^{\mathrm {pack}}_{\ep/2}(A;d)$ for every $\ep>0$. 

The Minkowski dimension of the $\gamma$-LQG metric is the Minkowski dimension of any open set with respect to the $\gamma$-LQG metric. In \cite{afs-metric-ball}, this quantity was shown to be equal to $d_\gamma$. The proof was based on the estimate \eqref{eqn:ball-volume-asymptotics} on the volume of LQG metric balls, which we utilize prominently throughout our paper.

\begin{thm}[{\cite[Theorem 1.1]{afs-metric-ball}}]\label{thm:metric-ball-volume}
    Let $h$ be a whole-plane GFF normalized so that $h_1(0) = 0$. For any compact set $K \subset \C$ and $\zeta > 0$, almost surely, 
    \begin{equation}\label{eqn:ball-volume-asymptotics}
        \sup_{\ep \in (0,1)} \sup_{z\in K} \frac{\mu_h(\mathcal B_\ep(z;D_h))}{\ep^{d_\gamma - \zeta}} < \infty \quad \text{and} \quad \inf_{\ep \in (0,1)} \inf_{z\in K} \frac{\mu_h(\mathcal B_\ep(z;D_h))}{\ep^{d_\gamma + \zeta}} >0.
    \end{equation}
    Moreover, for any bounded Borel measurable set $A \subset \C$ containing an open set, almost surely,
    \begin{equation}
        \lim_{\ep \to 0} \frac{\log N_\ep(A;D_h)}{\log \ep^{-1}} = d_\gamma .
    \end{equation}
\end{thm}

The \textit{Minkowski content} is a method of assigning sizes to subsets of a metric space by using the quantities used to find their Minkowski dimensions. It has several different definitions in the literature which are not equivalent. One such definition is the following: if the limit computed in \eqref{eqn:minkowski-dimension-def} equals $\delta$, then the Minkowski content of $A$ is
\begin{equation}\label{eqn:minkowski-content-simple-def}
    \mathrm{Cont}_\delta(A;d) = \lim_{\ep \to 0} \ep^\delta N_\ep(A;d)
\end{equation}
if the limit exists. Replacing $N_\ep(A;d)$ with other quantities that give rise to equivalent definitions of the Minkowski dimension, such as $N_\ep^{\mathrm{pack}}(A;d)$, do not necessarily give identical values for the Minkowski content. As such, it makes sense to introduce the following general notion of Minkowski content. 

\begin{defn}\label{def:scaling-coefficients}
    For a constant $\delta>0$ and a family of constants $b = \{b_\ep\}_{\ep>0}$, we say that $b$ is a \textit{sequence of $\delta$-dimensional rescaling coefficients} if the following two conditions are satisfied.
    \begin{enumerate}[(i)]
        \item There exist constants $0<c_1<c_2<\infty$ and $\ep_0>0$ such that $c_1 \ep^{-\delta} \leq b_\ep \leq c_2 \ep^{-\delta}$ for every $\ep \in (0,\ep_0)$.

        \item The function $\ep \mapsto b_\ep$ is regularly varying at 0 with index $-\delta$. That is, $\lim_{\ep \to 0} b_{r\ep}/b_\ep = r^{-\delta}$ for every $r>0$.
    \end{enumerate}
\end{defn}

For instance, $\{\ep^{-\delta}\}_{\ep>0}$ is a trivial sequence of $\delta$-dimensional rescaling coefficients. Proposition~\ref{prop:scaling-constant-properties} says that $\{\mathfrak b_\ep\}_{\ep>0}$ as defined in \eqref{eqn:scaling-constant} is a sequence of $d_\gamma$-dimensional rescaling coefficients. In Definition~\ref{def:scaling-coefficients}, we require condition (ii) so that $\ep^\delta b_\ep$ does not fluctuate arbitrarily as $\ep \to 0$.

\begin{defn}\label{def:general-minkowski-content}
    Let $(X,d)$ be a metric space and $b = \{b_\ep\}_{\ep>0}$ be a sequence of $\delta$-dimensional rescaling coefficients. For $A\subset X$ with $\dim_{\mathrm M}(A;d) = \delta$, the \textit{Minkowski content of $A$ with respect to coefficients }$b$ is the limit
    \begin{equation}\label{eqn:general-minkowski-content}
        \mathrm{Cont}_{b}(A;d) = \lim_{\ep \to 0} b_{\ep}^{-1} N_{\ep}(A;d).
    \end{equation}
    if it exists. In that case, we say that $A$ is \textit{Minkowski measurable with respect to coefficients $b$}.
\end{defn}

Theorem~\ref{thm:main-thm} states that with $\mathfrak b = \{\mathfrak b_\ep\}_{\ep>0}$ as in \eqref{eqn:scaling-constant}, for every bounded Borel set $A\subset \C$ with $\mu_h(\bd A) = 0$, the Minkowski content of $A$ with respect to $\mathfrak b$ exists and is almost surely equal to $\mu_h(A)$. The condition $\mu_h(\bd A)=0$ is natural; if an open set $A\subset X$ is Minkowski measurable with respect to coefficients $b$, then $\mathrm{Cont}_b(A;d) = \mathrm{Cont}_b(\overline A;d)$. This is since $N_\ep(A;d) \leq N_\ep(\overline A;D) \leq N_{(1-\zeta)\ep}(A;d)$ for every $\ep>0$ and $\zeta \in (0,1)$, and we require $b_\ep$ to vary regularly.

In general, it is difficult to a priori find the correct coefficients $\{b_\ep\}_{\ep>0}$ such that the limit \eqref{eqn:general-minkowski-content} exists. Indeed, we will use \eqref{eqn:minkowski-content-simple-def} as an ansatz for the Minkowski content in $\gamma$-LQG, and we show only in a later stage of the proof that \eqref{eqn:scaling-constant} is the correct rescaling coefficient to use.

One reason for considering the Minkowski content for the LQG metric space is that the Minkowski content is extremely useful in the context of random fractal subsets. For instance, the Minkowski content has been used to construct natural measures on the Schramm--Loewner evolution (SLE) curve and its subsets (e.g., \cite{alberts-sheffield-bdy-measure,hlls-cut-pts,lawler-bdy-minkowski,lawler-rezai-nat,zhan-sle-minkowski-content}). However, the Minkowski content has been considered traditionally in the context of fractal subsets of Euclidean spaces. When $A$ is a fractal subset of a Euclidean space $\R^n$, the Minkowski dimension of $A$ can be defined equivalently as
\begin{equation} 
    \dim_{\mathrm M}(A) = n - \lim_{\ep \to 0} \frac{\log (\mathrm{vol} \,B_\ep(A))}{\log \ep}.
\end{equation}
This is why the $\delta$-dimensional Minkowski content of $A \subset \R^n$ for $\delta < n$ is defined usually as
\begin{equation}\label{eqn:old-minkowski-content} 
    \mathrm{Cont}_\delta (A) = \lim_{\ep \to 0} \frac{\mathrm{vol}\,B_\ep(A)}{\ep^{n-\delta}}. 
\end{equation}

We eventually wish to construct a Borel measure from the Minkowski content on an LQG metric space and compare it with the LQG measure in Corollary~\ref{cor:forgetting-parameterization}. However, Minkowski content is not countably additive in general. Thus, we define a \textit{Minkowski content measure} to be any Borel measure that is compatible with some version of the Minkowski content. The following definition aligns with various definitions of similar concepts in the literature (e.g., \cite{minkowski-measure-metric-space,zhan-sle-minkowski-content}). 

\begin{defn}\label{def:minkowski-content-measure}
    Let $(X,d)$ be a locally compact metric space whose Minkowski dimension is $\delta>0$: i.e., $\dim_{\mathrm M}(U;d) = \delta$ for every totally bounded open set $U\subset X$. A \textit{Minkowski content measure} on $X$ is a Borel measure $\mu$ on $X$ that satisfies the following conditions. 
    \begin{enumerate}[(i)]
        \item The measure $\mu$ is finite on all compact subsets of $X$.
        \item There is a sequence $b = \{b_\ep\}_{\ep>0}$ of $\delta$-dimensional rescaling coefficients such that  $\mathrm{Cont}_b(K;d)$ exists and equals $\mu(K)$ for every compact set $K\subset X$ with $\mu(\bd K)=0$.
    \end{enumerate}
\end{defn}

In other words, Theorem~\ref{thm:main-thm} states that the LQG measure $\mu_h$ is a Minkowski content measure on the LQG metric space $(\C,D_h)$. In the proof of Corollary~\ref{cor:forgetting-parameterization}, we give an explicit method to recover $\mu_h$ as a Minkowski content measure on the LQG metric space $(\C,D_h)$. In particular, we give a random $\pi$-system of compact sets coupled with the GFF $h$ which a.s.\ generates the Borel $\sigma$-algebra, so that the values of Minkowski content $\mathrm{Cont}_{\mathfrak b}(K;D_h)$ for the sets $K$ in this $\pi$-system a.s.\ uniquely determine the Minkowski content measure $\mu_h$.

\subsection{Whole-plane GFF}\label{sec:whole-plane-gff}
We give a brief introduction to the whole-plane GFF insofar as it is relevant to the rest of the paper. We refer the reader to the introductory sections of~\cite{dms-lqg-mating,ig4} and the expository articles~\cite{bp-lqg-notes,shef-gff,pw-gff-notes} for further details.

The \textit{whole-plane Gaussian free field} (GFF) $h$ is a centered Gaussian process on $\C$ with covariances 
\begin{equation}\label{eqn:gff-covariance}
    \mathrm{Cov}(h(z),h(w)) = G(z,w) : = \log \frac{(\max\{|z|,1\})(\max\{|w|,1\})}{|z-w|}, \quad \forall z,w\in \C.
\end{equation}
This definition does not make literal sense since $\lim_{z\to w} G(z,w) = \infty$, but we can make sense of the whole-plane GFF as a random distribution (i.e., generalized function). Let $\mathcal M$ be the collection of signed Borel measures $\rho$ with compact supports on $\C$ and $\int_{\C \times \C} G(z,w)\,|\rho|(dz)|\rho|(dw) < \infty$. We define \cite[Section~3.1]{pw-gff-notes} the whole-plane GFF as the centered Gaussian process indexed by $\mathcal M$ with
\begin{equation}
    \mathrm{Cov}((h,\rho_1), (h,\rho_2)) = \int_{\C \times \C} G(z,w)\,\rho_1(dz)\rho_2(dw).
\end{equation}
A random distribution $h$ on $\C$ is called a \textit{whole-plane GFF plus a continuous function} if there exists a coupling of $h$ with a random continuous function $f:\C\to\R$ such that $h-f$ has the law of a whole-plane GFF.

Given a whole-plane GFF $h$, we define the \textit{circle average} on the circle $\bd B_r(z)$ as the pairing $h_r(z):= (h,\lambda_{z,r})$ where $\lambda_{z,r}$ is the uniform probability measure on the circle $\bd B_r(z)$. The following properties of the circle average process were given in~\cite[Section 3.1]{shef-kpz}.

\begin{lem}\label{lem:circle-average-continuous}
    There exists a version of whole-plane GFF $h$ such that the map $(z,r) \mapsto h_r(z)$ is a.s.\ continuous on $\C \times (0,\infty)$. Moreover, for each $z \in \C$, the process $\{B_t\}_{t\in \R}$, $B_t := h_{e^{-t}}(z)$, is a two-sided standard Brownian motion with the initial value $B_0 = h_1(z)$.
\end{lem}

Let $h$ be a whole-plane GFF plus a continuous function. By the \textit{circle average part} of $h$, we refer to the function $f(z) = h_{|z|}(0)$. By the \textit{lateral part} of $h$, we refer to the distribution $g = h -f$. The following is a key property of the whole-plane GFF used in this paper.

\begin{lem}[{\cite[Lemma 4.9]{dms-lqg-mating}}]\label{lem:GFF-orthogonal-decomposition}
    The circle average and lateral parts of the whole-plane GFF are independent.
\end{lem}

We note that the term whole-plane GFF also refers to the random distribution considered modulo additive constant. 
Our choice of covariance kernel in~\eqref{eqn:gff-covariance} corresponds to fixing the additive constant so that the average $h_1(0)$ of $h$ over the unit circle is zero~\cite[Section 2.1.1]{vargas-dozz-notes}. The law of the whole-plane GFF is invariant under deterministic complex affine transformations of $\C$, modulo additive constant. That is, if $a\in\C\setminus \{0\}$ and $b\in\C$, then
\begin{equation} \label{eqn:gff-translate}
    h(a\cdot + b) - h_{|a|}(b) \overset{d}{=} h .
\end{equation}
In the few instances where we refer to the whole-plane GFF in this sense, we always write explicitly that we are considering it modulo additive constant.

\begin{rem}\label{rem:gff-on-uhp}
    We close our discussion of the planar GFF with a brief discussion of the free-boundary GFF on the upper half-plane $\mathbb H = \{z: \mathrm{Im}\,z>0\}$, which will be relevant in Section~\ref{sec:mating-of-trees-theorem}. This is a random distribution $\tilde h$ on $\mathbb H$ which has the law of $(h(z)+h(\bar z))/2$ where $h$ is a whole-plane GFF; it can be rigorously defined as a centered Gaussian process indexed by signed Borel measures in $\mathcal M$ which are supported on $\mathbb H$, with $\mathrm{Cov}(\tilde h(\rho_1),\tilde h(\rho_2)) = \int_{\mathbb H\times \mathbb H} [G(z,w) + G(z,\bar w)] \rho_1(dz) \rho_2(dw)$ where $G$ is the whole-plane Green's function given in \eqref{eqn:gff-covariance}. For $x\in \mathbb R$ and $r>0$, we define $\tilde h_r(x)$ to be the paring $\tilde h(\tilde \lambda_{x,r})$ where $\tilde \lambda_{x,r}$ is the uniform probability measure on the semicircle $\partial B_r(x) \cap \mathbb H$. Define $\tilde h_\parallel(z):= \tilde h(z) - \tilde h_{|z|}(0)$ be the projection of $\tilde h$ onto the space of functions that have average zero on all semi-circles centered at the origin (``lateral part"). Then, similarly to Lemma~\ref{lem:GFF-orthogonal-decomposition}, the semicircle averages $\{\tilde h_r(0)\}_{r>0}$ and the lateral part $\tilde h_\parallel$ are independent \cite[Lemma~4.2]{dms-lqg-mating}.
\end{rem}

\subsection{LQG metric axioms} \label{sec:lqg-metric-axioms}

Due to its variational formulation, it is difficult to work with the definition \eqref{eqn:lqg-metric-def} of the LQG metric. The axiomatic characterization of the LQG metric given in \cite{gm-uniqueness} is often a more tractable means of studying the LQG metric.

Before we state the LQG metric axioms, we recall the following definitions regarding metric spaces. Let $(X,D)$ be a metric space. A \textit{curve} in $X$ is a continuous function $P: [a,b] \to X$. The \textit{$D$-length} of $P$ is 
\begin{equation}
    \mathrm{len}(P;D) = \sup_T \sum_{i=1}^{|T|}D(P(t_{i-1}), P(t_{i }))
\end{equation}
where the supremum is over all partitions $T: a = t_0 < t_1 < \cdots < t_{|T|} = b$. For $Y \subset X$, the \textit{internal metric of $D$ on $Y$} is defined as 
\begin{equation}\label{eqn:internal-metric}
    D^Y(x,y) = \inf_{P \subset Y} \mathrm{len}(P;D)
\end{equation}
where the infimum is over all paths $P$ in $Y$ from $x$ to $y$. 

We say that $(X,D)$ is a length space if for each $x,y \in X$ and $\ep>0$, there exists a curve $P$ in $X$ from $x$ to $y$ with $\mathrm{len}(P;D) < D(x,y) + \ep$. We say that a metric $D$ is \textit{continuous metric} on an open domain $U \subset \C$ if it induces the Euclidean topology on $U$. In the following, we equip the space of continuous metrics on $U$ with the local uniform topology for functions $U \times U \to [0,\infty)$.

\begin{defn}\label{def:lqg-metric-axiom}
    For $\gamma \in (0,2)$, a $\gamma$-LQG metric is a Borel measurable function $h \mapsto D_h$ from the space of distributions on $\C$ to the space of the continuous metrics on $\C$ such that the following are true whenever $h$ is a whole-plane GFF plus a continuous function. Here, $Q$ and $\xi$ are constants defined in \eqref{eqn:Q-and-xi-def}.

    \begin{enumerate}[I.]
       \item \textit{Length space.} $(\C,D_h)$ is almost surely a length space. That is, the $D_h$-distance between any two points in $\C$ is the infimum over the $D_h$-lengths of curves between these two points.

       \item \textit{Locality.} Let $U\subset \C$ be a deterministic set. Then the internal metric $D_h^U$ is almost surely determined by (cf.\ Section~\ref{sec:notations}) the restriction $h|_U$.\footnote{The restriction $h|_U$ of the whole-plane GFF $h$ to $U$ can be defined precisely as the process $\{(h,\rho)\}_{\rho \in \mathcal M_U}$ where the index set $\mathcal M_U$ comprises signed Borel measures $\rho \in \mathcal M$ (recall Section~\ref{sec:whole-plane-gff}) with $\mathrm{supp}(\rho) \subset U$.}

       \item \textit{Weyl scaling.} For a continuous function $f: \C \to \R$, define 
       \begin{equation}\label{eqn:weyl-scaling}
            (e^{\xi f}\cdot D_h)(z,w) := \inf_{P: x\to y}        \int_0^{\mathrm{len}(P;D_h)} e^{\xi f(P(t))}\,dt \quad \forall z,w\in \C
       \end{equation}
       where the infimum is over all curves from $z$ to $w$ parameterized by $D_h$-length. The following holds almost surely: we have $D_{h+f} = e^{\xi f} \cdot D_h$ for every continuous function $f:\C \to \R$.

       \item \textit{Coordinate change for translation and  scaling.} For each fixed deterministic $r>0$ and $z\in \C$, almost surely
       \begin{equation}\label{eqn:coordinate-change-axiom}
           D_h(ru+z, rv+z) = D_{h(r\cdot + z) + Q\log r} (u,v) \quad \forall u,v\in \C.
       \end{equation}
    \end{enumerate}
\end{defn}

In \cite{gm-uniqueness}, it was shown that the random metric defined in \eqref{eqn:lqg-metric-def} using LFPP is a $\gamma$-LQG metric as in the sense of the above definition, and each $\gamma$-LQG metric is a deterministic constant multiple of it. Hence, it makes sense to refer to \eqref{eqn:lqg-metric-def} as \emph{the} $\gamma$-LQG metric. The paper~\cite{lqg-metric-estimates} contains an extensive list of estimates for the LQG metric deduced from the axiomatic definition, which we introduce as necessary. 

Finally, we refer the reader to \cite[Remark 1.2]{gm-confluence} for the definition of the $\gamma$-LQG metric on a proper subdomain of $\C$.

\subsection{Quantum surfaces} \label{sec:quantum-surfaces}

Recall from \eqref{eqn:Q-and-xi-def} that $Q = 2/\gamma+\gamma/2$. Consider the pair $(U,h)$ where $U\subset \C$ is an open set and $h$ is a distribution on $U$. A \textit{$\gamma$-quantum surface} (or a \textit{$\gamma$-LQG surface}) is an equivalence class of such pairs where $(U,h) \sim (\tilde U, \tilde h)$ if there exists a conformal transformation $\phi: \tilde U \to U$ such that 
\begin{equation}\label{eqn:quantum-surface-coordinate-change}
    \tilde h = h \circ \phi + Q\log |\phi'|.
\end{equation}
An \textit{embedding} of a quantum surface is a choice of representative $(U,h)$ from the equivalence class.

We often consider quantum surfaces with additional structures. As before, let $U\subset \C$ be an open set and $h$ be a distribution on $U$. Let $z_1,\dots,z_k$ be points in $\overline U$. A \textit{$\gamma$-quantum surface with $k$ marked points} is an equivalence class of the tuples $(U,h, z_1,\dots,z_k)$ where $(U,h, z_1,\dots,z_k) \sim (\tilde U, \tilde h, \tilde z_1,\dots,\tilde z_k)$ if there exists a conformal transformation $\phi: \tilde U \to U$ such that $z_j = \phi(\tilde z_j)$ for $j=1,\dots,k$ in addition to \eqref{eqn:quantum-surface-coordinate-change}. If $\eta$ is a curve in $\overline U$, then a \textit{curve-decorated $\gamma$-quantum surface} is an equivalence class of triples $(U,h,\eta)$ where $(U,h,\eta) \sim (\tilde U, \tilde h, \tilde \eta)$ if there exists a conformal transformation $\phi: \tilde U \to U$ such that $\eta = \phi \circ \tilde \eta$ in addition to \eqref{eqn:quantum-surface-coordinate-change}. We define a \textit{curve-decorated $\gamma$-quantum surface with $k$ marked points} by combining these definitions.

The equivalence relation \eqref{eqn:quantum-surface-coordinate-change} of a $\gamma$-LQG surface is chosen so that the $\gamma$-LQG measure and metric transform naturally between different embeddings of the same quantum surface. Suppose $h$ is a GFF plus a continuous function on an open set $U \subset \C$.\footnote{We say that a random distribution $h$ on an open subset $U\subset \C$ is a GFF plus a continuous function if it can be coupled with a random continuous function $f:U\to \C$ such that $h-f$ has the law of a zero-boundary GFF on $U$ (or a whole-plane GFF if $U=\C$). Note the definition of the whole-plane GFF plus a continuous function above Theorem~\ref{thm:main-thm}.} Let $\phi: \tilde U \to U$ be a fixed conformal transformation. It was established in \cite{shef-kpz} that, almost surely,
\begin{equation}\label{eqn:lqg-measure-coordinate-change}
    \mu_h(A) = \mu_{h\circ \phi + Q\log |\phi'|} (\phi^{-1} (A)) \quad \text{for every Borel } A \subset U
\end{equation}
and in \cite{gm-coord-change} that, almost surely,
\begin{equation}\label{eqn:lqg-metric-coordinate-change}
    D_h(z,w) = D_{h\circ \phi + Q\log |\phi'|}(\phi^{-1}(z),\phi^{-1}(w)) \quad \text{for every } z,w \in U.
\end{equation}

While not sufficiently emphasized in the early literature on LQG, it is necessary to include random conformal transformations in the definition of quantum surfaces to be able to compare their laws under a canonical embedding rule. For instance, there are uncountably many ways to embed a quantum surface with the disk topology into the unit disk $\ud$ if it is unmarked or has only one or two marked boundary points. To specify a canonical embedding, we need to use information about the GFF. This requires a random conformal change of coordinates since the field is random. 

Nevertheless, this random change in coordinates does not present an issue for the LQG measure. The following theorem implies that if $\phi$ is the random conformal transformation that maps a given embedding $(U, h)$ of a quantum surface to its canonical embedding $(\tilde U, \tilde h)$, then it is a.s.\ the case that the LQG measure $\mu_{\tilde h}$ is well-defined and equal to the pushforward measure $\phi_*\mu_h$.

\begin{thm}[{\cite[Theorem 1.4]{sw-all-maps}}]\label{thm:all-maps-coord-change-measure}
    Let $U\subset \C$ be a simply connected domain and $h$ be a GFF plus a continuous function on $U$. Let $\Lambda$ be the collection of all conformal maps $\phi: \tilde U \to U$ where $\tilde U\subset\C$ is any simply connected domain. It is almost surely the case that for all $\phi \in \Lambda$, the measures $\mu_{h\circ \phi + Q\log |\phi'|}$ are well-defined and the transformation rule \eqref{eqn:lqg-measure-coordinate-change} holds simultaneously for all $\phi \in \Lambda$.
\end{thm}

\subsubsection{Coordinate change for random translation and scaling}

An analog of Theorem~\ref{thm:all-maps-coord-change-measure} for the LQG metric is expected to be true but has not yet been established. (After the acceptance of this article, it was proven in \cite{devlin-lfpp-as} that, almost surely, \eqref{eqn:lqg-metric-coordinate-change} holds for all complex affine transformations simultaneously.) In the following two lemmas, we show that the transformation rule \eqref{eqn:lqg-metric-coordinate-change} holds almost surely for a certain subset of random conformal maps from $\C$ to itself. These are random translations (Lemma~\ref{lem:translation-metric-invariance}) and random scalings where the scaling factor is almost surely determined by the circle average part of the field (Lemma~\ref{lem:scaling-metric-covariance}). These correspond exactly to the random transformations that determine a canonical embedding, which appears in the continuum mating-of-trees theory, which we present in the next section.

\begin{lem}\label{lem:translation-metric-invariance}
    Suppose $h$ is a whole-plane GFF plus a continuous function and $z\in \C$ is any random point (not necessarily independent from $h$). If $h(\cdot+z)$ has the law of a whole-plane GFF plus a continuous function, then, almost surely,  
    \begin{equation}\label{eqn:translation-metric-invariance}
        D_{h(\cdot + z)}(u,v) = D_h(u+z,v+z) \quad \text{for all } u,v\in \C . 
    \end{equation}
\end{lem}
\begin{proof}
    Denote $\hat h:= h(\cdot + z)$.
    Let $\ep>0$ and suppose $P: [0,1] \to \R$ is a piecewise continuously differentiable path. Letting $\hat P(t): = P(t) + z$, almost surely,
    \begin{equation}\label{eqn:lfpp-random-translation}
        \int_0^1 e^{\xi \hat h_\ep^*(P(t))} |P'(t)|\,dt = \int_0^1 e^{\xi h_\ep^*(\hat P(t))} |\hat P'(t)|\,dt 
    \end{equation}
    for all such paths $P$. From the definition \eqref{eqn:LFPP-metric} of the $\ep$-LFPP metric, almost surely, $D_{\hat h}^\ep(u,v) = D_h^\ep(u+z,v+z)$ for all $u,v\in \C$. Since the rescaled $\ep$-LFPP metric $\mathfrak a_\ep^{-1} D_h^\ep$ converges in probability with respect to the local uniform topology on $\C \times \C$ to $D_h$ \cite[Theorem~1.1]{gm-uniqueness}, so does $\mathfrak a_\ep^{-1} D_{\hat h}^\ep$. If $\hat h$ is a whole-plane GFF plus a continuous function, then the LQG metric $D_{\hat h}$ is the limit. 
\end{proof}

We remark that we can use \eqref{eqn:translation-metric-invariance} to define $D_{h(\cdot+z)}$ even when $h(\cdot+z)$ does not have the law of a whole-plane GFF plus a continuous function.

As for random scaling, we only consider the case for which the scaling factor $r$ is measurable with respect to the circle average part of the field (recall the discussion just above Lemma~\ref{lem:GFF-orthogonal-decomposition}).

\begin{lem}\label{lem:scaling-metric-covariance}
    Let $h$ be a whole-plane GFF plus a continuous function whose circle average and lateral parts are independent (by Lemma~\ref{lem:GFF-orthogonal-decomposition}; this is the case for the whole-plane GFF). If $r>0$ is a random scaling factor which is almost surely determined by the circle average part of $h$, then almost surely
    \begin{equation}\label{eq:scaling-metric-covariance} D_h(ru,rv) = D_{h(r\cdot) + Q\log r}(u,v) \quad \forall\, u,v \in \C. \end{equation}
\end{lem}
\begin{proof}
    Let us denote the circle average part and the lateral part of $h$ as $h^{\mathrm{circ}}$ and $h^{\mathrm{lat}}$, respectively. Let $\tilde h^{\mathrm{circ}}$ be an independent and identically distributed copy of $h^{\mathrm{circ}}$, which is also independent of $h^{\mathrm{lat}}$. Let $\tilde h:= \tilde h^{\mathrm{circ}} + h^{\mathrm{lat}}$, which is a whole-plane GFF plus a continuous function independent of $h^{\mathrm{circ}}$. Let $f = h - \tilde h = h^{\mathrm{circ}} - \tilde h^{\mathrm{circ}}$. From Lemma~\ref{lem:circle-average-continuous}, $t \mapsto h_{e^{-t}}(0)$ and $t\mapsto \tilde h_{e^{-t}}(0)$ are independent random continuous functions on $\R$. Hence, $f$ is almost surely a continuous function on $\C\setminus \{0\}$.
    
    We claim that the Weyl scaling axiom holds for $f$ even though it may be discontinuous at 0. That is, almost surely, 
    \begin{equation}\label{eqn:weyl-scaling-discontinuity}
        D_h = e^{\xi f} \cdot D_{\tilde h}
    \end{equation} 
    as random continuous metrics on $\C$. To make sense of the right-hand side, we first define $e^{\xi f} \cdot D_{\tilde h}$ between points in $\C \setminus \{0\}$ as in \eqref{eqn:weyl-scaling} except that we take the infimum over all curves that stay in $\C \setminus \{0\}$. We then extend $e^{\xi f} \cdot D_{\tilde h}$ to a continuous metric on all of $\C$ if possible; this is how the LQG metric is defined for a field with logarithmic singularities in the discussion preceding \cite[Theorem~1.10]{lqg-metric-estimates}. Since $D_h$ is a.s.\ a continuous length metric on $\C$, it suffices to check that, almost surely, the $D_h$-lengths and the $e^{\xi f}\cdot D_{\tilde h}$-lengths agree for all curves in $\C \setminus \{0\}$. For $\delta>0$, let $f_\delta$ be a random function which is almost surely continuous on $\C$ and agrees with $f$ on $\C \setminus B_\delta(0)$. By the locality axiom, almost surely, the $D_h$-lengths and the $D_{\tilde h+f_\delta}$-lengths agree for all curves in $\C\setminus B_\delta(0)$. By the Weyl scaling axiom, $D_{\tilde h+f_\delta} = e^{\xi f_\delta}\cdot D_{\tilde h}$ almost surely. Letting $\delta\to 0$ proves the claim. 
    
    By the same reasoning, $D_{h(r\cdot)+Q\log r} = e^{\xi f(r\cdot)}\cdot D_{\tilde h(r\cdot) + Q\log r}$ almost surely. Since $r>0$ is independent of $\tilde h$, by the coordinate change axiom \eqref{eqn:coordinate-change-axiom} for deterministic scaling, we almost surely have 
    \begin{equation}\label{eqn:weyl-scaling-random-zooming}
        D_{\tilde h}(ru,rv) = D_{\tilde h(r\cdot)+Q\log r}(u,v) \quad \forall\, u,v\in\C.
    \end{equation}
    The lemma now follows by combining \eqref{eqn:weyl-scaling-discontinuity} and \eqref{eqn:weyl-scaling-random-zooming}.
\end{proof}

It is straightforward to check that Lemmas~\ref{lem:translation-metric-invariance} and \ref{lem:scaling-metric-covariance} are also valid when $h$ is equal to a whole-plane GFF plus a continuous function plus a finite number of logarithmic singularities of the form $-\alpha \log|\cdot-z|$ for $z\in \C$ and $\alpha < Q$. In particular, they can be applied to a $\gamma$-quantum cone. 

\subsection{Mating-of-trees theory} \label{sec:mating-of-trees}

The continuum mating-of-trees theorem is the central tool we utilize in this paper. We first review the setup for the theorem, in particular the definitions and properties of the $\gamma$-quantum cone and the whole-plane space-filling SLE$_{\kappa'}$ curve. We state the mating-of-trees theorem afterward.

\subsubsection{Quantum cone} \label{sec:quantum-cone}

In the previous section, we discussed that for some quantum surfaces, we must use information about the field to define a canonical embedding of the surface. In this paper, we consider doubly-marked quantum surfaces parameterized by the Riemann sphere: i.e., those with embeddings of the form $(\C, h,0,\infty)$. Fixing the two marked points at $0$ and $\infty$ gives an embedding of the quantum surface that is unique only up to scaling. One choice of canonical embedding for such a quantum surface is called the circle average embedding. 

\begin{defn}\label{def:circle-average-embedding}
    We say that an embedding $(\C,h,0,\infty)$ of a doubly-marked $\gamma$-LQG surface is a \textit{circle average embedding} if 
\begin{equation}\label{eqn:circle-average-embedding}
    \sup\{r>0: h_r(0) + Q\log r = 0\} = 1
\end{equation}
where $h_r(0)$ is the circle average of $h$ on $\bd B_r(0)$.
\end{defn} 

That is, given an embedding $(\mathbb C, h, z,\infty)$ of a $\gamma$-LQG surface where one marked point is at $z\in \mathbb C$ and the other marked point is at $\infty$, the circle average embedding of this LQG surface is $(\mathbb C, h(R\cdot - z) + Q\log R, 0,\infty)$ where 
\begin{equation}
    R := \sup\{r>0: h_r(z) + Q\log r = 0\}.
\end{equation}

From the perspective of circle average embedding, the quantum surface induced by a whole-plane GFF is unnatural in that it does not satisfy scale invariance. That is, if $h$ is a whole-plane GFF and $C$ is a nonzero constant, the circle average embeddings of $(\C,h,0,\infty)$ and $(\C,h+C,0,\infty)$ do not agree in law. The natural scale-invariant analog of such a surface is called the quantum cone. 

\begin{defn}\label{def:quantum-cone}
    Let $\{A_t\}_{t\in \R}$ be a real-valued stochastic process with the following distribution. 
    \begin{itemize}
        \item For $t<0$, $A_t = \widehat B_{-t} + \gamma t$ where $\{\widehat B_s\}_{s\geq 0}$ is a standard Brownian motion with $\widehat B_0 = 0$ conditioned so that $\widehat B_s + (Q-\gamma)s > 0$ for every $s>0$.
        \item For $t\geq 0$, $A_t = B_t + \gamma t$ where $\{B_s\}_{s\geq 0}$ is a standard Brownian motion with $B_0= 0$ that is sampled independently of $\{\widehat B_s\}_{s\geq 0}$.
    \end{itemize}        
    A \textit{$\gamma$-quantum cone} is a doubly-marked quantum surface whose circle average embedding $(\C,h^\gamma,0,\infty)$ has the following law.
    \begin{itemize}
        \item The circle average process $t\mapsto h^\gamma_{e^{-t}}(0)$ has the same law as the process $A$.
        \item The lateral part $h^\gamma - h^\gamma_{|\cdot|}(0)$ of $h^\gamma$ agrees in law with the lateral part of a whole-plane GFF. 
        \item The circle average and the lateral parts of $h^\gamma$ are independent.
    \end{itemize}
\end{defn}

A $\gamma$-quantum cone has the following scale invariance property.

\begin{prop}[{\cite[Proposition 4.13(i)]{dms-lqg-mating}}] \label{prop:quantum-cone-scale-invariance}
    Let $(\C,h^\gamma,0,\infty)$ be the circle average embedding of a $\gamma$-quantum cone. Let $C$ be a real constant. Then the quantum surfaces represented by $(\C,h^\gamma,0,\infty)$ and $(\C,h^\gamma + C,0,\infty)$ agree in law. That is, let 
    \begin{equation}\label{eqn:scaling-factor-circle-average-embedding}
        R_C := \sup\left\{ r: h^\gamma_r(0) + Q\log r + C = 0 \right\}
    \end{equation}
    so that $(\C, h^\gamma(R_C\cdot) + Q\log R_C + C, 0,\infty)$ is the circle average embedding of the $\gamma$-LQG surface represented by $(\C,h^\gamma+C,0,\infty)$. Then $h^\gamma \stackrel{d}{=} h^\gamma(R_C\cdot) + Q\log R_C + C$. 
\end{prop}

Since $\mu_{h^\gamma+C} = e^{\gamma C}\mu_{h^\gamma}$, this property means that the law of a $\gamma$-quantum cone is invariant under scaling its LQG measure by a positive constant. Similarly, the fact that $D_{h^\gamma+C/\gamma} = e^{\xi C}D_{h^\gamma}$ implies that the law of a quantum cone is invariant under scaling its LQG metric by a constant.
Note that the spatial scaling factor \eqref{eqn:scaling-factor-circle-average-embedding} is a random variable that depends on the circle averages of $h^\gamma$. Hence, it is only due to Lemma~\ref{lem:scaling-metric-covariance} that the scale invariance of the quantum cone extends to the LQG metric.

Another reason to consider the $\gamma$-quantum cone is that it is the $\gamma$-LQG surface one obtains by starting with a generic $\gamma$-LQG surface, choosing a marked point on it from the $\gamma$-LQG measure, and then ``zooming in" near the marked point. More precisely, we have the following lemma, which is helpful for transferring results about a whole-plane GFF to a $\gamma$-quantum cone and is used in Section~\ref{sec:additivity}.

\begin{lem}\label{lem:zooming-in}
    Let $dh$ denote the law of a whole-plane GFF with $h_1(0) = 0$. Let let $(z,h)$ be a pair sampled from the probability measure $Z^{-1}\boldsymbol{1}_\ud(z) \mu_h(dz)dh$ where $Z = \Exp[\mu_h(\ud)]$ is the normalization constant. Then, the field under the circle average embedding of the quantum surface $(\C, h+C, z,\infty)$ --- i.e., $h(R_C\cdot) + Q\log R_C + C$ where $R_C$ is defined as in \eqref{eqn:scaling-factor-circle-average-embedding} --- converges locally in total variation distance as $C\to \infty$ to $h^\gamma$, the field of the $\gamma$-quantum cone under circle average embedding.
\end{lem}
\begin{proof}
    By \cite[Lemma~A.10]{dms-lqg-mating}, we can sample $(z,\tilde h)$ from $Z^{-1}\boldsymbol{1}_\ud(z) \mu_h(dz) dh$ by first sampling $(z,h)$ from $Z^{-1}\boldsymbol{1}_\ud(z)dzdh$, then letting $\tilde h = h -\gamma \log |\cdot-z| + \gamma \log \max(|\cdot|,1)$. Conditioned on $z\in \ud$, the field under the circle average embedding of $(\C, \tilde h + C, z, \infty)$ converges locally in total variation distance to $h^\gamma$ as $C\to \infty$ as in \cite[Proposition~4.13(ii)]{dms-lqg-mating}. It follows that the desired convergence also holds without conditioning on $z$.
\end{proof}

\subsubsection{Space-filling SLE}\label{sec:space-filling-sle}

The whole-plane space-filling SLE$_{\kappa'}$ curve from $\infty$ to $\infty$, defined for $\kappa'>4$, is a random space-filling continuous curve which intersects itself but does not cross itself. Such a curve was initially constructed for the chordal version (0 to $\infty$ in $\mathbb H$) in \cite[Theorem~1.16]{ig4}, and extended to the whole-plane version from $\infty$ to $\infty$ in \cite[Footnote~4]{dms-lqg-mating}. We refer to these sources and \cite[Section 3.6]{ghs-mating-survey} for more detailed descriptions of the curve.
Here, we summarize the construction and structure of the space-filling SLE curve as given in these references.

The space-filling SLE$_{\kappa'}$ curve is defined in terms of \textit{flow lines}, which are SLE$_{16/\kappa'}$-type curves coupled with a GFF introduced in the context of imaginary geometry \cite{ig1,ig4}. In particular, the whole-plane space-filling SLE$_{\kappa'}$ curve is the Peano curve tracing between dual space-filling trees formed by flow lines.

Here is the detailed construction. Let $\kappa = 16/\kappa'$. Starting from a whole-plane GFF $h$, for each $z\in \C$ and $\theta \in \R$, we can define the flow line starting from $z$ with angle $\theta$ as a random curve a.s.\ determined by $h$ with the law of a whole-plane SLE$_\kappa(2-\kappa)$ curve from $z$ to $\infty$ \cite[Theorems~1.4]{ig4}. We consider the angles $\pm \frac{\pi}{2}$; denote the flow line starting from $z$ with angle $\pm \frac{\pi}{2}$ as $\eta_z^{\pm}$, and orient the flow line from $z$ to $\infty$. For every $z,w\in \C$, almost surely, $\eta_z^+$ and $\eta_w^+$ merge and so do $\eta_w^-$ and $\eta_w^-$. Moreover, for each $z\in \C$, $\eta_z^+$ and $\eta_z^-$ almost surely do not cross each other \cite[Theorem~1.7]{ig4}. Thus, given a dense countable subset $\{z_j\}_{j\in \N}$ of $\C$, the unions of flow lines $\{\eta_{z_j}^+\}_{j\in \N}$ and $\{\eta_{z_j}^-\}_{j\in \N}$ form dual trees rooted at $\infty$ with leaves $\{z_j\}_{j\in \N}$ \cite[Theorem~1.10]{ig4}. To define concretely the Peano curve between these dual trees, define a total order on $\{z_j\}_{j\in \N}$ by saying that $z_j$ comes before $z_k$ if $z_k$ lies in the connected component of $\C \setminus (\eta_{z_j}^+ \cup \eta_{z_j}^-)$ whose boundary consists of the left side of $\eta_{z_j}^-$ and the right side of $\eta_{z_j}^+$ (when orienting these flow lines from $z_j$ to $\infty$). There almost surely exists a unique space-filling curve $\eta$ which traces the points $\{z_j\}_{j\in \N}$ in this order, which is visualized in Figure~\ref{fig:space-filling-sle}. The curve $\eta$ is continuous when parameterized by the Lebesgue measure on $\C$. Moreover, $\eta$ almost surely does not depend on the choice of the dense countable set $\{z_j\}_{j\in \N}$; it is a measurable function of the GFF which generates the flow lines \cite[Theorem~1.16]{ig4}. This $\eta$ is the whole-plane space-filling SLE$_{\kappa'}$ curve from $\infty$ to $\infty$. Here are a few basic properties of this curve following from the definition, which were collected in \cite[Section~3.6.4]{ghs-mating-survey}.

\begin{figure}
    \centering
    \includegraphics[scale=.9]{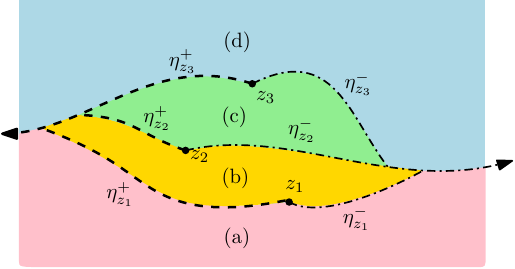}
    \caption{Definition of whole-plane space-filling SLE. Almost surely, the flow lines $\{\eta_{z_j}^+\}_{j\in \N}$ and $\{\eta_{z_j}^-\}_{j\in \N}$ each merge to form two trees. The flow lines in this illustration determine the ordering $z_1<z_2<z_3$, the order in which the space-filling SLE curve $\eta$ visits these three points. In particular, $\eta$ fills the four colored regions in the order (a)$\to$(b)$\to$(c)$\to$(d).}
    \label{fig:space-filling-sle}
\end{figure}

\begin{itemize}
    \item For each fixed $z\in \C$, almost surely, $\eta$ visits $z$ only once. If $z = \eta(t)$, then $\bd \eta(-\infty,t] = \bd \eta[t,\infty) = \eta_z^+ \cup \eta_z^-$.

    \item Let $\eta$ be parameterized by Lebesgue measure with $\eta(0) = 0$. Then $\eta^R: t\mapsto \eta(-t)$ has the same law as $\eta$. This property is referred to as \textit{reversibility}.

    \item Let $\phi$ be a deterministic conformal transformation of the Riemann sphere $\C \cup \{\infty\}$ which fixes $\infty$. That is, $\phi$ is a composition of scaling, rotation, and translation. Then $\phi \circ \eta$ has the same law as $\eta$ up to reparameterization.

    \item For $\kappa'\geq 8$, $\eta(-\infty,0]$ and $\eta[0,\infty)$ are both homeomorphic to the closed half-plane $\overline{\mathbb H}$. Note that $\eta_0^- \cup \eta_0^+ = \eta(-\infty,0] \cap \eta[0,\infty)$. Conditioned on $\eta_0^- \cup \eta_0^+$, the conditional law of $\eta|_{[0,\infty)}$ is that of a chordal SLE$_{\kappa}$ curve from 0 to $\infty$ in $\eta[0,\infty)$, and the conditional law of the time reversal of $\eta|_{(-\infty,0]}$ is that of a chordal SLE$_{\kappa}$ curve from 0 to $\infty$ in $\eta(-\infty,0]$. Moreover, the two curves are conditionally independent given $\eta_0^- \cup \eta_0^+$. See also~\cite[Footnote 4]{dms-lqg-mating}. 
    
    \item For $\kappa' \in (4,8)$, the interiors of $\eta(-\infty,0]$ and $\eta[0,\infty)$ are both  infinite chains of Jordan domains. Again $\eta_0^- \cup \eta_0^+ = \eta(-\infty,0] \cap \eta[0,\infty)$. Conditioned on $\eta_0^- \cup \eta_0^+$, the curve $\eta|_{[0,\infty)}$ and the time reversal of $\eta|_{(-\infty,0]}$ are conditionally independent concatenations of chordal SLE$_{\kappa}$ curves in the connected components of $\eta[0,\infty)$ and $\eta(-\infty,0]$, respectively. The two curves are conditionally independent given $\eta_0^- \cup \eta_0^+$. See also~\cite[Footnote 4]{dms-lqg-mating}.
\end{itemize}

The following proposition states that every segment of a whole-plane space-filling SLE$_{\kappa'}$ curve contains a Euclidean ball of comparable size with high probability. In \cite{ghm-kpz}, this estimate was used to show that $\eta$ parameterized according to the Lebesgue measure is locally H\"older continuous with any exponent less than $1/2$ and is not locally H\"older continuous with any exponent greater than $1/2$ with respect to the Euclidean metric (cf.\ Theorem~\ref{thm:sle-holder}). This is a key estimate in the proof of Theorem~\ref{thm:main-thm}. 

\begin{prop}[{\cite[Proposition 3.4 and Remark 3.9]{ghm-kpz}}]\label{prop:SLE-includes-Euclidean-ball}
    Fix $\kappa' > 4$ and let $\eta$ be a whole-plane space-filling SLE$_{\kappa'}$ curve from $\infty$ to $\infty$. For each $r \in (0,1)$ and $R>0$, the following happens with superpolynomially high probability as $\ep \to 0$: for each $\delta \in (0,\ep]$ and every $a<b$ such that $\eta[a,b] \subset B_R(0)$ and $\mathrm{diam} \,\eta[a,b] \geq \delta^{1-r}$, the set $\eta[a,b]$ contains a Euclidean ball of radius at least $\delta$.
\end{prop}

\subsubsection{Translation and scale invariance of quantum cone decorated with space-filling SLE}

Let $\gamma \in (0,2)$ and $\kappa'\in (4,\infty)$. Let $(\C,h^\gamma,0,\infty,\eta)$ be the circle average embedding of a $\gamma$-quantum cone decorated by an independent whole-plane space-filling SLE$_{\kappa'}$ curve $\eta$ from $\infty$ to $\infty$. Reparameterize $\eta$ by the $\gamma$-LQG measure $\mu_{h^\gamma}$. That is, $\eta(0) = 0$ and $\mu_{h^\gamma}([s,t]) = t-s$ for all $s<t$.\footnote{The a.s.\ continuity of the reparameterized curve follows from Proposition~\ref{prop:SLE-includes-Euclidean-ball} combined with the fact that every bounded open subset of $\C$ a.s.\ has positive $\mu_{h^\gamma}$-mass.} This is the default parameterization of $\eta$ that we consider in the rest of this paper. 

We already saw in Proposition~\ref{prop:quantum-cone-scale-invariance} that the circle average embedding of a $\gamma$-quantum cone is invariant under adding a deterministic constant to the field. This operation also preserves the law of the independent space-filling SLE$_{\kappa'}$ curve $\eta$ decorating the quantum cone.

\begin{lem}\label{lem:scale-invariance-quantum-cone-sle}
    For each fixed constant $C\in \R$, the circle average embedding of $(\C, h^\gamma + C/\gamma, 0,\infty,\eta(e^C\cdot))$ agrees in law with $(\C,h^\gamma,0,\infty,\eta)$.
\end{lem}
\begin{proof}
    Let $R_C$ be as in \eqref{eqn:scaling-factor-circle-average-embedding}, so that for $\tilde h := h^\gamma(R_C\cdot) + Q\log R_C + C/\gamma$, the circle average embedding of $(\C, h^\gamma + C/\gamma, 0,\infty, \eta)$ is $(\C, \tilde h, 0,\infty, R_C^{-1} \eta)$. Recall from proposition~\ref{prop:quantum-cone-scale-invariance} that $\tilde h \stackrel{d}{=} h^\gamma$. Since $\eta$ modulo reparameterization is independent from $h^\gamma$, it is also independent from $R_C$. By the scale invariance of whole-plane space-filling SLE$_{\kappa'}$ (see previous section), $R_C^{-1} \eta$ modulo reparameterization agrees in law with $\eta$ modulo reparameterization and is independent of $\tilde h$. Therefore, the joint law of $(\tilde h, R_C^{-1} \eta)$ agrees with that of $(h^\gamma, \eta)$ with the curves viewed modulo reparameterization. Since $\mu_{\tilde h} = e^C \mu_{h^\gamma}(R_C\cdot)$, the $\mu_{\tilde h}$-parameterization of $R_C^{-1} \eta$ is given by $R_C^{-1}\eta(e^C\cdot)$.
\end{proof}

Another important property of $(\C,h^\gamma,0,\infty,\eta)$ is that for each fixed $t\in \R$, the law of the circle average embedding is invariant under re-centering this quantum surface at $\eta(t)$ (i.e., translating by $-\eta(t)$).

\begin{lem}[{\cite[Lemma 8.3]{dms-lqg-mating}}]\label{lem:stationary-configuration}
    The law of $(\C,h^\gamma,0,\infty,\eta)$ as a path-decorated quantum surface is invariant under shifting a fixed amount of $\gamma$-LQG area. That is, for each $t \in \R$, the circle average embedding of $(\C, h^\gamma, \eta(t),\infty, \eta(\cdot+t))$ agrees in law with $(\C, h^\gamma, 0,\infty, \eta)$.
\end{lem}

We emphasize that Lemmas~\ref{lem:scale-invariance-quantum-cone-sle} and \ref{lem:stationary-configuration} hold for any $\gamma\in (0,2)$ and $\kappa'>4$, including when $\kappa' \neq 16/\gamma^2$. The key fact behind both the proof of Lemma~\ref{lem:scale-invariance-quantum-cone-sle} presented above and the proof of Lemma~\ref{lem:stationary-configuration} in \cite{dms-lqg-mating} is Lemma~\ref{lem:zooming-in}, which implies that we get a $\gamma$-quantum cone if we zoom in at a point sampled according to the $\gamma$-LQG measure on a space-filling SLE$_\kappa'$ curve $\eta$. The only property of $\eta$ used here is that it is almost surely a continuous space-filling curve. With the additional property that the law of $\eta$ is scale-invariant, we have that the law of $\eta$ is preserved when zooming in according to an independent field. Both properties are true regardless of the value of $\kappa'>4$. (The condition $\kappa'=16/\gamma^2$ is necessary to identify the law of $\eta(-\infty,t]$ and $\eta[t,\infty)$ as independent quantum surfaces for each $t\in \R$; see Proposition~\ref{prop:left-right-independence}.)

Notice that both lemmas state the invariance in law of the circle average embedding of a path-decorated quantum surface (either after adding a constant to the field or re-centering the quantum cone). In our applications of these results, we consider the laws of random variables defined almost surely in terms of the pointed curve-decorated metric measure space structure of these quantum surfaces: e.g., $N_\ep(\eta[s,t];D_{h^\gamma})$ where $s<t$ are fixed numbers. The scaling and stationarity properties of such random variables do not follow trivially from Lemmas~\ref{lem:scale-invariance-quantum-cone-sle} and \ref{lem:stationary-configuration} because mapping one embedding of a quantum surface to its circle average embedding involves a random scaling factor that is determined by the circle average part of the field. Nevertheless, we can use Lemmas~\ref{lem:translation-metric-invariance} and \ref{lem:scaling-metric-covariance} to translate Lemmas~\ref{lem:scale-invariance-quantum-cone-sle} and \ref{lem:stationary-configuration} in terms of the laws of pointed curve-decorated metric measure spaces.

\begin{rem}\label{rem:uihpq-convergence}
    Consider the equivalence relation on pointed curve-decorated metric measure spaces where $(X,x,d,\mu,\eta) \sim (X', x', d', \mu',\eta')$ if there exists an isometry $f: X \to X'$ such that $f(x) = x'$, $f_*\mu = \mu'$, and $f\circ \eta = \eta'$. We identify a pointed curve-decorated metric measure space with the equivalence class that it belongs to. The pointed Gromov--Hausdorff--Prokhorov--uniform (GHPU) metric introduced in \cite{gwynne-miller-uihpq} is a natural choice of metric on the space of above equivalence classes of noncompact pointed curve-decorated metric measure spaces. The precise definition of this metric is not essential for our purposes; instead, we will only need the corresponding Borel $\sigma$-algebra.

    Let $h$ and $\tilde h$ be two instances of whole-plane GFF plus a continuous function, and let $\eta$ and $\tilde \eta$ be random continuous space-filling curves. Suppose $(h,\eta) \stackrel{d}{=} (\tilde h, \tilde \eta)$ with respect to the product of the following two topologies: the weak-* topology for distributions on $\C$ with respect to smooth and compactly supported test functions\footnote{This is equivalent to the topology corresponding to considering the whole-plane GFF as a stochastic process indexed by signed Borel measures $\mathcal M$ (recall Section~\ref{sec:whole-plane-gff}) by It\^o's isometry for the GFF \cite[Section~1.7]{bp-lqg-notes}.}, and the local uniform topology on functions $\C\to \C$. Note that $\mu_h$ is almost surely determined by $h$ \cite{shef-kpz,rhodes-vargas-review} and $D_h$ is a.s.\ determined by $h$ \cite{gm-uniqueness}; analogous statements hold for $\tilde h$, $\mu_{\tilde h}$, and $D_{\tilde h}$. Hence, if $\eta$ (resp.\ $\tilde \eta$) is parameterized by $\mu_h$ (resp.\ $\mu_{\tilde h}$) and $\eta(0) = \tilde \eta(0)$ a.s., then $(D_h,\mu_h,\eta) \stackrel{d}{=}(D_{\tilde h},\mu_{\tilde h},\tilde \eta)$ with respect to the product of the local uniform topology on functions $\C \times \C\to \R$, the weak-* topology on signed Borel measures on $\C$, and the local uniform topology on functions $\R \to \C$. It is straightforward to check from the definition of the local GHPU metric that the following map is continuous: the map from the tuple $(d,\mu,\eta)$ of a continuous metric, a signed Borel measure, and a continuous curve whose space is assigned the above product topology, to the pointed curve-decorated metric measure space $(\C,0,d,\mu,\eta)$ whose space is assigned the local GHPU topology. Therefore, $(\C,0,D_h,\mu_h,\eta) \stackrel{d}{=} (\C,0,D_{\tilde h}, \mu_{\tilde h}, \tilde \eta)$ w.r.t.\ the local GHPU topology. This conclusion continues to hold when $h$ has finitely many singularities of the form $-\alpha \log |\cdot-z|$ with $\alpha < Q$ since $D_h$ is almost surely a continuous metric determined by $h$ \cite[Theorem 1.10]{lqg-metric-estimates}. (For $\alpha>Q$, almost surely, $D_h(z,w)=\infty$ for every $w\in \C\setminus\{z\}$.)
\end{rem}

\begin{prop}\label{prop:qc-mmspace-invariance}
    Let $(\C, h^\gamma,0,\infty,\eta)$ be a $\gamma$-quantum cone in the circle average embedding that is decorated with an independent whole-plane space-filling SLE$_{\kappa'}$ curve, which is then parameterized so that $\eta(0) = 0$ and $\mu_{h^\gamma}(\eta[a,b]) = b-a$ for every $a<b$. The following statements hold w.r.t\ the local GHPU topology on pointed curve-decorated metric measure spaces.
    \begin{enumerate}[(i)]
        \item For each fixed $s>0$, 
        \begin{equation}\label{eqn:qc-mmspace-scaling-invariance}
            (\C, 0, s^{1/d_\gamma}D_{h^\gamma}, s\mu_{h^\gamma}, \eta(s\cdot)) \stackrel{d}{=} (\C, 0, D_{h^\gamma}, \mu_{h^\gamma}, \eta).
        \end{equation}        
        \item For each fixed $t\in \R$,
        \begin{equation}\label{eqn:qc-mmspace-translation-invariance}
        (\C, \eta(t),D_{h^\gamma}, \mu_{h^\gamma}, \eta(\cdot+t)) \stackrel{d}{=} (\C, 0,D_{h^\gamma},\mu_{h^\gamma},\eta).
    \end{equation} 
    \end{enumerate}
\end{prop}
\begin{proof}
    Given a conformal map $\phi: C\to \C$ and an embedding $(\C,h,x,\infty,\eta)$ of a path-decorated quantum surface, denote 
    \begin{equation}
        \phi_* h:= h\circ \phi^{-1} + Q\log|(\phi^{-1})'|
    \end{equation}
    and
    \begin{equation}
        \phi_*(\C,h,x,\infty,\eta):= (\C, \phi_*h, \phi(x),\infty,\phi\circ \eta).
    \end{equation}
    That is, $\phi_*(\C, 0,\infty,h,\eta)$ is the pushforward of the embedding $(\C, 0,\infty,h,\eta)$ under $\phi$ using \eqref{eqn:quantum-surface-coordinate-change}. For a continuous metric $d$ on $\C$ and a Borel measure $\mu$ on $\C$, denote their pushforwards under $\phi$ as $\phi_* d$ and $\phi_* \mu$, respectively.
    \begin{enumerate}[(i)]
        \item Let $s >0$ be fixed and denote $\tilde h := h^\gamma + (\log s)/\gamma$. Let $r := R_{\log s}$ be defined as in \eqref{eqn:scaling-factor-circle-average-embedding} and define $\phi:\C\to \C$ by $\phi(z) = r^{-1}z$ so that $\phi_*(\C, \tilde h, 0,\infty,\eta(s\cdot))$ is the circle average embedding. Setting $C = \log s$ in Lemma~\ref{lem:scale-invariance-quantum-cone-sle} gives $(\C, \phi_*\tilde h, 0,\infty,\phi\circ \eta(s\cdot))\stackrel{d}{=}(\C,h^\gamma,0,\infty,\eta)$. As discussed in Remark~\ref{rem:uihpq-convergence}, this implies 
        \begin{equation}\label{eqn:qc-mmspace-scaling-proof-law}
            (\C,0,D_{\phi_*\tilde h}, \mu_{\phi*\tilde h},\phi\circ \eta(s\cdot)) \stackrel{d}{=} (\C, 0, D_{h^\gamma}, \mu_{h^\gamma},\eta)
        \end{equation}
        w.r.t.\ the local GHPU topology.
        
        Since $r$ is a.s.\ determined by the circle average part of $\tilde h$, Lemma~\ref{lem:scaling-metric-covariance} implies $D_{\phi_*\tilde h} = \phi_*D_{\tilde h}$ almost surely. By Theorem~\ref{thm:all-maps-coord-change-measure}, $\mu_{\phi_*\tilde h}=\phi_*(\mu_{\tilde h})$ almost surely. Hence,
        \begin{equation}\label{eqn:qc-mmspace-scaling-proof-as}
        \begin{aligned}
            (\C,0,D_{\phi_*\tilde h}, \mu_{\phi_*\tilde h}, \phi\circ \eta(s\cdot)) &= (\C,0,\phi_*D_{\tilde h}, \phi_*\mu_{\tilde h}, \phi\circ\eta(s\cdot)) \\ &= (\C,0,D_{\tilde h}, \mu_{\tilde h},\eta(s\cdot)) \\ &= (\C,0,s^{1/d_\gamma} D_{h^\gamma}, s\mu_{h^\gamma}, \eta(s\cdot)) 
        \end{aligned}
        \end{equation}
        almost surely as pointed curve-decorated metric measure spaces. 
        We obtain \eqref{eqn:qc-mmspace-scaling-invariance} by combining \eqref{eqn:qc-mmspace-scaling-proof-law} and \eqref{eqn:qc-mmspace-scaling-proof-as}.  

        \item The proof is similar to part (i). Let $t\in \R$ be fixed. Let $r>0$ be the random constant such that the circle average embedding of $(\C,h^\gamma,\eta(t),\infty,\eta(\cdot+t))$ is its pushforward under $\phi(z) := r(z-\eta(t))$. Lemma~\ref{lem:stationary-configuration} gives $(\C,\phi_*h^\gamma, 0,\infty,\phi\circ \eta(\cdot+t))\stackrel{d}{=}(\C,h^\gamma,0,\infty,\eta)$, which implies 
        \begin{equation}\label{eqn:qc-mmspace-translation-proof-law}
            (\C,0,D_{\phi_*h^\gamma},\mu_{\phi_*h^\gamma}, \phi\circ \eta(\cdot+t)) \stackrel{d}{=} (\C,0,D_{h^\gamma},\mu_{h^\gamma}, \eta)
        \end{equation}
        w.r.t.\ the local GHPU topology.

        Since $r$ is a.s.\ determined by the circle average part of $h^\gamma(\cdot +\eta(t))$, Lemmas~\ref{lem:translation-metric-invariance} and \ref{lem:scaling-metric-covariance} imply $D_{\phi_* h^\gamma} = \phi_* D_{h^\gamma}$ almost surely. (It follows from the proof of Lemma~\ref{lem:stationary-configuration} in \cite{dms-lqg-mating} that $h^\gamma(\cdot+\eta(t))$ is a whole-plane GFF plus a continuous function plus $-\gamma\log|\cdot|$.) By Theorem~\ref{thm:all-maps-coord-change-measure}, $\mu_{\phi_* h^\gamma} = \phi_* \mu_{h^\gamma}$ almost surely. Hence,
        \begin{equation}\label{eqn:qc-mmspace-translation-proof-as}
        \begin{aligned}
            (\C,0,D_{\phi_* h^\gamma}, \mu_{\phi_* h^\gamma}, \phi\circ \eta(\cdot+t)) &= (\C,0,\phi_*D_{h^\gamma}, \phi_*\mu_{h^\gamma}, \phi\circ\eta(\cdot+t)) \\ &= (\C,\eta(t),D_{h^\gamma}, \mu_{h^\gamma},\eta(\cdot+t)) 
        \end{aligned}
        \end{equation}
        almost surely as pointed curve-decorated metric measure spaces. 
        We obtain \eqref{eqn:qc-mmspace-translation-invariance} by combining \eqref{eqn:qc-mmspace-translation-proof-law} and \eqref{eqn:qc-mmspace-translation-proof-as}.
    \end{enumerate}\vspace{-20pt}
\end{proof}

From Proposition~\ref{prop:qc-mmspace-invariance}, we immediately obtain the following stationarity and scaling result for the number of LQG metric balls needed to cover a space-filling SLE$_{\kappa'}$ segment.

\begin{cor}\label{cor:covering-number-relation}
    Let $(\C,h^\gamma,0,\infty,\eta)$ be as in Proposition~\ref{prop:qc-mmspace-invariance}. For any $s>0$, $t\in\R$, and $\ep>0$,
    \begin{equation}
        N_\ep(\eta[t,t+s];D_{h^\gamma}) \stackrel{d}{=} N_\ep(\eta[0,s];D_{h^\gamma}) \stackrel{d}{=} N_{\ep s^{-1/d_\gamma}}(\eta[0,1];D_{h^\gamma}).
    \end{equation}
\end{cor}

\subsubsection{Mating-of-trees theorem}\label{sec:mating-of-trees-theorem}

The above results are valid for any choices of $\gamma\in (0,2)$ and $\kappa'>4$. In contrast, the following independence property requires an exact relation between $\gamma$ and $\kappa'$. As in Proposition~\ref{prop:qc-mmspace-invariance}, let $(\C, h^\gamma,0,\infty,\eta)$ be the circle average embedding of a $\gamma$-quantum cone decorated with an independent whole-plane space-filling SLE$_{\kappa'}$ curve, which is then parameterized so that $\eta(0) = 0$ and $\mu_{h^\gamma}(\eta[a,b]) = b-a$ for every $a<b$.

\begin{prop}[{\cite[Theorem 1.9]{dms-lqg-mating}}]\label{prop:left-right-independence}
     Suppose $\kappa' = 16/\gamma^2$. Denote the interiors of $\eta(-\infty,0]$ and $\eta[0,\infty)$ as $U_-$ and $U_+$, respectively. Then the $\gamma$-LQG surfaces represented by $(U_-, h^\gamma|_{U_-},0,\infty)$ and $(U_+, h^\gamma|_{U_+},0,\infty)$ are independent and identically distributed quantum surfaces called $\frac{3\gamma}{2}$-quantum wedges (also known as quantum wedges with weight $2-\frac{\gamma^2}{2}$).
\end{prop}

As in Lemmas~\ref{lem:scale-invariance-quantum-cone-sle} and \ref{lem:stationary-configuration}, what Proposition~\ref{prop:left-right-independence} means is that the fields under the canonical embeddings of $(U_-, h^\gamma|_{U_-},0,\infty)$ and $(U_+, h^\gamma|_{U_+},0,\infty)$ as described in \cite[Sections~4.2~and~4.4]{dms-lqg-mating} are independent and identically distributed. Likewise, we can rephrase this statement in terms of curve-decorated metric measure spaces, stated precisely in Proposition~\ref{prop:left-right-independent-mmspace}. The key idea is that canonical embedding of a $\frac{3\gamma}{2}$-quantum wedge is defined in terms of the field average, similarly to the circle average embedding of a quantum cone. The subtlety lies in choosing the correct topology for the $\frac{3\gamma}{2}$-quantum wedge considered as a metric measure space. Once this is done, we can extend Lemma~\ref{lem:scaling-metric-covariance} to show that the LQG metric is preserved when we reparameterize the quantum wedge to its canonical embedding.

Below, we introduce the precise definition of the $\frac{3\gamma}{2}$-quantum wedge and its canonical embedding as well as the topology on the metric measure space necessary to establish Proposition~\ref{prop:left-right-independent-mmspace}, but we only need the proposition itself for the proofs of our main results. Upon first reading, we suggest that the reader skip to the statement of Proposition~\ref{prop:left-right-independent-mmspace}.

Recall from Section~\ref{sec:space-filling-sle} the two regimes for the topology of $U_\pm$ depending on the value of $\gamma$. For $\gamma \in (0,\sqrt 2]$, the two domains $U_\pm$ are each almost surely homeomorphic to the upper half plane $\mathbb H$. The canonical embedding of the $\frac{3\gamma}{2}$-quantum wedge in this case is $(\mathbb H, h, 0,\infty)$ where $\sup\{r>0: h_r(0) + Q\log r = 0\} = 1$, similar to Definition~\ref{def:circle-average-embedding} of the circle average embedding of a quantum cone. Under this embedding, the Gaussian field $h$ for the $\frac{3\gamma}{2}$-quantum cone has the following law defined in terms of the semicircle averages and the lateral part (recall Remark~\ref{rem:gff-on-uhp}).
\begin{itemize}
    \item The semicircle average process $t\mapsto h_{e^{-t}}(0)$ has the same law as the following process $A$.
    \begin{itemize}
    \item For $t<0$, $A_t = \widehat B_{-2t} + \frac{3\gamma}{2} t$ where $\{\widehat B_s\}_{s\geq 0}$ is a standard Brownian motion with $\widehat B_0 = 0$ conditioned so that $\widehat B_{2s} + (Q-\frac{3\gamma}{2})s > 0$ for every $s>0$.
    \item For $t\geq 0$, $A_t = B_{2t} + \frac{3\gamma}{2} t$ where $\{B_s\}_{s\geq 0}$ is a standard Brownian motion with $B_0= 0$ that is sampled independently of $\{\widehat B_s\}_{s\geq 0}$.
\end{itemize}        
    \item The lateral part $h - h_{|\cdot|}(0)$ of $h$ agrees in law with the lateral part of a free-boundary GFF on $\mathbb H$. 
    \item The semicircle average and the lateral parts of $h$ are independent.
\end{itemize}

When $\gamma \in (\sqrt 2, 2)$, a $\frac{3\gamma}{2}$-quantum wedge is a concatenation of countably many connected components, each of which is a quantum disk (i.e., a simply connected quantum surface) with two marked points on its boundary. They are attached to other components via their marked points. We call each component of this quantum wedge a \textit{bead}, and a quantum surface with the same topology as the quantum wedge a beaded quantum surface. In the setup of Proposition~\ref{prop:left-right-independence}, the space-filling SLE$_{\kappa'}$ curve $\eta$ fills up one component of the domain $U_\pm$ at a time, inducing a chronological order on them. We define the canonical embedding of the $\frac{3\gamma}{2}$-quantum cone in this regime by specifying the embedding of each component. We embed each bead to $(\mathbb H, h, 0, \infty)$ so that $h_r(0) + Q\log r$ achieves its maximum at $r=1$.\footnote{We cannot use the circle average embedding because $h_r(0) + Q\log r \to -\infty$ almost surely as $r\to 0$ and as $r\to \infty$. The canonical embedding in \cite[Section~4.4]{dms-lqg-mating} is given on the strip $\R + [0,i\pi]$; to avoid introducing additional notations, we give an equivalent description under the LQG coordinate change rule corresponding to the conformal map $z\mapsto e^z$.} Under this embedding, the $\frac{3\gamma}{2}$-quantum wedge has the following law.
\begin{enumerate}
    \item Sample a Poisson point process $\Lambda$ with intensity measure $du \otimes dh$, where $du$ is the Lebesgue measure on $(0,\infty)$ and $dh$ is an infinite measure on distributions on $\mathbb H$ with the following description.
    \begin{itemize}
        \item Let $d\tilde h$ be the probability measure on the space of distributions on $\mathbb H$ corresponding to the Gaussian field $\tilde h$ sampled in the following way.
        \begin{itemize}
            \item The semicircle average process $t\mapsto \tilde h_{e^{-t}}(0)$ has the same law as the following process $A$.
            \begin{itemize}
            \item For $t<0$, $A_t = \widehat B_{-2t} + \frac{3\gamma}{2} t$ where $\{\widehat B_s\}_{s\geq 0}$ is a standard Brownian motion with $\widehat B_0 = 0$ conditioned so that $\widehat B_{2s} + (Q-\frac{3\gamma}{2})s > 0$ for every $s>0$.
            \item For $t\geq 0$, $A_t = B_{2t} + \frac{3\gamma}{2} t$ where $\{B_s\}_{s\geq 0}$ is a standard Brownian motion with $B_0= 0$ that is sampled independently of $\{\widehat B_s\}_{s\geq 0}$.
        \end{itemize}        
            \item The lateral part $\tilde h - \tilde h_{|\cdot|}(0)$ of $\tilde h$ agrees in law with the lateral part of a free-boundary GFF on $\mathbb H$. 
            \item The semicircle average and the lateral parts of $\tilde h$ are independent.
        \end{itemize}
        \item Sample $(m,h)$ from the infinite law $m^{1-4/\gamma^2}dm \otimes d\tilde h$, where $dm$ is the Lebesgue measure on $(0,\infty)$. Then, $dh$ agrees with the infinite law of the distribution $h = \tilde h + (2/\gamma)\log m$. Here, the sum $\tilde h + (2/\gamma)\log m$ refers to the distribution on $\mathbb H$ obtained by adding the constant function $(2/\gamma)\log m$ to the distribution $\tilde h$. The number $(2/\gamma)\log m$ corresponds to the value of the semicircle average $h_1(0)$ under the canonical embedding described above.
    \end{itemize}
    \item To each point $(u,h)$ in the p.p.p. $\Lambda$, correspond to it the quantum surface $(\mathbb H,h,0,\infty)$. We concatenate these components according to the first coordinate $u$ in increasing order to sample the $\frac{3\gamma}{2}$-quantum wedge. They are concatenated at the marked points 0 and $\infty$ so that removing the point corresponding to $0$ (resp.\ $\infty$) of the component $(u,h)\in \Lambda$ disconnects it from all components $(u',h') \in \Lambda$ with $u'<u$ (resp.\ $u<u'$).
\end{enumerate}

We deduce from these definitions that reparameterizing the $\frac{3\gamma}{2}$-quantum wedge appearing in Proposition~\ref{prop:left-right-independence} by its canonical embedding preserves the LQG metric. First, we map the quantum wedge (or its bead) conformally to $\mathbb H$, sending the two marked points to 0 and $\infty$. We can choose this conformal map only using the flow lines that cut out the quantum wedge, which then preserves the LQG metric a.s.\ because these flow lines are independent of the quantum wedge (recall \eqref{eqn:lqg-metric-coordinate-change}). Now, we merely need to scale each component by a random factor that is almost surely determined by the semicircle averages of the field. Since the semicircle average part of the quantum wedge is a continuous function that is independent of the lateral part, we can apply the Weyl scaling axiom as in the proof of Lemma~\ref{lem:scaling-metric-covariance} to conclude that the LQG metric is preserved under this scaling. 

Therefore, we can restate Proposition~\ref{prop:left-right-independence} in the language of random curve-decorated measure metric spaces as below. We omit a detailed proof, which is analogous to that of Proposition~\ref{prop:qc-mmspace-invariance} except for the additional consideration that conditioned on $\eta(-\infty,0]\cap \eta[0,\infty) = \eta_0^+ \cup \eta_0^-$, the curve $\eta|_{[0,\infty)}$ and the time reversal of $\eta|_{(-\infty,0]}$ are conditionally independent and identically distributed (as discussed in Section~\ref{sec:space-filling-sle}). Recall that $D_h^U$ denotes the internal metric on $U$, which was defined in \eqref{eqn:internal-metric}. 

\begin{prop}\label{prop:left-right-independent-mmspace}
	Let $\kappa' = 16/\gamma^2$ and $U_\pm$ be as in Proposition~\ref{prop:left-right-independence}. Then the random curve-decorated metric measure spaces represented by $(U_-,D_{h^\gamma}^{U_-},\mu_{h^\gamma}|_{U_-}, \eta|_{(-\infty,0]})$ and $(U_+,D_{h^\gamma}^{U_+},\mu_{h^\gamma}|_{U_+},\eta|_{[0,\infty)})$ are independent and identically distributed.
\end{prop}

\begin{rem}
When $\gamma \in (0,\sqrt 2]$, the $\frac{3\gamma}{2}$-quantum wedge decorated with an independent space-filling SLE curve is simply connected. Hence, it is measurable with respect to the Borel $\sigma$-algebra generated by the pointed Gromov--Hausdorff--Prokhorov--uniform topology pointed at 0. (The LQG metric extends continuously to the boundary of the quantum wedge \cite[Proposition~1.6]{hm-metric-gluing}.) We use this $\sigma$-algebra in the above proposition as we did for the quantum cone in Proposition~\ref{prop:qc-mmspace-invariance}.

We require an alternative $\sigma$-algebra for $\gamma \in (\sqrt{2},2)$, since the $\frac{3\gamma}{2}$-quantum wedge is a beaded surface in this regime. The different beads lie at an infinite distance from each other by the definition of the internal metric. We use the Borel $\sigma$-algebra with respect to the following metric topology on the space of equivalence classes of curve-decorated beaded metric measure spaces (with the property that the curve enters the beads in chronological order and does not re-enter any bead after entering a subsequent bead) modulo measure-and-curve-preserving isometries. It is given by an extension of the Prokhorov--uniform metric on curve-decorated beaded domains defined in \cite[Section~2.2.5]{gwynne-miller-char}, where we use the GHPU metric instead.
\begin{enumerate}
    \item Let $\mathbb{M}^{\mathrm{GHPU}}$ be the space of equivalence classes of compact metric measure spaces decorated with a continuous curve modulo measure-preserving, curve-preserving isometries. Let $d^{\mathrm{GHPU}}$ be the GHPU metric on $\mathbb{M}^{\mathrm{GHPU}}$ defined in \cite[Section~1.3]{gwynne-miller-uihpq}.
    \item Given a curve-decorated beaded metric measure space $\mathcal S$, for $t\geq 0$, let $\mathcal K_t \in \mathbb{M}^{\mathrm{GHPU}}$ be the bead of $\mathcal S$ with the property that the sum of the measures of the previous beads (not including the bead itself) is at least $t$, equipped with the curve restricted to this bead. We view $\mathcal K$ as a function $[0,\infty) \to \mathbb{M}^{\mathrm{GHPU}}$ defined for almost every $t$. 
    to it a function $\mathcal K:[0,\infty) \to \mathbb{M}^{\mathrm{GHPU}}$ defined as 
    \item Let $\mathbb{M}_{\mathrm{bead}}^{\mathrm{GHPU}}$ be the set of all Borel measurable functions $\mathcal{K}:[0,\infty) \to \mathbb{M}_{\mathrm{bead}}^{\mathrm{GHPU}}$ which are defined almost everywhere. Define a metric on $\mathbb{M}_{\mathrm{bead}}^{\mathrm{GHPU}}$ by 
    \begin{equation}
        d_{\mathrm{bead}}^{\mathrm{GHPU}} (\mathcal{K}, \widetilde{\mathcal{K}}) = \int_0^\infty e^{-t} \left( 1 \wedge d^{\mathrm{GHPU}}(\mathcal K_t, \widetilde{\mathcal{K}}_t) \right) dt.
    \end{equation}
\end{enumerate}
Observe that the total contribution to $d_{\mathrm{bead}}^{\mathrm{GHPU}}$ of beads of LQG measure less than $\ep$ is bounded and tends to zero as $\ep \to 0$. 
\end{rem}

When $\gamma^2 = \kappa = 16/\kappa'$, define the process $\{L_t\}_{t\in \R}$ (resp.\ $\{R_t\}_{t\in \R}$) to be the change in the left (resp.\ right) quantum boundary length of $\eta(-\infty,t]$ with respect to 0. Here is the precise construction of this process. Recall the countable dense set of points $\{z_k\}_{k\in \N}$ that we used to define the whole-plane space-filling SLE$_{\kappa'}$ curve $\eta$. If $z_k = \eta(t)$, then the left (resp.\ right) boundary of $\eta(-\infty,t]$ is $\eta_{z_k}^+$ (resp.\ $\eta_{z_k}^-$), which is an SLE$_{\kappa}$-type curve. Given a GFF-type field $h$ and an independent SLE$_{\kappa}$-type curve $\tilde \eta$, we can define the \textit{$\gamma$-LQG length measure} $\nu_h$ on $\tilde \eta$ \cite{shef-kpz,shef-zipper, benoist-lqg-chaos}. The quantum length of the entire flow line $\eta_{z_k}^+$ is infinite, but it a.s.\ merges with $\eta_0^+$, so it makes sense to define the difference $L_t$ between the quantum boundary lengths of $\eta_{z_k}^+$ and $\eta_0^+$ as illustrated in Figure~\ref{fig:peanosphere-bm}. $R_{\tau_k}$ is defined analogously using $\eta_{z_k}^-$ and $\eta_0^-$. A priori, this defines the processes $L$ and $R$ on a countable dense subset of (random) times $t$. 

\begin{figure}
    \centering
    \includegraphics[scale=.9]{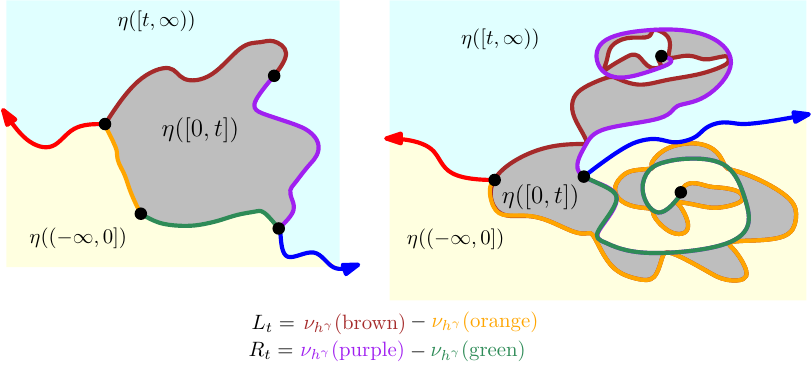}
    \caption{Illustration of the quantum boundary length process $(L,R)$ when $\kappa'\geq 8$ (left) and $\kappa' \in (4,8)$ (right). Suppose $z_k = \eta(t)$, where $t>0$. The left (resp.\ right) boundary of $\eta(-\infty,0]$, i.e., the flow line $\eta_0^+$ (resp.\ $\eta_0^-)$, is the union of red and orange (resp.\ green and blue) curves. The left (resp.\ right) boundary of $\eta(-\infty,t]$, i.e., the flow line $\eta_{z_k}^+$ (resp.\ $\eta_{z_k}^-$), is the union of red and brown (resp.\ purple and blue) curves.}
    \label{fig:peanosphere-bm}
\end{figure}

The following main theorem of mating-of-trees theory states that, almost surely, the processes $L$ and $R$ can be continuously extended to all real times and that they together give a complete description of the quantum surface $(\C,h^\gamma,0,\infty,\eta)$. The term mating-of-trees comes from the observation that the collections of flow lines $\{\eta_z^+\}_{z\in \C}$ and $\{\eta_z^-\}_{z\in \C}$ are trees that we can ``mate" to obtain the curve-decorated quantum surface $(h^\gamma,\eta)$.

\begin{thm}[{\cite[Theorems 1.9 and 1.11]{dms-lqg-mating}, \cite{kappa8-cov}}]\label{thm:boundary-length-BM}
    Let $(\C, h^\gamma,0,\infty,\eta)$ be the circle average embedding of a $\gamma$-quantum cone decorated by an independent space-filling SLE$_{\kappa'}$ curve, which is then reparemeterized according to the $\gamma$-LQG measure $\mu_{h^\gamma}$. Let $L_t$ (resp.\ $R_t$) denote the $\gamma$-LQG length $\nu_{h^\gamma}$ of the left (resp.\ right) boundary of $\eta(-\infty,t]$ minus that of $\eta(-\infty,0]$. Then $(L,R)$ evolves as a correlated two-dimensional Brownian motion. In particular, there is a non-random constant $a>0$ such that
    \begin{equation}
        \mathrm{Var}(L_t) = a|t|, \;\;\; \mathrm{Var}(R_t) = a|t|, \;\;\; \text{and} \;\;\; \mathrm{Cov}(L_t,R_t) = -a\cos\left(\frac{4\pi}{\kappa'}\right)|t| \;\;\; \text{for} \;\;\; t \in \R.
    \end{equation}
    Moreover, the pair $(L,R)$ almost surely determines both $h^\gamma$ and $\eta$ up to a rigid rotation of $\C$ about the origin.
\end{thm}

\subsection{Proof strategy}\label{sec:proof-strategy}

For the proof of Theorem~\ref{thm:main-thm}, we consider the configuration $(\C,h^\gamma,0,\infty,\eta)$ of a $\gamma$-quantum cone decorated with an independent whole-plane space-filling SLE$_{\kappa'}$, where $\kappa' = 16/\gamma^2$. The proof follows the following outline.

\begin{enumerate}
    \item In Section~\ref{sec:tightness}, we prove that $\{\mathfrak b_\ep^{-1} N_\ep(\eta[s,t];D_{h^\gamma})\}_{\ep \in (0,1)}$ is tight for each $s<t$ using estimates for the space-filling SLE and the LQG metric. 
    Hence, the infinite-dimensional random vector $(\mathfrak b_\ep^{-1} N_\ep (\eta[s,t];D_{h^\gamma}): s,t\in \Q, s<t)$ has a weak subsequential limit. Denote it as $(X_{[s,t]}: s<t)$.
    
    \item In Section~\ref{sec:additivity}, we show that $X_{[s,t]}$ is additive: i.e., $X_{[r,t]} = X_{[r,s]} + X_{[s,t]}$ a.s.\ for every rational $r<s<t$. 
    Hence, we can define a ``Minkowski content process" $\{Y_t\}_{t\in \Q}$ where $X_{[s,t]} = Y_t- Y_s$. Roughly speaking, this step holds because the Minkowski dimension of $\bd\eta[s,t]$ is less than $d_\gamma$.

    \item In Section~\ref{sec:minkowski-content-determinstic}, we prove that $\{Y_t\}_{t\in \Q}$ can be extended a continuous process on $\R$ with independent and stationary increments. We also show that $X_{[s,t]} = Y_t-Y_s >0$ a.s.\ for every $s<t$.
    By Blumenthal's zero-one law, a strictly increasing Brownian motion cannot be random. Thus, $t\mapsto Y_t$ must be a deterministic linear function.
    
    \item Our choice \eqref{eqn:scaling-constant} of $\mathfrak b_\ep$ ensures that the slope of $Y_t$ is exactly 1.    
    Any subsequence of $\mathfrak b_\ep^{-1} N_\ep (\eta[s,t];D_{h^\gamma})$ has a further subsequence which converges weakly to $t-s$, so $\lim_{\ep \to 0} \mathfrak b_\ep^{-1} N_\ep (\eta[s,t];D_{h^\gamma}) = t-s$ in probability (Proposition~\ref{prop:main-thm-qc-sle}).
    
    \item In Section~\ref{sec:generalization}, we extend to general sets and other GFF-type fields.
\end{enumerate}

We remark that our proof is similar in a superficial sense to Le Gall's proof in \cite{legall-measure} that the volume measure on the Brownian sphere $\mathbf m_\infty$ is a Hausdorff measure. In his proof, Le Gall considers the natural projection map $\mathbf p: [0,1] \to \mathbf m_\infty$ and computes the Hausdorff measures of segments $\mathbf p[s,t]$ of this space-filling curve. However, the curve $\mathbf p$ is distinct from the space-filing SLE$_{\kappa'}$ curve $\eta$ appearing in our proof, and correspondingly, the two proofs use a different set of tools (Brownian snake vs.\ mating-of-trees).

As for the possibility of constructing the $\gamma$-LQG measure as a Hausdorff measure for general $\gamma \in (0,2)$, a key step in such a construction would be to identify the suitable gauge function. This comes down to finding an up-to-constant estimate for the volume of LQG metric balls for general $\gamma \in (0,2)$. In \cite{legall-measure}, Le Gall proves the estimate $\mathrm{Vol}(\mathcal B_r(x)) = \Theta(h(r))$ as $r\to 0$ for the Brownian sphere, where $h(r) = r^4 \log\log(1/r)$; from this estimate, $h(r)$ is identified as the correct gauge function. On the other hand, for general $\gamma \in (0,2)$, the best available estimate for the $\gamma$-LQG volume of $\gamma$-LQG metric balls is from \cite{afs-metric-ball}, which states $\mu_h(\mathcal B_r(z; D_h)) = r^{d_\gamma + o(1)}$ as $r\to 0$, (We have used this fact extensively in our paper.) Le Gall's proof of the up-to-constant metric ball volume estimate strongly relies on the Brownian snake encoding of the Brownian sphere; currently, we do not have an alternative method to improve the estimate for general $\gamma$.

\section{Tightness of Minkowski content approximations} \label{sec:tightness}
	
Let $(\C, h^\gamma, 0,\infty)$ be the circle average embedding of a $\gamma$-quantum cone and let $\eta$ be an independent whole-plane space-filling SLE$_{\kappa'}$ curve from $\infty$ to $\infty$. Let $\eta$ be parameterized by $\mu_{h^\gamma}$. As discussed in Section~\ref{sec:minkowski-content-def}, we use 
\begin{equation}\label{eqn:minkowski-content-ansatz}
    \lim_{\ep \to 0} \ep^{d_\gamma} N_\ep(\eta[s,t];D_{h^\gamma})
\end{equation}
as an ansatz for the Minkowski content of $\eta[s,t]$ with respect to the $\gamma$-LQG metric $D_{h^\gamma}$. The first step is to show that this limit exists along subsequences. 

The goal of this section is to show that for each fixed $s<t$, the family of random variables 
\begin{equation}\label{eqn:tight-family-mc-approx}
\{\ep^{d_\gamma} N_\ep(\eta[s,t];D_{h^\gamma})\}_{\ep \in (0,1)}
\end{equation}
is tight, so that a subsequential limit in distribution exists for $\ep^{d_\gamma} N_\ep(\eta[s,t];D_{h^\gamma})$ as $\ep \to 0$. Moreover, the tightness of \eqref{eqn:tight-family-mc-approx} is also used to show $\mathfrak{b}_\ep\asymp \ep^{d_\gamma}$ (Corollary~\ref{cor:normalization-constant-comparison}); this is why we first use $\ep^{d_\gamma}$ in \eqref{eqn:minkowski-content-ansatz} rather than $\mathfrak{b}_\ep$ as in the statement of Theorem~\ref{thm:main-thm}. 

\begin{rem*}
    All results in this section hold for any fixed $\gamma \in (0,2)$ and $\kappa' \in (4,\infty)$, even when $\gamma^2 \neq 16/\kappa'$.
\end{rem*}

\subsection{Exit time of space-filling SLE from an LQG metric ball}

The key estimate in showing the tightness of \eqref{eqn:tight-family-mc-approx} is an upper bound for the probability that the $D_{h^\gamma}$-diameter of a space-filling SLE$_{\kappa'}$ segment $\eta[s,t]$ is large (Proposition~\ref{prop:sle-diameter}). More precisely, we prove an equivalent estimate, which is an upper bound for the probability that the exit time
\begin{equation} 
   \tau_1 := \inf\{t>0: \eta(t) \notin \mathcal{B}_1(0;D_{h^\gamma})\}
\end{equation}
is small.

\begin{prop}\label{prop:exit-time}
    The lower tail of the exit time $\tau_1$ is superpolynomially small. That is, for each $p>0$, 
    \begin{equation}
        \Prob \{\tau_1 \leq \ep\} = o(\ep^p) \quad \text{as } \ep\to 0^+.
    \end{equation}
\end{prop}

The key input for the proof of Proposition~\ref{prop:exit-time} is the following relation between the LQG distance and the circle averages of GFF. Recall that $\xi = \gamma/d_\gamma$ is the $\gamma$-dependent factor which appears in the definition \eqref{eqn:LFPP-metric} of the LQG metric.

\begin{lem}[{\cite[Proposition 3.15]{lqg-metric-estimates}}]\label{lem:lqg-metric-circle-average-estimate}
    Let $h$ be a whole-plane GFF normalized so that $h_1(0) = 0$. Then, with superpoynomially high probability as $C\to \infty$, at a rate which is uniform in the choice of $z\in \C\setminus \{0\}$,
    \begin{equation}
        D_{h - \gamma \log |\cdot|}(0,z;B_{4|z|}(0)) \leq C \int_{-\log \frac{|z|}{2}}^\infty \left( e^{\xi h_{e^{-t}}(0) - \xi(Q-\gamma)t} + e^{\xi h_{e^{-t}}(z) - \xi Q t} \right)dt.
    \end{equation}
\end{lem}

Let us first give a heuristic argument for Proposition~\ref{prop:exit-time}. To begin, $\bd\mathcal B_1(0;D_{h^\gamma})$ is macroscopic with high probability (i.e., it has constant-order Euclidean diameter). Since $D_{h^\gamma}(0,\eta(\tau_1)) =1$, there has to be some not too small $r>0$ such that either $h_r(0)$ or $h_r(\eta(\tau_1))$ is large, along the lines of Lemma~\ref{lem:lqg-metric-circle-average-estimate}. By Proposition~\ref{prop:SLE-includes-Euclidean-ball}, $\eta[0,\tau_1] \cap B_r(0)$ and $\eta[0,\tau_1] \cap B_r(\eta(\tau_1))$ each contain a macroscopic Euclidean ball. By comparing the LQG volumes of these Euclidean balls with $e^{\gamma h_r(0)}$ and $e^{\gamma h_r(\eta(\tau_1))}$, we conclude that with high probability, $\tau_1 = \mu_{h^\gamma}(\eta[0,\tau_1])$ cannot be too small.

There are three main points in making this heuristic rigorous. First, the laws of $h^\gamma$ and $h-\gamma \log|\cdot|$ agree only when restricted to the unit disk $\ud$. To this end, we analyze 
\begin{equation} \label{eqn:sle-dist-time}
   \tau_\ep := \inf\{t>0: \eta(t) \notin \mathcal{B}_\ep(0;D_{h^\gamma})\}
\end{equation}
instead of $\tau_1$, because $\mathcal B_{\ep}(0;D_{h^\gamma})\subset \ud$ with increasingly high probability as $\ep \to 0$. Observe from Proposition~\ref{prop:qc-mmspace-invariance} that for each $\ep>0$, 
\begin{equation}\label{eqn:exit-time-scaling}
    \tau_\ep \stackrel{d}{=} \ep^{d_\gamma}\tau_1.
\end{equation}
Second, the point $z$ in Lemma~\ref{lem:lqg-metric-circle-average-estimate} has to be deterministic, whereas the point $\eta(\tau_1)$ is random. The idea to get around this issue is to take a union bound over $z \in \ud \cap \ep^s \Z^2$ for a large constant $s$ and show that there is some point in $\ud \cap \ep^s\Z^2$ which is close to $\eta(\tau_\ep)$ in both Euclidean and LQG distances. Third, we can compare $\mu_{h^\gamma}(B_r(w))$ and $h_r(w)$ only for fixed $w$ and $r$. We again need to take a union bound over Euclidean balls $B_r(w)$, polynomially many in $\ep$, that are possibly contained in $\eta[0,\tau_1]$.

\begin{figure}
    \centering
    \includegraphics[scale=.9]{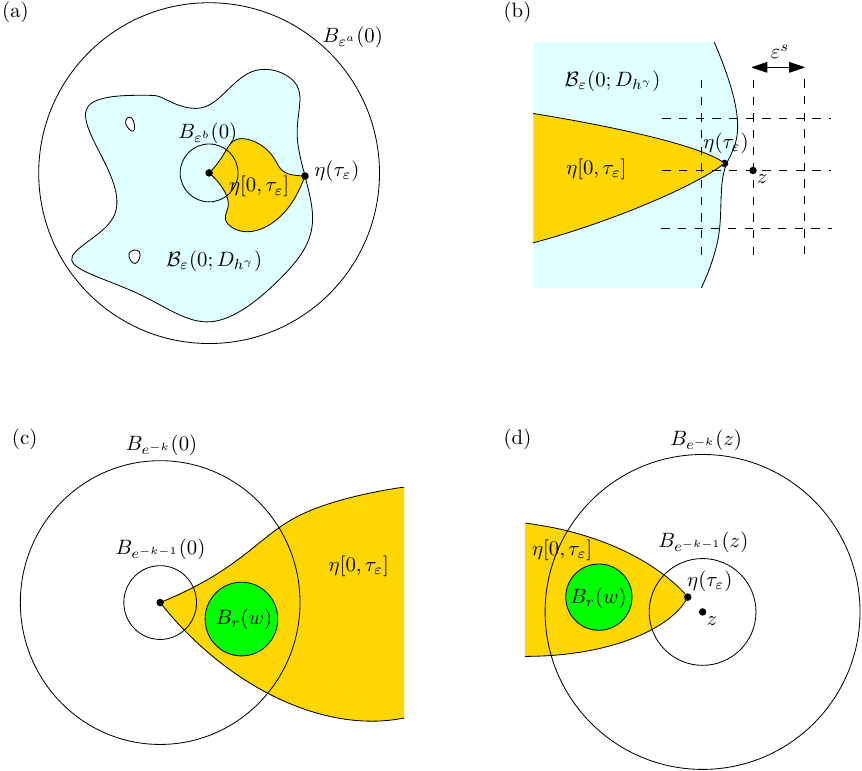}
    \caption{Proof of Proposition~\ref{prop:exit-time}. Let $h = h^\gamma + \gamma \log |\cdot|$ and $p>0$. The constants $a,b,s,N,q,\zeta$ are defined in \eqref{eqn:exit-time-constants-1} and \eqref{eqn:exit-time-constants-2} in terms of $p$. \textbf{(a)~Step~\ref{exit-time-proof:step-1}:} With probability $1-O(\ep^{p/3})$, the boundary of $\mathcal B_\ep(0;D_{h^\gamma})$ is contained within the Euclidean annulus $B_{\ep^a}(0) \setminus \overline{B_{\ep^b}(0)}$. In particular, $\ep^b < |\eta(\tau_\ep)| < \ep^a$. \textbf{(b)~Step~\ref{exit-time-proof:step-2}:} With probability $1-O(\ep^{p-Q/\xi})$, there exists $z \in \ep^s \Z^2$ for which $|\eta(\tau_\ep)-z|<\ep^s$ and $D_{h^\gamma}(\eta(\tau_\ep),z) < \ep/2$. Let $k_0 = \lfloor -\log(|z|/2)\rfloor$ and $k_1 = k_0 + N\log \ep^{-1}$. Conditioned on the previous events, with superpolynomially high probability, at least one of the following two cases holds. \textbf{(c)~Steps~\ref{exit-time-proof:step-3} and \ref{exit-time-proof:step-4}, Case 1:} There exists an integer $k \in [k_0, k_1]$ such that $\gamma h_{e^{-k}}(0) \geq \gamma (Q-\gamma)k + (d_\gamma + \frac{1}{3}) \log \ep$. For this $k$, there exists a Euclidean ball $B_{e^{-k(1+\zeta)}}(w)$ contained in the intersection of $\eta[0,\tau_\ep]$ and $B_{e^{-k}}(0) \setminus \overline{B_{e^{-k-1}}(0)}$ with $w\in e^{-k(1+\zeta)}\Z^2$. \textbf{(d)~Steps~\ref{exit-time-proof:step-3} and \ref{exit-time-proof:step-4}, Case 2:} There exists an integer $k \in [k_0, k_1]$ such that $\gamma h_{e^{-k}}(z) \geq \gamma Qk + (d_\gamma + \frac{1}{3}) \log \ep$. For this $k$, there exists a Euclidean ball $B_{e^{-k(1+\zeta)}}(w)$ contained in the intersection of $\eta[0,\tau_\ep]$ and $B_{e^{-k}}(z) \setminus \overline{B_{e^{-k-1}}(z)}$ with $w\in e^{-k(1+\zeta)}\Z^2$. \textbf{Steps~\ref{exit-time-proof:step-5} and \ref{exit-time-proof:step-6} (not visualized):} We obtain a lower bound of $\mu_{h^\gamma}(B_r(w))$ in terms of $h_r(w)$, which we compare with either $h_{e^{-k}}(0)$ (Case~1) or $h_{e^{-k}}(z)$ (Case~2). For sufficiently large $p$, on the event that all of the previous conditions hold, $\tau_\ep \geq \mu_{h^\gamma}(B_r(w)) \geq \ep^{d_\gamma + 1}$. By \eqref{eqn:exit-time-scaling}, this implies $\Prob\{\tau_1 \geq \ep\} = 1-O(\ep^p)$ for any given $p>0$. }
    \label{fig:exit-time-proof}
\end{figure}

\begin{proof}[Proof of Proposition~\ref{prop:exit-time}] 
If we show that $\tau_\ep = \mu_{h^\gamma}(\eta[0,\tau_\ep]) > \ep^{d_\gamma+1}$ with superpolynomially high probability as $\ep\to 0$, then the proposition follows by \eqref{eqn:exit-time-scaling}. 

Below, we describe an event consisting of six steps on which $\mu_{h^\gamma}(\eta[0,\tau_\ep]) > \ep^{d_\gamma+1}$. These events are stated using constants $0<a<b<s$, $0 < N < s-b$, $\zeta \in (0,1)$, and $q>2$. Eventually, these constants will be chosen in terms of a single constant $p>0$, which we will eventually allow to be arbitrarily large. We shall then verify that the stated event holds with probability $1 - O(\ep^{p/3})$ as $\ep \to 0$ using the lemmas stated and proven just after the main body of the proof.

Throughout, $h$ refers to the field $h^\gamma + \gamma \log |\cdot|$, whose restriction to the unit disk agrees in law with the corresponding restriction of a whole-plane GFF.

\begin{enumerate}
    \item \label{exit-time-proof:step-1} \textit{The points which lie at $D_{h^\gamma}$-distance $\ep$ from the origin --- i.e., $\partial \mathcal{B}_\ep(0;D_{h^\gamma})$ --- are within the Euclidean annulus $\overline{B_{\ep^a}(0)} \setminus B_{\ep^b}(0)$. In particular, $\ep^b \leq |\eta(\tau_\ep)| \leq \ep^a$. (Figure~\ref{fig:exit-time-proof}(a))}

    \item \label{exit-time-proof:step-2} \textit{For $\ep^b \leq |\eta(\tau_\ep)| \leq \ep^a$, there exists a point $z \in \ep^s \mathbb{Z}^2$ satisfying $\ep^b \leq |z| \leq \ep^a$ which is close to $\eta(\tau_\ep)$ in both Euclidean and LQG distances: $|\eta(\tau_\ep) - z| < \ep^s$ and $D_{h^\gamma}(\eta(\tau_\ep),z) \leq \ep/2$ precisely. The latter condition implies $D_{h^\gamma}(0,z) \geq \ep/2$ since $D_{h^\gamma}(0,\eta(\tau_\ep))=\ep$. (Figure~\ref{fig:exit-time-proof}(b))}
    
    \item \label{exit-time-proof:step-3} \textit{Given the point $z$ found in the previous step, let $k_0 = \lfloor -\log(|z|/2) \rfloor$. There exists an integer $k \in [k_0, k_0 + N\log \ep^{-1}]$ such that at least one of the circle averages $h_{e^{-k}}(0)$ or $h_{e^{-k}}(z)$ is bounded below in terms of the logarithm of the LQG distance $D_{h^\gamma}(0,z)$. Specifically, at least one of the following two cases is true.}
    \begin{itemize}
        \item \textit{Case 1: There exists $k \in [k_0, k_0 + N\log \ep^{-1}]$ such that \begin{equation}e^{\gamma h_{e^{-k}}(0)} \geq 2^{d_\gamma} \ep^{\frac{1}{3}} e^{\gamma (Q-\gamma) k} \left( D_{h^\gamma}(0,z) \right)^{d_\gamma} .\end{equation}}

        \item \textit{Case 2: There exists $k \in [k_0, k_0 + N\log \ep^{-1}]$ such that \begin{equation}e^{\gamma h_{e^{-k}}(z)} \geq 2^{d_\gamma} \ep^{\frac{1}{3}} e^{\gamma Q k} \left(D_{h^\gamma}(0,z)\right)^{d_\gamma} .\end{equation}}
    \end{itemize}
    \textit{Since $D_{h^\gamma}(0,z) \geq \ep/2$ by our choice of $z$, we have} \[e^{\gamma h_{e^{-k}}(0)} \geq e^{\gamma (Q-\gamma) k} \ep^{d_\gamma+\frac{1}{3}} \quad \textit{in Case 1,} \quad e^{\gamma h_{e^{-k}}(z)} \geq e^{\gamma Q k} \ep^{d_\gamma + \frac{1}{3}} \quad \textit{in Case 2.}\]
    
    \item \label{exit-time-proof:step-4} \textit{Given the integer $k$ found in the previous step, let $r = e^{-k(1+\zeta)}$. There exists a point $w\in r\Z^2$ such that:  
    \begin{itemize}
        \item In Case 1, $B_r(w) \subset \eta[0,\tau_\ep] \cap (B_{e^{-k}}(0) \setminus \overline{B_{e^{-k-1}}(0)})$. (Figure~\ref{fig:exit-time-proof}(c))
        \item In Case 2, $B_r(w) \subset \eta[0,\tau_\ep] \cap (B_{e^{-k}}(z) \setminus \overline{B_{e^{-k-1}}(z)})$. (Figure~\ref{fig:exit-time-proof}(d))
    \end{itemize}}
    
    \item \label{exit-time-proof:step-5} \textit{For $z$, $k$, and $w$ as in the previous three steps, we have the following comparison of circle averages.}
    \begin{itemize}
        \item \textit{Case 1: $|h_r(w) - h_{e^{-k}}(0)| \leq kq\zeta$. Therefore, $e^{\gamma h_r(w)} \geq e^{\gamma (Q-\gamma - q\zeta})k \ep^{d_\gamma+\frac{1}{3}}$.}
    
        \item \textit{Case 2: $|h_r(w) - h_{e^{-k}}(z)| \leq kq\zeta$. Therefore, $e^{\gamma h_r(w)} \geq e^{\gamma (Q -q \zeta})k \ep^{d_\gamma+\frac{1}{3}}$.}
    \end{itemize}
     
    \item \label{exit-time-proof:step-6} \textit{The LQG measure of the Euclidean ball $B_r(w)$ is bounded below by}
    \begin{equation}\label{eqn:euclidean-ball-volume-lower-bound} 
        \mu_h (B_r(w))\geq \ep^{\frac{1}{3}} r^{\gamma Q} e^{\gamma h_r(w)}.  
    \end{equation}
\end{enumerate}

These are the six steps comprising the event on which $\tau_\ep > \ep^{d_\gamma + 1}$ with the choice of constants stated below (\eqref{eqn:exit-time-constants-1} and \eqref{eqn:exit-time-constants-2}). Let us first verify that on the stated event, we have the correct lower bound on $\tau_\ep$ as claimed. We begin by substituting the lower bound for $h_r(w)$ claimed in Step~\ref{exit-time-proof:step-5} into \eqref{eqn:euclidean-ball-volume-lower-bound}. Recall $r = e^{-k(1+\zeta)}$ and $h^\gamma = h - \gamma \log |\cdot|$.
\begin{itemize}
    \item Case 1: Since $B_r(w) \subset B_{e^{-k}}(0)$, we have $\mu_{h^\gamma}(B_r(w)) \geq e^{\gamma^2 k} \mu_h(B_r(w))$. Hence,
    \begin{equation} 
        \mu_{h^\gamma}(B_r(w)) \geq e^{\gamma^2 k} \mu_h(B_r(w)) \geq e^{\gamma^2 k}\ep^{\frac{1}{3}} (e^{-k(1+\zeta)})^{\gamma Q} e^{\gamma (Q-\gamma - q\zeta})k \ep^{d_\gamma+\frac{1}{3}} 
    \end{equation}
    \item Case 2: Since $B_r(w) \subset \ud$, we have $\mu_{h^\gamma}(B_r(w)) \geq \mu_h(B_r(w))$. Hence,
    \begin{equation} \mu_{h^\gamma}(B_r(w)) \geq \mu_h(B_r(w)) \geq \ep^{\frac{1}{3}} (e^{-k(1+\zeta)})^{\gamma Q} e^{\gamma (Q - q\zeta})k \ep^{d_\gamma+\frac{1}{3}}. \end{equation}
\end{itemize}
Collecting exponents, we conclude in both cases that 
\begin{equation} \label{eqn:exit-time-lower-end}
    \tau_\ep \geq \mu_{h^\gamma}(B_r(w)) \geq \ep^{d_\gamma+\frac{2}{3}} e^{-\gamma(Q+q)\zeta k}.
\end{equation}

We now set
\begin{equation}\label{eqn:exit-time-constants-1}
    a = \frac{1}{\xi^2 p} \asymp{p^{-1}} , \quad b = \frac{p}{(Q-\gamma)^2} \asymp p , \quad s = \left(2+\frac{\gamma Q}{Q^2/2-2}\right)b \asymp p  
\end{equation}
where $\asymp$ means equality up to $\gamma$-dependent multiplicative constants when $p$ is large. We also choose 
\begin{equation}\label{eqn:exit-time-constants-2}
    N = \sqrt[4]{1 + \frac{6s}{p}}\sqrt{\xi p} \asymp p^{1/2} , \quad q = N^7 \asymp p^{7/2} , \quad  \zeta = N^{-10} \asymp p^{-5} .
\end{equation}
Note that $s-b> b= \frac{p}{(Q-\gamma)^2}>N$ for sufficiently large $p$. Since we only consider $k \leq (b+N) \log \ep^{-1} = O(N^2\log \ep^{-1})$ in Step \ref{exit-time-proof:step-3},
\begin{equation} 
    e^{-\gamma(Q+q)\zeta k} \geq \ep^{O(N^{-1})} \geq \ep^{1/3} 
\end{equation}
for all $0< \ep < 1/2$ given that $p$ is sufficiently large. Hence, the lower bound for $\tau_\ep$ from~\eqref{eqn:exit-time-lower-end} satisfies
\begin{equation}\label{eqn:exit-time-conclusion}
    \tau_\ep > \ep^{d_\gamma + \frac{2}{3}} e^{-\gamma(Q+q)\zeta k} \geq \ep^{d_\gamma + 1}
\end{equation}
for all large $p$.

It remains to estimate the probability of the above event. We compute this probability step-by-step given the constants in \eqref{eqn:exit-time-constants-1} and \eqref{eqn:exit-time-constants-2}.

\begin{enumerate}
    \item In Lemma~\ref{lem:lqg-euclidean-balls-mutual-inclusion}, we show that for sufficiently large $p$, Step~\ref{exit-time-proof:step-1} occurs with probability $1-O(\ep^{p/3})$ as $\ep \to 0$. 

    \item In Lemma~\ref{lem:lattice-point-near-exit}, we show that given our choices of $a$, $b$, and $s$, truncated on the event $\ep^b \leq |\eta(\tau_\ep)| \leq \ep^a$, the probability that Step~\ref{exit-time-proof:step-2} fails is bounded above by $O(\ep^{p-Q/\xi})$ as $\ep \to 0$.

    \item In Lemma~\ref{lem:large-circle-avg}, we show that for each fixed $z \in B_{1/4}(0) \setminus \{0\}$, Step~\ref{exit-time-proof:step-3} occurs with superpolynomially high probability as $\ep \to 0$ at a rate uniform in the choice of $z$. A union bound over $O(\ep^{-2s})$-many points of $z \in \ep^s \Z^2$ with $\ep^b \leq |z| \leq \ep^a$ gives a superpolynomially high probability for this step.

    \item Consider $\ep>0$ sufficiently small so that $\ep^s < \frac{1}{2e} \ep^{b+N}$ (this holds for all small enough $\ep$ since we chose our parameters so that $s > b+N$). If $\ep^b \leq |z| \leq \ep^a$ and $\eta(\tau_\ep) \in B_{\ep^s}(z)$, then $\eta[0,\tau_\ep] \cap (B_{e^{-k}}(0) \setminus \overline{B_{e^{-k-1}}(0)})$ and $\eta[0,\tau_\ep] \cap (B_{e^{-k}}(z) \setminus \overline{B_{e^{-k-1}}(z)})$ are both nonempty and the Euclidean diameters of the two intersections are at least $e^{-k-1}$. For the latter set, this is true because $\eta[0,\tau_\ep]$ is a connected set and 
    \[\eta(0) = 0 \notin B_{2e^{-k_0-1}}(z)\quad \text{and}\quad \eta(\tau_\ep) \in B_{\ep^s}(z)\subset B_{e^{-k_0+N\log \ep - 1}}(z),\] 
    where we used $\ep^s < \frac{1}{2e} \ep^{b+N} < \frac{|z|}{2e} \ep^N < e^{-k_0-1} \ep^N$. A similar argument works for the former set. 

    Hence, Proposition~\ref{prop:SLE-includes-Euclidean-ball} implies that for each lattice point $z \in \ep^s \Z^2 \cap (B_{\ep^a}(0) \setminus \overline{B_{\ep^b}(0)})$ and integer $k \in [k_0, k_0 +N \log \ep^{-1}]$, the following holds with superpolynomially high probability: on the event that $\eta(\tau_\ep) \in B_{\ep^s}(z)$, the sets $\eta[0,\tau_\ep] \cap (B_{e^{-k}}(0) \setminus \overline{B_{e^{-k-1}}(0)})$ and $\eta[0,\tau_\ep] \cap (B_{e^{-k}}(z) \setminus \overline{B_{e^{-k-1}}(z)})$ each contain a Euclidean ball of radius $e^{-(k+1)(1+\frac{\zeta}{2})}$. For sufficiently small $\ep$ (hence sufficiently large $k_0$), we can always find within each of these Euclidean balls a smaller Euclidean ball $B_r(w)$ where $r= e^{-k(1+\zeta)}$ and $w \in r\Z^2$. Now take a union bound over all such pairs $(z,k)$. The number of these pairs is at most some negative power of $\ep$, so Step \ref{exit-time-proof:step-4} holds with superpolynomially high probability in $\ep$.

    \item In Lemma~\ref{lem:circle-avg-compare}, we show that given an integer $k_0\geq 2$, for each
    \begin{equation} \label{eqn:grid-annulus}
        z \in \ep^s\Z^2 \cap \{z\in \C : |z| \in (2e^{-k_0-1}, 2e^{-k_0}] \}  ,
    \end{equation}
    the following event holds with probability $1-O(e^{-(q^2/2 - 2) \zeta k_0})$: the comparisons in Step \ref{exit-time-proof:step-5} hold for all possible choices of $k$ and $w$. There are $O(\ep^{-2s}e^{2k_0})$-many points $z$ satisfying~\eqref{eqn:grid-annulus} for each positive integer $k_0$. Since $|z| \in [\ep^a, \ep^b]$ and $k_0 = \lfloor -\log(|z|/2) \rfloor$, we only need to consider $k_0$ satisfying $a\log \ep^{-1} + \log(2/e) \leq k_0 \leq b\log \ep^{-1} + \log 2$. Taking a union bound over these $k_0$ and $z$, Step \ref{exit-time-proof:step-5} holds with probability
    \begin{equation} \label{eqn:circle-avg-union-error} 
        1 - O\left( \sum_{k_0 = \lfloor a\log \ep^{-1} \rfloor}^{\lceil b\log \ep^{-1} \rceil} \ep^{-2s} e^{2k_0} e^{- (q^2/2 - 2) \zeta k_0} \right) = 1 - O\left( \ep^{((\frac{q^2}{2}-2)\zeta - 2)a - 2s} \right) 
    \end{equation}
    provided that $(\frac{q^2}{2}-2)\zeta - 2>0$. The constants $q$ and $\zeta$ chosen in \eqref{eqn:exit-time-constants-2} satisfy this condition for all sufficiently large $p$. Moreover, with the further choice of constants $a$ and $s$ in \eqref{eqn:exit-time-constants-1}, the error term in~\eqref{eqn:circle-avg-union-error} satisfies
    \begin{equation}
        1 - O\left( \ep^{((\frac{q^2}{2}-2)\zeta - 2)a - 2s} \right) \geq 1- O\left(\ep^{\frac{N^4}{3}a - 2s}\right) \geq 1- O(\ep^{p/3}). 
    \end{equation}

    \item In Lemma~\ref{lem:lqg-mass-from-circle-avg}, we show that for each possible choice of $z$, $k$, and $w$, Step~\ref{exit-time-proof:step-6} holds with superpolynomially high probability as $\ep \to 0$ with rate uniform over all such choices. Now take a union bound over these $z$, $k$, and $w$, of which there are only polynomially many in $\ep$.
\end{enumerate}

To summarize, the probabilities of the events we truncate on each step are:
\begin{itemize}
    \item Steps~\ref{exit-time-proof:step-1} and \ref{exit-time-proof:step-5}: $1-O(\ep^{p/3})$;
    \item Step~\ref{exit-time-proof:step-2}: $1-O(\ep^{p-Q/\xi})$;
    \item Steps~\ref{exit-time-proof:step-3}, \ref{exit-time-proof:step-4}, and \ref{exit-time-proof:step-6}: superpolynomially high in $\ep$.
\end{itemize}
Since the choice of $p>0$ was arbitrary, we conclude that $\tau_\ep > \ep^{d_\gamma+1}$ holds with superpolynomially high probability in $\ep$ as $\ep \to 0$.
\end{proof}

We now state and prove the lemmas used in the proof of Proposition~\ref{prop:exit-time}, which are all based on standard LQG estimates. First, we record a property of whole-plane GFF that we will use repeatedly. It is a straightforward consequence of Lemmas~\ref{lem:circle-average-continuous} and \ref{lem:GFF-orthogonal-decomposition}.

\begin{lem}[{\cite[Lemma 3.4]{dg-critical-lqg}}]\label{lem:domain-markov-circle-average}
    Let $h$ be a whole-plane GFF normalized so that $h_1(0) = 0$. For each deterministic $r>0$, the field $(h-h_r(0))|_{B_r(0)}$ is independent from the circle average $h_r(0)$.
\end{lem}

The following lemma calculates the probability of the event $\{\ep^b \leq |\eta(\tau_\ep)| \leq \ep^a\}$ appearing in Step~\ref{exit-time-proof:step-1}.

\begin{lem}\label{lem:lqg-euclidean-balls-mutual-inclusion}
	Let $\mathcal{B}_\ep(0;D_{h^\gamma})$ be the LQG metric ball of radius $\ep$ and $B_r(0)$ be the Euclidean ball of radius~$r$, both centered at the origin. 
	\begin{enumerate}[(i)]
	    \item For each $p>0$, there exists a constant $C_p>0$ depending only on $p$ such that for every $\ep,r \in (0,1)$,
    	\begin{equation}
    	    \Prob \left\{\mathcal{B}_\ep(0;D_{h^\gamma}) \subset B_r(0)\right\} \geq 1 - C_p \ep^{p} r^{- (Q-\gamma)\xi p - \frac{1}{2}\xi^2 p^2}.
    	\end{equation}    	
        In particular, for each sufficiently large $p$, as $\ep \to 0$,
        \begin{equation}
    	    \Prob \left\{\mathcal{B}_\ep(0;D_{h^\gamma}) \subset B_{\ep^{1/(\xi^2 p)}}(0)\right\} \geq 1 - C_p \ep^{p/2- (Q-\gamma)/\xi} = 1 - O(\ep^{p/3}).
    	\end{equation}   
     
        \item There exists a constant $C>0$ such that for every $\ep \in (0,1)$ and $r \in (0,\frac{1}{2})$,
	    \begin{equation}
            \Prob \left\{B_r(0) \subset \mathcal{B}_\ep(0;D_{h^\gamma}) \right\} \geq 1 - C \ep^{-\frac{Q-\gamma}{\xi}} r^{\frac{(Q-\gamma)^2}{2}}.
        \end{equation}
        In particular, for each sufficiently large $p$, as $\ep \to 0$,
        \begin{equation}
            \Prob \left\{B_{\ep^{p/(Q-\gamma)^2}}(0) \subset \mathcal{B}_\ep(0;D_{h^\gamma}) \right\} \geq 1 - C \ep^{p/2 - (Q-\gamma)/\xi} = 1-O(\ep^{p/3}).
        \end{equation}
	\end{enumerate}
\end{lem}
\begin{proof}
    Let $h$ refer to the field $h^\gamma + \gamma \log |\cdot|$, whose restriction to $\ud$ agrees in law with the corresponding restriction of a whole-plane GFF.
    \begin{enumerate}[(i)]
	   \item If $\mathcal{B}_\ep(0;D_{h^\gamma}) \not\subset B_r(0)$, then $D_{h^\gamma}(0,\partial B_r(0)) < \ep$. By the Markov inequality,
	   \begin{equation}
		  \Prob\{ \mathcal B_\ep(0;D_{h^\gamma}) \not\subset B_r(0) \} \leq \Prob \left\{D_{h^\gamma}(0,\partial B_r(0)) < \ep \right\} \leq \ep^p\, \Exp \left[ \left( D_{h^\gamma}(0,\partial B_r(0)) \right)^{-p} \right].
	   \end{equation}
	   By the Weyl scaling and locality properties of the LQG metric, $e^{-\xi h_r(0)}D_{h^\gamma}(0,\partial B_r(0))$ is a.s.\ determined by $(h-h_r(0))|_{B_r(0)}$, which is independent of $h_r(0)$ by Lemma~\ref{lem:domain-markov-circle-average}. By this independence, \cite[Proposition 3.12]{lqg-metric-estimates}, and the fact that $h_r(0)\sim N(0,\log(1/r))$, there exists a constant $C_p>0$ such that 
	   \begin{equation}\begin{aligned}
	       &\Exp \left[ \left( D_{h^\gamma}(0,\partial B_r(0)) \right)^{-p} \right] \\&= r^{(\gamma-Q)\xi p}\, \Exp\left[ e^{-\xi p h_r(0)} \right] \Exp\left[ \left( r^{\xi(\gamma-Q)} e^{-\xi h_r(0)} D_{h^\gamma}(0,\partial B_r(0)) \right)^{-p} \right] \\& \leq C_p r^{-(Q-\gamma)\xi p -\frac{1}{2}\xi^2 p^2}.  \end{aligned}
        \end{equation}

        \item Let $0 < u < (\frac{4}{\gamma^2}-1)d_\gamma$. If $B_r(0) \not\subset \mathcal{B}_\ep(0;D_{h^\gamma})$, then $\sup_{w \in B_r(0)} D_{h^\gamma} (0,w) > \ep$, which implies $\sup_{w \in B_r(0)} D_{h^\gamma}^{B_{2r}(0)} (0,w) > \ep$. (Recall the definition \eqref{eqn:internal-metric} of the internal metric.) From \cite[Proposition 3.17]{lqg-metric-estimates}, there exists a constant $\widetilde C_u>0$ depending only on $u$ such that
        \begin{equation}  
            \Exp\left[ \left((2r)^{(\gamma-Q)\xi} e^{-\xi h_{2r}(0)} \sup_{w \in B_r(0)} D_h^{B_{2r}(0)} (0,w) \right)^u \right]  \leq \widetilde C_u.
        \end{equation}
        From the independence of $h_{2r}(0)$ and $(h-h_{2r}(0))|_{B_{2r}(0)}$ (Lemma~\ref{lem:domain-markov-circle-average}) and the fact that $h_{2r}(0) \sim N(0,\log(1/(2r)) )$, we have
        \begin{equation}\begin{aligned}
            & \Prob \left\{\sup_{w \in B_r(0)} D_{h^\gamma}^{B_{2r}(0)}(0,w) > \ep \right\} 
            \leq \ep^{-u}\, \Exp\left[ \left(\sup_{w \in B_r(0)} D_h^{B_{2r}(0)}(0,w)\right)^u \right] \\
            &= \ep^{-u} \, \Exp\left[ \left((2r)^{(Q-\gamma)\xi} e^{\xi h_{2r}(0)}\right)^u \right]\Exp\left[ \left((2r)^{(\gamma-Q) \xi} e^{-\xi h_{2r}(0)} \sup_{w \in B_r(0)} D_h^{B_{2r}(0)} (0,w) \right)^u \right] \\
            &\leq (2^u\widetilde C_u)\ep^{-u} r^{(Q-\gamma)\xi u -\frac{1}{2} \xi^2 u^2}.
        \end{aligned}\end{equation}
	We obtain the lemma by choosing $u = \frac{1}{\xi}(Q-\gamma) = \frac{1}{2}(\frac{4}{\gamma^2} - 1) d_\gamma$.
	\end{enumerate}\vspace{-20pt}
\end{proof}

The following lemma computes the probability of the event in Step~\ref{exit-time-proof:step-2}, in which we find a point $z$ on the lattice $\ep^s \Z^2$ which is close to $\eta(\tau_\ep)$ in both Euclidean and LQG distances. Again, recall the definition of the internal metric in \eqref{eqn:internal-metric}.

\begin{lem}\label{lem:lattice-point-near-exit}
	Suppose $0<a<b$ are given. For $s>b$, let $E_{\ep,s}$ be the event that for all $z \in \ep^s\mathbb{Z}^2$ with $\ep^b \leq |z| \leq \ep^a$, we have
	\begin{equation}
		\sup_{w \in B_{\ep^s}(z)} D_{h^\gamma}^{B_{2\ep^s}(z)}(z,w) \leq \frac{\ep}{2}.
	\end{equation}
	Then, as $\ep \to 0$,
	\begin{equation}
		\Prob(E_{\ep,s}) = 1- O\big( \ep^{(\frac{Q^2}{2}-2)s - \gamma Q b - \frac{Q}{\xi}} \big).
	\end{equation}
\end{lem}
\begin{proof}
	Recall the notation $h = h^\gamma + \gamma \log |\cdot|$. Consider a fixed $z \in \ep^s \Z^2$ with $\ep^b \leq |z| \leq \ep^a$ for now. For $\ep > 0$ is small enough that $|z| - 2\ep^s \geq \ep^b - 2\ep^s \geq \ep^b/2$, by the Weyl scaling axiom, we have
	\begin{equation} \begin{aligned}
        \sup_{w \in B_{\ep^s}(z)} D_{h^\gamma}^{B_{2\ep^s}(z)}(z,w) &\leq (|z|-2\ep^s)^{-\xi \gamma} \sup_{w \in B_{\ep^s}(z)} D_h^{B_{2\ep^s}(z)} (z,w) \\ &\leq \bigg(\frac{\ep^b}{2}\bigg)^{-\xi \gamma}\sup_{w \in B_{\ep^s}(z)} D_h^{B_{2\ep^s}(z)} (z,w). \end{aligned}
    \end{equation}
	Let $p \in (0,4d_\gamma/\gamma^2)$. By Lemma~\ref{lem:domain-markov-circle-average} and the translation invariance of the law of the whole-plane GFF modulo the additive constant as in~\eqref{eqn:gff-translate}, $e^{-\xi h_{2\ep^s}(z)} \sup_{w \in B_{\ep^s}(z)} D_h^{B_{2\ep^s}(z)} (z,w)$ is independent of $h_{2\ep^s}(z)$. By this independence and the fact that $h_{2\ep^s}(z) \sim N(0,\log(2\ep^s))$, we have 
	\begin{equation}\begin{aligned}
        & \Prob \left\{\sup_{w \in B_{\ep^s}(z)} D_{h^\gamma}^{B_{2\ep^s}(z)}(z,w) > \frac{\ep}{2} \right\} 
		\leq \Prob\left\{ \sup_{w \in B_{\ep^s}(z)} D_h^{B_{2\ep^s}(z)}(z,w) > \frac{\ep}{2} \bigg(\frac{\ep^b}{2}\bigg)^{\xi \gamma} \right\}\\
		&\leq \bigg(\frac{2^{1+\xi \gamma}}{\ep^{1+\xi \gamma b}}\bigg)^p \Exp\bigg[ \bigg(\sup_{w \in B_{\ep^s}(z)} D_h^{B_{2\ep^s}(z)}(z,w)\bigg)^p \bigg] \\
		&= \bigg(\frac{2^{1+\xi \gamma}}{\ep^{1+\xi \gamma b}}\bigg)^p \Exp\left[ \left((2\ep^s)^{\xi Q} e^{\xi h_{2\ep^s}(z)}\right)^p \right]\Exp\bigg[ \bigg((2\ep^s)^{-\xi Q} e^{-\xi h_{2\ep^s}(z)} \sup_{w \in B_{\ep^s}(z)} D_h^{B_{2\ep^s}(z)} (z,w) \bigg)^p \bigg]\\
        &= \bigg(\frac{2^{1+\xi \gamma}}{\ep^{1+\xi \gamma b}}\bigg)^p (2\ep^s)^{\xi Q p - \frac{1}{2}\xi^2 p^2}\Exp\bigg[ \bigg((2\ep^s)^{-\xi Q} e^{-\xi h_{2\ep^s}(z)} \sup_{w \in B_{\ep^s}(z)} D_h^{B_{2\ep^s}(z)} (z,w) \bigg)^p \bigg].
	\end{aligned}\end{equation}	By \cite[Proposition 3.9]{lqg-metric-estimates}, there exists a constant $C_p>0$ depending only on $p$ and $\gamma$ such that
	\begin{equation}  
        \Exp\left[ \left((2\ep^s)^{-\xi Q} e^{-\xi h_{2\ep^s}(z)} \sup_{w \in B_\ep^s(z)} D_h^{B_{2\ep^s}(z)} (z,w) \right)^p \right]  \leq C_p.
    \end{equation}
	By plugging this into the previous inequality, we get that there exists another constant $\widetilde C_p>0$ depending only on $p$ and $\gamma$ such that 
	\begin{equation} 
        \Prob \left\{\sup_{w \in B_\ep^s(z)} D_{h^\gamma}^{B_{2\ep^s}(z)}(z,w) > \frac{\ep}{2} \right\}  \leq \widetilde C_p \ep^{p(-1-\xi \gamma b + \xi Q s - \frac{1}{2}\xi^2 sp)}.
    \end{equation}
	The constant $\widetilde C_p$ does not depend on the choice of $z \in \overline{B_{\ep^a}(0)} \setminus B_{\ep^b}(0)$. Choose $p = Q/\xi = (2+\frac{\gamma^2}{2})\frac{d_\gamma}{\gamma^2} < \frac{4d_\gamma}{\gamma^2}$. The lemma follows by taking a union bound over $O(\ep^{-2s})$-many points $z\in \ep^s \Z^2$ such that $\ep^b \leq |z| \leq \ep^a$.
\end{proof}

The following lemma is used in Step~\ref{exit-time-proof:step-3}. This is the key step in the proof of Proposition~\ref{prop:exit-time}, in which we combine Lemma~\ref{lem:lqg-metric-circle-average-estimate} with $D_{h^\gamma}(0,\tau_\ep) = \ep$ to find an integer $k$ such that either $h_{e^{-k}}(0)$ or $h_{e^{-k}}(z)$ is large.

\begin{lem}\label{lem:large-circle-avg}
Let $N>0$ be given. Let $z\in B_{1/4}(0)\setminus \{0\}$ be a fixed point and let $k_0$ be the positive integer such that $2e^{-k_0-1} < |z| \leq 2e^{-k_0}$. Denote $k_1 = \lfloor k_0 + N\log C\rfloor$ and $h = h^\gamma + \gamma \log |\cdot|$. Then, with superpolynomially high probability as $C\to \infty$, at a rate uniform over the choice of $z$, there exists an integer $k \in [k_0, k_1]$ such that either
\begin{equation}\label{eqn:large-circle-avg}
    e^{\gamma h_{e^{-k}}(0)} \geq C^{-1} e^{\gamma (Q-\gamma) k} (D_{h^\gamma}^{\ud}(0,z))^{d_\gamma} \quad \text{or} \quad e^{\gamma h_{e^{-k}}(z)} \geq C^{-1} e^{\gamma Q k} (D_{h^\gamma}^{\ud}(0,z))^{d_\gamma}.
\end{equation}
\end{lem}
\begin{proof}
    From Lemma~\ref{lem:lqg-metric-circle-average-estimate}, it holds with superpolynomially high probability as $C\to \infty$, at a rate which is uniform in the choice of $z \in B_{1/4}(0) \setminus \{0\}$, that
\begin{equation} \label{eqn:use-circle-avg-int}
 D_{h^\gamma}^\ud (0,z) \leq D_{h^\gamma}^{B_{4|z|}(0)}(0,z) \leq C \int_{-\log\frac{|z|}{2}}^\infty \left( e^{\xi h_{e^{-t}}(0) - \xi (Q-\gamma) t} + e^{\xi h_{e^{-t}}(z) - \xi Qt} \right) dt . 
\end{equation}
    The main idea of the proof is to bound the above integral by a Riemann sum approximation.
    
    By Lemma~\ref{lem:circle-average-continuous}, $t\mapsto h_{e^{-t}}(z)$ is a standard Brownian motion for $t\geq k_0$ given the initial value $h_{e^{-k_0}}(z)$. For each integer $k \in [k_0, k_1-1]$, 
    by the reflection principle,
    \begin{equation}  \Prob \left\{\sup_{t \in [k,k+1]} |h_{e^{-t}}(z) - h_{e^{-k}}(z)| > \xi^{-1} \log C \right\} = O\left( \frac{e^{-\frac{(\log C)^2}{2\xi^2}}}{\log C}\right) \end{equation}  
    as $C\to \infty$. The constant here is uniform over $z$ and $k$.
    On the complementary event,
    \begin{equation} \label{eqn:exit-time-circle-average-z}  e^{\xi h_{e^{-t}}(z) - \xi Q t} \leq C e^{\xi h_{e^{-t}}(z) - \xi Q k} \quad \text{for all } t \in [k,k+1]. \end{equation}
    Since $Q-\gamma = (2/\gamma) - (\gamma/2) > 0$, the same lower bound applies to the probability that
    \begin{equation} \label{eqn:exit-time-circle-average-0}  e^{\xi h_{e^{-t}}(0) - \xi (Q-\gamma) t} \leq C e^{\xi h_{e^{-t}}(0) - \xi (Q-\gamma) k} \quad \text{for all } t \in [k,k+1]. \end{equation} 
    On the intersection of the events \eqref{eqn:exit-time-circle-average-z} and \eqref{eqn:exit-time-circle-average-0}, 
    \begin{equation} \label{eqn:exit-time-circle-average-integral-macroscopic} 
        \int_{k}^{k+1} \left( e^{\xi h_{e^{-t}}(0) - \xi (Q-\gamma)t} + e^{\xi h_{e^{-t}}(z) - \xi Q t} \right) dt \leq C \left( e^{\xi h_{e^{-k}}(0) - \xi (Q-\gamma) k} + e^{\xi h_{e^{-k}}(z) - \xi Q k} \right). 
    \end{equation} 
    
    We now deal with the integral from $k_1$ to infinity. Since $s \mapsto h_{e^{-s-k_1}}(z) - h_{e^{-k_1}}(z)$ for $s\geq 0$ is a standard Brownian motion, 
    \begin{equation}
        \int_0^\infty e^{\xi (h_{e^{-s-k_1}}(z) - h_{e^{-k_1}}(z)) - \xi Q s}\,ds \stackrel{d}{=} \frac{2}{\xi^2} Z_{\frac{2Q}{\xi^2}}, 
    \end{equation}
    where $Z_{\nu}$ is a Gamma random variable of index $\nu$ \cite{dufresne-exp-bm,yor-exp-bm}. That is, for $c\geq 0$,
    \begin{equation} 
        \Prob\{Z_\nu > c\} = \int_c^\infty \frac{s^{\nu-1} e^{-s}}{\Gamma(\nu)}\,ds.  
    \end{equation}
    In particular, there exists a number $u>0$ depending only on $\gamma$ such that
    \begin{equation}\label{eqn:exit-time-circle-average-microscopic-z}
        \Prob\left\{ \int_{k_1}^\infty e^{\xi h_{e^{-t}}(z) - \xi Qt}\,dt \leq C e^{\xi h_{e^{-k_1}}(z) - \xi Q k_1} \right\} = 1 - \Prob\left\{Z_\frac{2Q}{\xi^2} > \frac{\xi^2}{2}C\right\} = 1 - O\left(e^{-u  C }\right)
    \end{equation} 
    as $C \to \infty$ with the rate uniform on $z$. Similarly, 
    \begin{equation}\label{eqn:exit-time-circle-average-microscopic-0}
        \begin{aligned} \Prob\left\{ \int_{k_1}^\infty e^{\xi h_{e^{-t}}(0) - \xi (Q-\gamma)t}\,dt \leq C e^{\xi h_{e^{-k_1}}(z) - \xi (Q-\gamma) k_1} \right\} &= 1 - \Prob\left\{Z_\frac{2(Q-\gamma)}{\xi^2} > \frac{\xi^2}{2}C\right\} \\&= 1 - O\left(e^{-uC }\right) \end{aligned}
    \end{equation}
    as $C \to \infty$ with the rate uniform on $z$.
    
    Let $u>0$ be a constant to be determined later. Now take a union bound over the events~\eqref{eqn:exit-time-circle-average-z} and~\eqref{eqn:exit-time-circle-average-0} for all integers $k \in [k_0, k_0+uN\log C)$ as well as the events in~\eqref{eqn:exit-time-circle-average-microscopic-z} and \eqref{eqn:exit-time-circle-average-microscopic-0}. Recalling~\eqref{eqn:use-circle-avg-int} and~\eqref{eqn:exit-time-circle-average-integral-macroscopic}, we find that with superpolynomially high probability as $C \to \infty$, at a rate uniform in $z$,
    \begin{equation}\label{eqn:exit-time-circle-average-riemann-sum} \begin{aligned} D_{h^\gamma}^\ud(0,z) &\leq C \int_{-\log\frac{|z|}{2}}^\infty \left( e^{\xi h_{e^{-t}}(0) - \xi (Q-\gamma) t} + e^{\xi h_{e^{-t}}(z) - \xi Qt} \right) dt  \\&\leq  C^2 \sum_{k=k_0}^{k_0+uN\log C} \left( e^{\xi h_{e^{-k}}(0) - \xi (Q-\gamma) k} + e^{\xi h_{e^{-k}}(z) - \xi Qk} \right). \end{aligned} \end{equation}
    Given~\eqref{eqn:exit-time-circle-average-riemann-sum}, there must exist an integer $k \in [k_0,k_0+uN\log C]$ satisfying either
    \begin{equation}\label{eqn:large-circle-avg-prelimiary}
        e^{\gamma h_{e^{-k}}(0)} \geq C^{-2d_\gamma} e^{\gamma (Q-\gamma) k} \left(\frac{D_{h^\gamma}^\ud(0,z)}{N\log C}\right)^{d_\gamma} \;\; \text{or} \;\;\; e^{\gamma h_{e^{-k}}(z)} \geq C^{-2d_\gamma} e^{\gamma Q k} \left(\frac{D_{h^\gamma}^\ud(0,z)}{N\log C}\right)^{d_\gamma}.
    \end{equation}
    
    We now let $u = 3d_\gamma$ and replace $C$ in \eqref{eqn:large-circle-avg-prelimiary} by $C^{1/u}$. For sufficiently large $C$, since $N>0$ is a fixed constant, the right hand sides of \eqref{eqn:large-circle-avg-prelimiary} are bounded below by $C^{-1} e^{\gamma (Q-\gamma) k} (D_{h^\gamma}^{\ud}(0,z))^{d_\gamma}$ and $C^{-1} e^{\gamma Q k} (D_{h^\gamma}^{\ud}(0,z))^{d_\gamma}$, respectively. Hence, the probability that there exists an integer $k\in [k_0, k_0+N\log C]$ satisfying \eqref{eqn:large-circle-avg} is bounded below by $1-C^{\rho /u}$ as $C\to \infty$ for every $\rho>0$ as claimed.
\end{proof}

The final two lemmas are standard estimates on the circle averages of GFF and the LQG measure. In Step~\ref{exit-time-proof:step-5}, we obtain a lower bound on the circle average $h_r(w)$ on the boundary of the Euclidean ball $B_r(w)$ found in Step~\ref{exit-time-proof:step-4}. This is done by comparing $h_r(w)$ with either $h_{e^{-k}}(0)$ (Case~1) or $h_{e^{-k}}(z)$ (Case~2) using the following lemma.

\begin{lem}\label{lem:circle-avg-compare}
    Let $h$ be a whole-plane GFF normalized so that $h_1(0) = 0$. Let $\zeta \in (0,1)$ and $q>2$ be given constants. Let $z \in B_{2e^{-2}}(0)$ and denote $k_0 = \lfloor -\log(|z|/2)\rfloor$. For all $C > C_0(q, \zeta)$, the following event is true with probability $1 - O(e^{-(q^2/2 -2)\zeta k_0})$, where the rate is otherwise uniform over all choices of $z$. For every integer $k \geq k_0$, writing $r = e^{-k(1+\zeta)}$:
    \begin{enumerate}[(i)]
        \item\label{item:exit-time-circle-avg-comparison-0} For every $w \in r\Z^2$ such that $B_{r}(w) \subset B_{e^{-k}}(0) \setminus \overline{B_{e^{-k-1}}(0)}$, we have \begin{equation}|h_r(w) - h_{e^{-k}}(0)| \leq k q \zeta.\end{equation}
        \item\label{item:exit-time-circle-avg-comparison-z} For every $w \in r\Z^2$ such that $B_{r}(w) \subset B_{e^{-k}}(z) \setminus \overline{B_{e^{-k-1}}(z)}$, we have \begin{equation}|h_r(w) - h_{e^{-k}}(z)| \leq k q \zeta.\end{equation}
    \end{enumerate}
\end{lem}
\begin{proof} 
    \cite[Lemma 4.5]{afs-metric-ball} proves that the event (i) holds for every $k\geq k_0$ with probability
    \begin{equation}\label{eqn:exit-time-circle-average-comparison-probability}
        1 - O\left(\sum_{k=k_0}^\infty e^{k\zeta(2-q^2/2)}\right) = 1-O(e^{-(q^2/2-2)\zeta k_0}).
    \end{equation}
    The random variable $ \left| h_r(w) - h_{e^{-k}}(z) \right|$ depends on $h$ viewed modulo additive constant. By the translation invariance~\eqref{eqn:gff-translate} of the law of the GFF modulo additive constant, it follows that (ii) also holds with the same probability computed in \eqref{eqn:exit-time-circle-average-comparison-probability} regardless of the choice of $z$.
\end{proof}

In Step~\ref{exit-time-proof:step-6}, we give a lower bound on the LQG measure $\mu_h(B_r(w))$ in terms of the circle average $h_r(w)$.

\begin{lem}\label{lem:lqg-mass-from-circle-avg}
    Let $h$ be a whole-plane GFF normalized so that $h_1(0) = 0$. For each fixed $w \in \C$ and $r>0$, it holds with superpolynomially high probability as $C\to \infty$, with the rate uniform over $w$ and $r$, that
    \begin{equation}
        \mu_h(B_r(w)) \geq C^{-1} r^{\gamma Q} e^{\gamma h_r(w)}.
    \end{equation}
\end{lem}

\begin{proof}
The case when $w=0$ and $r=1$ is a standard estimate for the LQG measure. See, e.g., \cite[Lemma 4.6]{shef-kpz}  or~\cite[Theorem 2.12]{rhodes-vargas-review}.
The case of a general $w\in\C$ and $r > 0$ follows from the case when $w=0$ and $r=1$ since $r^{-\gamma Q} e^{-\gamma h_r(w)} \mu_h(B_r(w)) \overset{d}{=} \mu_h(B_1(0))$ due to the scale and translation invariance of the law of $h$, \eqref{eqn:gff-translate}, and the LQG coordinate change formula \eqref{eqn:lqg-measure-coordinate-change}.
\end{proof}

\subsection{Proof of tightness}

From Proposition~\ref{prop:exit-time}, we can not only deduce that \eqref{eqn:tight-family-mc-approx} is tight, but also that it is uniformly bounded in $L^p$ for every $p\geq 1$. This stronger result is necessary later, specifically for Proposition~\ref{prop:minkowski-content-bm}. Recall the notation $N_\ep(A) = N_\ep(A;D_{h^\gamma})$ for the number of $D_{h^\gamma}$-balls of radius $\ep$ needed to cover the set $A$.

\begin{prop}\label{prop:covering-num-moment-bound}
    For every $p\geq 1$, there exists a constant $C_p < \infty$ such that for every $s<t$,
    \begin{equation*}
        \sup_{0 < \ep < |t-s|^{1/d_\gamma}} \Exp[(\ep^{d_\gamma} N_\ep(\eta[s,t]))^p] \leq C_p |t-s|^p.
    \end{equation*}
    In particular, for each fixed $s<t$, the random variables $\ep^{d_\gamma} N_\ep(\eta[s,t])$ for $\ep \in (0,1)$ are tight.
\end{prop}

Proposition~\ref{prop:covering-num-moment-bound} is a consequence of the following lemma together with a scaling argument.

\begin{lem} \label{lem:radius1-bound}
For each $p\geq 1$, the number of LQG balls of radius 1 needed to cover $\eta[0,1]$ satisfies
\begin{equation}
\Exp[(N_1(\eta[0,1]))^p] < \infty .
\end{equation} 
\end{lem}

\begin{proof}
Let $K$ be a positive integer. For each integer $0 \leq k < K$, define the event
    \begin{equation} 
        E_{k,K} = \left\{ \eta \Big[ \frac{k}{K}, \frac{k+1}{K} \Big] \subset \mathcal B_1\big(\eta(k/K);D_{h^\gamma}\big) \right\}. 
    \end{equation}
    Clearly,
    \begin{equation}  
        \bigcap_{k=0}^{K-1} E_{k,K} \subset \left\{ N_1(\eta[0,1]) \leq K\right\}.  
    \end{equation}
    By Proposition~\ref{prop:qc-mmspace-invariance},
    \begin{equation}  
        \Prob(E_{k,K}) = \Prob (E_{0,K}) = \Prob\{\tau_1 > 1/K\} \quad \text{ for every } k . 
    \end{equation}
    Then
    \begin{equation} 
        \Prob \{ N_1(\eta[0,1]) > K \} \leq \sum_{k=0}^{K-1} \Prob(E_{k,K}^\complement) = K \cdot \Prob \{\tau_1 \leq 1/K\} .
    \end{equation}
    By Proposition~\ref{prop:exit-time}, it follows that $\Prob\{N_1(\eta[0,1])>K\}$ decays faster than any negative power of $K$ as $K\rightarrow \infty$. Therefore,
    \begin{equation}  
    \Exp[(N_1(\eta[0,1]))^p] \leq \sum_{K=1}^\infty ((K+1)^p - K^p) \cdot \Prob \{N_1(\eta[0,1]) \geq K\} <\infty. 
    \end{equation}
\end{proof}

\begin{proof}[Proof of Proposition~\ref{prop:covering-num-moment-bound}]
    It suffices to find a constant $C_p<\infty $ such that $\sup_{0 < \ep < 1} \Exp[(\ep^{d_\gamma} N_\ep(\eta[0,1])^p] \leq C_p$. Indeed, by Corollary~\ref{cor:covering-number-relation},
    \begin{equation} 
        \sup_{0 < \ep < |t-s|^{1/d_\gamma}} \Exp[(\ep^{d_\gamma} N_\ep(\eta[s,t]))^p] \stackrel{d}{=} |t-s|^p \left( \sup_{0 < u < 1} \Exp[(u^{d_\gamma} N_u(\eta[0,1])^p] \right).  
    \end{equation}
    We thus restrict our attention to uniform $L^p$-bounds for $\{\ep^{d_\gamma} N_\ep(\eta[0,1])\}_{0 < \ep < 1}$.
    
    For each $\ep \in (0,1)$,
    \begin{equation} 
        N_\ep(\eta[0,1]) \leq \sum_{k=0}^{\lfloor \ep^{-d_\gamma} \rfloor } N_\ep(\eta [k\ep^{d_\gamma}, (k+1)\ep^{d_\gamma} ] ). 
    \end{equation}
    By Corollary~\ref{cor:covering-number-relation}, for every integer $k$,
    \begin{equation} 
        N_1(\eta[0,1]) \stackrel{d}{=} N_\ep(\eta[k\ep^{d_\gamma}, (k+1)\ep^{d_\gamma}]). 
    \end{equation}
    Hence, by Jensen's inequality,
    \begin{equation} \label{eqn:pth-moment-jensen}
        \begin{aligned} \Exp\left[(N_\ep(\eta[0,1]))^p\right] &\leq (\ep^{-d_\gamma}+1)^{p-1} \sum_{k=0}^{\lfloor \ep^{-d_\gamma} \rfloor } \Exp \left[(N_\ep(\eta [k\ep^{d_\gamma}, (k+1)\ep^{d_\gamma} ] ))^p\right] \\ &\leq (\ep^{-d_\gamma}+1)^p \Exp[(N_1(\eta[0,1]))^p].  \end{aligned}
    \end{equation}
    By Lemma~\ref{lem:radius1-bound}, $\Exp[(N_1(\eta[0,1]))^p]$ is a finite constant depending only on $p$ and $\gamma$. Multiplying both sides of~\eqref{eqn:pth-moment-jensen} by $\ep^{d_\gamma p}$, we obtain $\Exp[(\ep^{d_\gamma}N_\ep(\eta[0,1]))^p] \leq 2\Exp[(N_1(\eta[0,1]))^p]=:C_p$. 
\end{proof}

\subsection{Further LQG metric estimates for space-filling SLE}

We record two consequences of Proposition~\ref{prop:exit-time}, which we do not use elsewhere in this paper but are of independent interest.
We first show that for each $s<t$, the $D_{h^\gamma}$-diameter of $\eta[s,t]$, which we denote by $\mathrm{diam}(\eta[s,t];D_{h^\gamma})$, has finite positive moments of all orders. The other result is the first half of Theorem~\ref{thm:sle-holder}, that for every exponent less than $1/d_\gamma$, the space-filling SLE$_{\kappa'}$ curve $\eta$ parameterized by $\mu_{h^\gamma}$ is almost surely locally H\"older continuous with respect to $D_{h^\gamma}$.

For completeness, we also prove that all negative moments of $\mathrm{diam}(\eta[s,t];D_{h^\gamma})$ are finite and that $\eta$ is almost surely not locally H\"older continuous for exponents greater than $1/d_\gamma$. These complementary results follow from the following extension of Theorem~\ref{thm:metric-ball-volume} to the $\gamma$-quantum cone. 

\begin{lem}\label{lem:qc-ball-volume}
    Let $h^\gamma$ be the field of a $\gamma$-quantum cone under the circle average embedding. For $\zeta \in (0,1)$ and $k\geq 1$, there exists a constant $C_{k,\zeta}$ such that for all $r\in (0,1)$,
    \begin{equation}\label{eqn:qc-ball-vol-moments}
        \Exp[(\mu_{h^\gamma}(\ud \cap \mathcal B_r(0;D_{h^\gamma})))^k] \leq C_{k,\zeta} r^{kd_\gamma - \zeta}.
    \end{equation}
    Moreover, for any compact set $K \subset \ud$ and $\zeta>0$, we almost surely have that 
    \begin{equation}\label{eqn:qc-ball-vol-sup}
        \sup_{r\in(0,1)} \sup_{z\in K} \frac{\mu_{h^\gamma}(\ud \cap \mathcal B_r(z;D_{h^\gamma}))}{r^{d_\gamma - \zeta}}<\infty.
    \end{equation}
\end{lem}
\begin{proof}
Recall that if $h$ is a whole-plane GFF with $h_1(0) = 0$, then $(h - \gamma \log|\cdot|)|_\ud$ agrees in law with $h^\gamma|_\ud$. Hence, \eqref{eqn:qc-ball-vol-moments} follows from the proofs of \cite[Lemmas 3.16 and 3.18]{afs-metric-ball}. The proof of \eqref{eqn:qc-ball-vol-sup} from \eqref{eqn:qc-ball-vol-moments} is analogous to that of \eqref{eqn:ball-volume-asymptotics} in \cite{afs-metric-ball}.
\end{proof}

We now show that the $D_{h^\gamma}$-diameter of $\eta[s,t]$ has finite moments of all orders.
\begin{prop} \label{prop:sle-diameter}
    Let $s<t$. For each $p\in \R$, there exists a constant $C_p>0$ only depending on $p$ such that 
    \begin{equation}
        \Exp[(\mathrm{diam}(\eta[s,t];D_{h^\gamma}))^p] \leq C_p |t-s|^{p/d_\gamma}.
    \end{equation}
\end{prop}
\begin{proof}
    By Proposition~\ref{prop:qc-mmspace-invariance}, for each $s<t$,
    \begin{equation}
        \mathrm{diam}(\eta[s,t];D_{h^\gamma}) \stackrel{d}{=} \mathrm{diam}(\eta[0,t-s];D_{h^\gamma}) \stackrel{d}{=} (t-s)^{1/d_\gamma} \mathrm{diam}(\eta[0,1];D_{h^\gamma}).
    \end{equation}
    Hence, it suffices to show that $\Exp[ (\mathrm{diam}(\eta[0,1];D_{h^\gamma}))^p ]<\infty$ for all $p$.

    Suppose $p>0$. For $r>0$, let $\tau_r$ be as in~\eqref{eqn:sle-dist-time}. Note that for $r>0$, if $\tau_r > 1$ then $\eta[0,1] \subset \mathcal B_r(0;D_{h^\gamma})$ and hence $\mathrm{diam}(\eta[0,1];D_{h^\gamma}) < 2r$. By \eqref{eqn:exit-time-scaling},
    \begin{equation}
        \begin{aligned} \Exp[ ((\mathrm{diam}(\eta[0,1];D_{h^\gamma}))^p ] &= \int_0^\infty \Prob\{ \mathrm{diam}(\eta[0,1];D_{h^\gamma}) \geq x^{1/p}\} \,dx \\
        &\leq \int_0^\infty \Prob \{ \tau_{(x^{1/p})/2} \leq 1\}\,dx = \int_0^\infty \Prob\{\tau_1 \leq 2^{-d_\gamma} x^{d_\gamma/p}\}\,dx 
        \end{aligned}
    \end{equation}
    This integral is finite since by Proposition~\ref{prop:exit-time}, $\Prob\{\tau_1 \leq 2^{-d_\gamma} x^{d_\gamma/p}\}$ decays superpolynomially as $x\to 0$.
    
    Now suppose $p<0$. If $\mathrm{diam}(\eta[0,1];D_{h^\gamma}) < r$, then $\mu_{h^\gamma}(\mathcal B_r(0;D_{h^\gamma})) \geq 1$. Then, by Proposition~\ref{prop:qc-mmspace-invariance},
    \begin{equation}\label{eqn:diameter-negative-exp}
        \begin{aligned}
        &\Exp[ ((\mathrm{diam}(\eta[0,1];D_{h^\gamma}))^p ] = \int_0^\infty \Prob\{((\mathrm{diam}(\eta[0,1];D_{h^\gamma}))^p>x\}\,dx  \\&= \int_0^\infty \Prob\{ \mathrm{diam}(\eta[0,1];D_{h^\gamma})< x^{1/p}\} \,dx 
        \leq \int_0^\infty \Prob \{ \mu_{h^\gamma}(\mathcal B_{x^{1/p}}(0;D_{h^\gamma})) \geq 1 \}\,dx  \\
        &\leq 1 + \int_1^\infty \Prob\{ \mu_{h^\gamma}(\ud \cap \mathcal B_{x^{1/p}}(0;D_{h^\gamma})) \geq 1\}\,dx + \int_1^\infty \Prob\{\mathcal B_{x^{1/p}}(0;D_{h^\gamma}) \not \subset \ud\}\,dx.
        \end{aligned}
    \end{equation}
    From \eqref{eqn:qc-ball-vol-moments} and Lemma~\ref{lem:lqg-euclidean-balls-mutual-inclusion} (with $r=1/2$, say, $\ep  = x^{1/p}$, and $k$ in place of $p$), for each $k\geq 1$ and each $\zeta > 0$, there are constants $C_{k,\zeta}  , C_k > 0$ such that for all $x>1$,
    \begin{equation}
        \Prob\{ \mu_{h^\gamma}(\ud \cap \mathcal B_{x^{1/p}}(0;D_{h^\gamma})) \geq 1\} \leq C_{k,\zeta} x^{(kd_\gamma - \zeta)/p} \quad \text{and} \quad 
        \Prob\{\mathcal B_{x^{1/p}}(0;D_{h^\gamma}) \not\subset \ud \}  \leq C_k x^{k/p}.
    \end{equation}
    We conclude that the right-hand side of \eqref{eqn:diameter-negative-exp} is finite by choosing a sufficiently large $k$. 
\end{proof}

Proposition~\ref{prop:sle-diameter} gives a quick proof of Theorem~\ref{thm:sle-holder}, on the H\"older continuity of the whole-plane space-filling SLE$_{\kappa'}$ curve $\eta$ with respect to the LQG metric $D_{h^\gamma}$.

\begin{proof}[Proof of Theorem~\ref{thm:sle-holder}]
    Let us first show that almost surely, $\eta$ is locally H\"older continuous for any exponent less than $1/d_\gamma$. We proved in Proposition~\ref{prop:sle-diameter} that for each $p> 0$, there exists a constant $C_p>0$ such that for all $s<t$,
    \begin{equation}
        \Exp[(D_{h^\gamma}(\eta(s),\eta(t)))^p] \leq \Exp[(\mathrm{diam}(\eta[s,t];D_{h^\gamma}))^p] \leq C_p |t-s|^{p/d_\gamma}.
    \end{equation}
    The claim now follows from the Kolmogorov continuity theorem as we let $p\to \infty$.\footnote{The Kolmogorov continuity theorem is usually stated for a stochastic process $X_t$ taking values on a fixed metric space $(S,d)$. However, the H\"older continuity part of the theorem merely requires that the real-valued random variables $d(X_s,X_t)$ are measurable and have appropriate uniform moments.}

    On the other hand, for every $\zeta\in (0,d_\gamma)$, Lemma~\ref{lem:qc-ball-volume} (plus the fact that $\eta$ is parameterized by $\mu_{h^\gamma}$-mass) implies that there a.s.\ exists a random $c_\zeta>0$ such that $\mathrm{diam}(\eta[s,t];D_{h^\gamma}) > c_\zeta (t-s)^{1/(d_\gamma - \zeta)}$ for all $s<t$ such that $\eta[s,t]\subset \ud$. In particular, $\eta$ is almost surely not H\"older continuous with exponent greater than $1/d_\gamma$ in any neighborhood of 0. By Proposition~\ref{prop:qc-mmspace-invariance}, $\eta$ is almost surely not H\"older continuous with any exponent greater than $1/d_\gamma$ in any neighborhood of $t \in \Q$, and therefore not in any bounded open interval.
\end{proof}

\section{Minkowski content of space-filling SLE segments}\label{sec:minkowski-content-sle-segments} 

\begin{rem*}
    We assume $\kappa' = 16/\gamma^2$ throughout this section, as the results in this section rely on Proposition~\ref{prop:left-right-independent-mmspace}. Recall the shorthand
    $N_\ep(\eta[s,t]) = N_\ep(\eta[s,t];D_{h^\gamma})$.
\end{rem*}

We have shown in Proposition~\ref{prop:covering-num-moment-bound} that for each $s<t$, the random variables $\ep^{d_\gamma} N_\ep(\eta[s,t])$ admit subsequential limits in law as $\ep\rightarrow 0$. We need to rule out further the possibility that the subsequential limit is zero, which is the purpose of Proposition~\ref{prop:covering-number-lower-bound}. Unlike in many results regarding fractal dimensions, where the lower bound is more difficult to prove than the upper bound, this proposition has a much shorter proof than Proposition~\ref{prop:covering-num-moment-bound}. This is thanks to the independence of the metric measure space structure on disjoint space-filling SLE segments, which comes from Proposition~\ref{prop:left-right-independent-mmspace}.

\begin{prop}\label{prop:covering-number-lower-bound}
    There exists a deterministic constant $c = c(\gamma)>0$ such that 
    \begin{equation}\label{eqn:covering-number-lower-bound}
        \lim_{\ep \to 0} \Prob\{\ep^{d_\gamma} N_\ep(\eta[0,1]) > c\} = 1.
    \end{equation}
\end{prop}
\begin{proof}
    For $r>0$, let $p_r$ be the probability that $\eta[0,1]$ contains a $D_{h^\gamma}$-ball of radius $r$. By Proposition~\ref{prop:SLE-includes-Euclidean-ball}, $\eta[0,1]$ almost surely contains a Euclidean ball, which in turn contains a $D_h$-ball since $D_h$ is a continuous metric. Hence, $p_r\to 1$ as $r\to 0$.  
    
    Fix $r>0$ such that $p_r>0$. For this $r$, define for each positive integer $n$ and integer $k\in[1,n]$ the event
    \begin{equation}
        E_{n,k}:= \left\{\eta\left[\frac{k-1}{n},\frac{k}{n}\right] \text{ contains a } D_{h^\gamma}\text{-ball of radius } r n^{-1/d_\gamma}\right\}. 
    \end{equation}
    By Proposition~\ref{prop:qc-mmspace-invariance}, $\Prob(E_{n,k}) = p_r$ for every $n$ and $k$. Furthermore, $E_{n,1},\dots,E_{n,n}$ are independent by Proposition~\ref{prop:left-right-independent-mmspace}. To see this, for fixed $t\in \R$, define $U_{t-}$ and $U_{t+}$ to be the interiors of $\eta(-\infty,t]$ and $\eta[t,\infty)$, respectively. Then, $E_{n,k}$ is almost surely determined by $(U_{t-}, D_{h^\gamma}^{U_{t-}}, \mu_{h^\gamma}|_{U_{t_-}}, \eta|_{(-\infty,t]})$ if $t \geq \frac{k}{n}$ and by $(U_{t+}, D_{h^\gamma}^{U_{t+}}, \mu_{h^\gamma}|_{U_{t_+}}, \eta|_{[t,\infty)})$ if $t \leq \frac{k-1}{n}$. These two curve-decorated metric measure spaces are independent by Proposition~\ref{prop:left-right-independent-mmspace} combined with the translation invariance of Proposition~\ref{prop:qc-mmspace-invariance}. Choosing $t = k/n$, we see that the random vectors $(\mathbf{1}_{E_{n,1}}, \dots,\mathbf{1}_{E_{n,k}})$ and $(\mathbf{1}_{E_{n,k+1}},\dots,\mathbf{1}_{E_{n,n}})$ are independent. Since this is true for every $k$, we conclude that $\mathbf{1}_{E_{n,1}},\dots,\mathbf{1}_{E_{n,n}}$ are i.i.d.\ Bernoulli random variables with success probability $p_r$.
    
    We now argue that for every $n$,
    \begin{equation}\label{eqn:covering-number-lower-bound-disjoint-balls}
        N_{rn^{-1/d_\gamma}}(\eta[0,1]) \geq \sum_{k=1}^n \mathbf{1}_{E_{n,k}}.
    \end{equation}
    Indeed, if $\sum_{k=1}^n \mathbf{1}_{E_{n,k}} = m$, then there are distinct points $z_1,\dots,z_m \in \eta[0,1]$ such that $D_{h^\gamma}(z_j,z_k) \geq 2rn^{-1/d_\gamma}$ for every $j\neq k$. Such $z_j$ and $z_k$ cannot be within a single $D_{h^\gamma}$-ball of radius $rn^{-1/d_\gamma}$, so $N_{rn^{-1/d_\gamma}}(\eta[0,1]) \geq m$. Hence, \eqref{eqn:covering-number-lower-bound-disjoint-balls} holds. By the law of large numbers,
    \begin{equation}\label{eqn:lower-bd-lim}
        \lim_{n\to \infty} \Prob\left\{ n^{-1} N_{rn^{-1/d_\gamma}}(\eta[0,1]) \geq p_r\right\} = 1.
    \end{equation} 
    
    Let $\ep \in (0,1)$ and set $n_\ep := \lfloor \ep^{-d_\gamma} \rfloor$. Since $N_{r\ep}(\eta[0,1])$ increases as $\ep$ decreases to 0,
    \begin{equation} \label{eqn:lower-bd-scale}
	   (r\ep)^{d_\gamma} N_{r\ep}(\eta[0,1]) \geq r^{d_\gamma} (n_\ep \ep^{d_\gamma}) \frac{N_{r n_\ep^{-1/d_\gamma}}(\eta[0,1])}{n_\ep} . 
    \end{equation} 
	Since $d_\gamma$ is positive, we have $n_\ep \ep^{d_\gamma} \to 1$ as $\ep\to 0$. By combining this with~\eqref{eqn:lower-bd-lim} and \eqref{eqn:lower-bd-scale}, we get
	\[ \lim_{\ep \to 0} \mathbb P\{ (r \ep)^{d_\gamma} N_{r \ep}(\eta[0,1]) \geq r^{d_\gamma} p_r \} = 1. \]
	Since $r$ is a constant which depends only on $\gamma$, this implies \eqref{eqn:covering-number-lower-bound} with $c = r^{d_\gamma} p_r$. 
\end{proof}

Another possibility that we have yet to rule out is that different subsequential limits of $\ep^{d_\gamma} N_\ep(\eta[s,t])$ may take different values. To this end, we replace the rescaling coefficients from $\ep^{-d_\gamma}$ to
\begin{equation}
    \mathfrak b_\ep := \Exp[N_\ep(\eta[0,1])]
\end{equation}
as introduced in \eqref{eqn:scaling-constant}, so that any subsequential limit of $\mathfrak b_\ep^{-1} N_\ep(\eta[0,1])$ has mean 1. Indeed, $\{\mathfrak b_\ep\}_{\ep>0}$ is a sequence of $d_\gamma$-dimensional rescaling coefficients (recall Definition~\ref{def:scaling-coefficients}).

\begin{cor} \label{cor:normalization-constant-comparison}
    There exists a deterministic constant $C>0$ such that for every $\ep \in (0,1)$,
    \[ C^{-1} \ep^{-d_\gamma} < \mathfrak b_\ep < C \ep^{-d_\gamma}. \]
\end{cor}

The upper bound follows immediately from Proposition~\ref{prop:covering-num-moment-bound} and the lower bound from Proposition~\ref{prop:covering-number-lower-bound}. Note that this corollary is the first part of Proposition~\ref{prop:scaling-constant-properties}.

The following proposition is the main result of this section.

\begin{prop}\label{prop:main-thm-qc-sle}
    Let $\kappa' = 16/\gamma^2$. Then, for each fixed $s<t$, 
    \begin{equation}\label{eqn:sle-minkowski-content}
        \lim_{\ep \to 0} \mathfrak b_\ep^{-1} N_\ep(\eta[s,t]) = t-s \quad \text{in probability. }
    \end{equation}
\end{prop}

Here is an overview of the proof of Proposition~\ref{prop:main-thm-qc-sle}. Let 
\begin{equation}
    \mathcal I_\Q = \{[s,t]: s,t\in \Q, s<t\}
\end{equation}
be the collection of all closed intervals with rational endpoints. Suppose we are given an arbitrary sequence of $\ep$-values decreasing to 0.

\begin{enumerate} 

\item By Proposition~\ref{prop:covering-num-moment-bound}, we can find a subsequence $\ep_n$ for which the sequence of $\R^{\mathcal I_\Q}$-valued random variables $(\mathfrak b_{\ep_n}^{-1} N_{\ep_n}(\eta(I)): I \in \mathcal I_\Q)$ converges in distribution with respect to the product topology on $\R^{\mathcal I_\Q}$. Denote this subsequential weak limit by $(X_I: I \in \mathcal I_\Q)$.

\item We show in Proposition~\ref{prop:finite-additivity} that $X_I$ is finitely additive: i.e., for each rational $r<s<t$, $X_{[r,s]} + X_{[s,t]} = X_{[r,t]}$ almost surely. We are thus able to construct the ``Minkowski content process" $\{Y_t\}_{t\in \Q}$ where $Y_t = X_{[0,t]}$ (resp.\ $-X_{[t,0]})$ for $t\geq 0$ (resp.\ $t<0$), so that $X_{[s,t]} = Y_t - Y_s$ for every $[s,t]\in \mathcal I_\Q$.

\item We show in Proposition~\ref{prop:minkowski-content-bm} that $\{Y_t\}_{t\in \Q}$ can be extended to a continuous process on $\R$ with independent and stationary increments: i.e., a Brownian motion with drift. 

\item From Proposition~\ref{prop:covering-number-lower-bound}, almost surely, $Y_s < Y_t$ for all rational $s<t$. Hence, $Y_t = a t$ for some deterministic constant $a>0$. In fact, $a=1$ because we chose $\mathfrak b_\ep$ so that $\Exp[\mathfrak b_\ep^{-1} N_\ep(\eta[0,1])] = 1$ for every $\ep>0$.

\item In conclusion, we have the following convergence in distribution:
\[ \lim_{n\to \infty} \mathfrak b_{\ep_n}^{-1} N_{\ep_n}(\eta[s,t]) = t-s \quad \text{ for all } [s,t] \in \mathcal I_\Q  \]
Since $t-s$ is a deterministic constant, the convergence holds in probability. We started with an arbitrary sequence of $\ep$ decreasing to 0, so \eqref{eqn:sle-minkowski-content} holds for every rational $s<t$. This extends to all real $s<t$ by the following simple observation: if $[s_1,t_1] \subset [s_2,t_2]$, then $N_\ep(\eta[s_1,t_1]) \leq N_\ep(\eta[s_2,t_2])$ for all $\ep>0$.
\end{enumerate}

\subsection{Finite additivity}\label{sec:additivity}

Given a closed and bounded set $A \subset \C$, for $\ep>0$, define
\begin{equation}
    \bd_\ep A := \{z \in A: D_{h^\gamma}(z, \bd A) < \ep\}
\end{equation}
to be the the intersection of $A$ with the $\ep$-neighborhood of $\bd A$ in the LQG metric. 
Denote  
\begin{equation}
    N_\ep^\circ(A) := N_\ep(A \setminus \bd_\ep A).
\end{equation}
Note that every ball counted in $N_\ep^\circ(A)$ is centered at a point in the interior of $A$. Also, 
\begin{equation}
    N_\ep^\circ(A) \leq N_\ep(A) \leq N_\ep^\circ(A) + N_\ep(\bd_\ep A).
\end{equation}
The following lemma is the main technical input in the proof of finite additivity.

\begin{lem}\label{lem:sle-nbhd-zero-content}
    For each fixed $s<t$,
    \begin{equation}\label{eqn:sle-boundary-blowup-covering}
        \lim_{\ep \to 0} \ep^{d_\gamma} N_\ep(\bd_\ep\eta[s,t]) = 0 \quad \text{almost surely.}
    \end{equation}
    Consequently,
    \begin{equation}
        \lim_{\ep \to 0} |\ep^{d_\gamma} N_\ep(\eta[s,t]) - \ep^{d_\gamma} N_\ep^\circ(\eta[s,t])| = 0 \quad \text{almost surely.}
    \end{equation}
\end{lem}

We first explain how Lemma~\ref{lem:sle-nbhd-zero-content} implies that the Minkowski content is additive over finitely many disjoint space-filling SLE segments, and then prove the lemma in the rest of the subsection. We need the following classic result on weak convergence in the proof of finite additivity.
\begin{lem}[{\cite[Theorems 2.7 and 3.1]{billingsley-weak-convergence}}] \label{lem:diag-weak-convergence}
    Suppose $(S,d)$ is a metric space and $(Y_n, Z_n)$ is a sequence of $S\times S$-valued Borel-measurable random variables. If $Y_n \stackrel{d}{\longrightarrow} Y$ and $d(Y_n, Z_n) \stackrel{p}{\longrightarrow} 0$ as $n\to \infty$, then $(Y_n,Z_n) \stackrel{d}{\longrightarrow} (Y,Y)$.
\end{lem}

\begin{prop}[Finite additivity]\label{prop:finite-additivity}
    Suppose that for some sequence $\{\ep_n\}_{n\in\N}$ of positive numbers decreasing to zero, $(\mathfrak b_{\ep_n}^{-1} N_{\ep_n}(\eta(I)): I \in \mathcal I_\Q)$ converges weakly to $(X_I: I \in \mathcal I_\Q)$ as $n\to \infty$. Then, for every rational $r<s<t$, we have $X_{[r,s]} + X_{[s,t]} = X_{[r,t]}$ almost surely.
\end{prop}
\begin{proof}
    Assume that Lemma~\ref{lem:sle-nbhd-zero-content} holds. Then, by Corollary~\ref{cor:normalization-constant-comparison}, we have 
    \begin{equation}
        \lim_{n\to \infty} |\mathfrak b_{\ep_n}^{-1} N_{\ep_n}(\eta(I)) - \mathfrak b_{\ep_n}^{-1} N_{\ep_n}^\circ(\eta(I))|=0 \quad \text{almost surely\ for all } I \in \mathcal I_\Q.
    \end{equation}
    Combining this with Lemma~\ref{lem:diag-weak-convergence}, we obtain
    \begin{equation} \label{eqn:covering-num-joint-convergence}
        ((\mathfrak b_{\ep_n}^{-1} N_{\ep_n}(\eta(I)), \mathfrak b_{\ep_n}^{-1} N_{\ep_n}^\circ(\eta(I))): I \in \mathcal I_\Q) \stackrel{d}{\longrightarrow} ((X_I, X_I): I \in \mathcal I_\Q) \quad \text{as } n\to \infty.
    \end{equation}
    
    Let $r<s<t$ be any triple of rationals. Fix $\ep>0$ for now. Note that every metric ball counted in $N_\ep^\circ(\eta[r,t])$ has its center in $\eta[r,t]$. The idea is to classify these balls based on whether their centers are in $\eta[r,s]$ or $\eta[s,t]$. Let $m = N_\ep^\circ(\eta[r,t])$ and let $z_1,\dots,z_k \in \eta[r,s]$, $z_{k+1},\dots,z_m \in \eta[s,t]$ be any collection of points such that
    $\bigcup_{j=1}^m \mathcal B_\ep(z_j; D_{h^\gamma})$ covers $\eta[r,t]\setminus \bd_\ep \eta[r,t]$. If $j\leq k$, then $\mathcal B_\ep(z_j; D_{h^\gamma}) \cap (\eta[s,t] \setminus \bd_\ep \eta[s,t]) = \varnothing$. Hence, $\bigcup_{j=k+1}^m B_\ep(z_j; D_{h^\gamma})$ covers $\eta[s,t] \setminus \bd_\ep \eta[s,t]$ and $N_\ep^\circ(\eta[s,t]) \leq m-k$. Similarly, $N_\ep^\circ(\eta[r,s]) \leq k$. Therefore,
    \begin{equation}
        N_\ep^\circ(\eta[r,s]) + N_\ep^\circ(\eta[s,t]) \leq N_\ep^\circ(\eta[r,t]) \leq N_\ep(\eta[r,t]) \leq N_\ep(\eta[r,s]) + N_\ep(\eta[s,t]).
    \end{equation}
    Since this inequality holds for every $\ep>0$, we have from \eqref{eqn:covering-num-joint-convergence} that $X_{[r,t]} = X_{[r,s]} + X_{[s,t]}$ almost surely.
\end{proof}

We now begin the proof of Lemma~\ref{lem:sle-nbhd-zero-content}. The first step in the proof of \eqref{eqn:sle-boundary-blowup-covering} is to show that the Minkowski dimension of $\partial \eta[s,t]$ is strictly less than $d_\gamma$. In fact, we prove that it is at most $d_\gamma/2$. 

\begin{lem}\label{lem:sle-dim-1/2}
    For each fixed $s<t$ and $\delta>d_\gamma/2$,
    \begin{equation}\lim_{\ep \to 0} \ep^{\delta} N_\ep(\partial \eta[s,t]) = 0 \quad \text{almost surely.}\end{equation}
\end{lem}
\begin{proof}
    By Proposition~\ref{prop:qc-mmspace-invariance}, for each $s<t$ and $\ep>0$, 
    \begin{equation} 
        N_\ep(\partial \eta[s,t]) \stackrel{d}{=} N_{\ep (t-s)^{-1/d_\gamma}}(\partial \eta[0,1]). 
    \end{equation}
    Hence, it suffices to prove the lemma for $[s,t] = [0,1]$ only. To prove the lemma in this case, we will first use Theorem~\ref{thm:boundary-length-BM} and an elementary Brownian motion estimate to bound the number of $\ep^{-d_\gamma}$-length space-filling SLE segments needed to cover $\partial \eta[0,1]$. We will then conclude by combining this bound with  Proposition~\ref{prop:covering-num-moment-bound}.
    
    Without loss of generality, we may re-scale the LQG boundary length measure by a $\gamma$-dependent constant factor so that for the boundary length process $(L,R)$ appearing in Theorem~\ref{thm:boundary-length-BM}, the variance parameter $a$ equals 1. Let $\partial^L$ and $\partial^R$ denote the left and right boundaries of $\eta[0,1]$, respectively. Further decompose $\partial^L$ and $\partial^R$ into
    \begin{equation} 
        \partial_0^q := \partial^q  \cap \eta(-\infty,0] \quad\text{and}\quad \partial_1^q  := \partial^q  \cap \eta[1,\infty) \quad \text{where } q\in \{L,R\}. 
    \end{equation}
    That is, in Figure~\ref{fig:peanosphere-bm} with $t=1$, the orange boundary corresponds to $\partial_0^L$, the green boundary to $\partial_0^R$, the brown boundary to $\partial_1^L$, and the purple boundary to $\partial_1^R$.
    
    Let us first show that $\lim_{\ep \to 0} \ep^{\delta} N_\ep(\partial_0^L) = 0$ almost surely. Let $K$ be a positive integer, which shall be determined later. Denote $I_k = [\frac{k}{K}, \frac{k+1}{K}]$. We start with the following trivial inequality: 
    \begin{equation} N_\ep(\partial_0^L) \leq \sum_{k=0}^{K-1} N_\ep( \partial_0^L \cap \eta(I_k)  ) \leq \sum_{k=0}^{K-1} \mathbf{1}\{\partial_0^L \cap \eta(I_k) \neq \varnothing\} \cdot N_\ep( \eta(I_k) ) . \end{equation}    
    We now consider the following geometric encoding of $\eta$ by $(L,R)$ described in \cite[Section 4.2.3]{ghs-mating-survey}: the times $t \in [0,1]$ such that $\eta(t) \in \partial_0^L $ (resp.\ $\partial_0^R$) are precisely those at which $L_t$ (resp.\ $R_t$) attains a running minimum. Recall that $L_t$ is a standard Brownian motion. A well-known result of P.\ L\'evy states that the process $L_t - \min_{0\leq s \leq t} L_s$ is a reflected Brownian motion, whereas the set $\{t\in [0,1]: \eta(t) \in \partial_0^L\}$ has the law of the zero set of Brownian motion on $[0,1]$ (e.g., see \cite[Theorem 2.34]{morters-peres-bm}). Then, by the arcsine law for the last time that a Brownian motion $B_t$ changes sign, for any nonnegative integer $k$ regardless of $K$,
    \begin{equation} \label{eqn:arctan}
    \begin{aligned}
         \Prob \{\partial_0^L \cap \eta(I_k) \neq \varnothing \} & = \Prob\{ B_t = 0 \text{ for some } t\in I_k \} = \Prob\left\{ B_t = 0 \text{ for some } t \in \left[\frac{k}{k+1},1\right]\right\} \\
         &= 1- \frac{2}{\pi} \arcsin \sqrt{\frac{k}{k+1}} = 1 - \frac{2}{\pi} \arctan \sqrt{k}.
    \end{aligned}
    \end{equation}
    For $p>1$, by H\"older's inequality followed by Corollary~\ref{cor:covering-number-relation} and~\eqref{eqn:arctan},
    \begin{equation} \begin{aligned} 
        \Exp[N_\ep (\partial_0^L)] &\leq \sum_{k=0}^{K-1} \Prob (\partial_0^L \cap \eta(I_k) \neq \varnothing )^{1-\frac{1}{p}} \cdot  \Exp\big[\big(N_\ep( \eta(I_k) )\big)^p\big]^{\frac{1}{p}} \\
        & =\sum_{k=0}^{K-1} \left(1 - \frac{2}{\pi} \arctan \sqrt k \right)^{1-\frac{1}{p}} \cdot \frac{K}{\ep^{d_\gamma}} \Exp\left[\left((\ep K^{-1/d_\gamma})^{d_\gamma} N_{\ep K^{-1/d_\gamma}}(\eta[0,1])\right)^p\right]^{\frac{1}{p}}.
    \end{aligned} \end{equation}
    We now take 
    \begin{equation}
        K = \lceil\ep^{-d_\gamma}\rceil .
    \end{equation}
    From Proposition~\ref{prop:covering-num-moment-bound}, there exists a constant $C_p>0$ such that 
    \begin{equation} 
        \sup_{0 < \ep < 1} \Exp\left[\left(\ep K^{1/d_\gamma})^{d_\gamma} N_{\ep K^{-1/d_\gamma}}(\eta[0,1])\right)^p\right]^{\frac{1}{p}} \leq C_p. 
    \end{equation}
    Since $1 - \frac{2}{\pi} \arctan \sqrt k = O(1/\sqrt{k})$ as $k\to \infty$,
    \begin{equation}  
        \begin{aligned} \Exp[ \ep^{\delta} N_\ep(\partial_0^L)] &\leq 2 C_p \ep^{\delta} \sum_{k=0}^{K} \left(1 - \frac{2}{\pi}\arctan \sqrt k\right)^{1-\frac{1}{p}} \\ &=O\left( \ep^{\delta} K^{1 - \frac{1}{2}(1-\frac{1}{p})} \right) = O\left(\ep^{\delta - \frac{d_\gamma}{2}(1+\frac{1}{p}))}\right)  \quad \text{as } \ep \to 0.  \end{aligned}
    \end{equation}
    Since $\delta > d_\gamma/2$, the right-hand side tends to 0 as $\ep \to 0$ for sufficiently large $p$. This proves the claim that $\lim_{\ep \to 0} \ep^{\delta} N_\ep(\partial_0^L) = 0$ almost surely.
    
    Replacing $L_t$ with $R_t$ gives $\lim_{\ep \to 0} \ep^{\delta} N_\ep(\partial_0^R) = 0$ almost surely. By the reversibility of the whole-plane space-filling SLE $\eta$ and Proposition~\ref{prop:qc-mmspace-invariance},
    \begin{equation}
        (\C, \eta(1), D_{h^\gamma}, \mu_{h^\gamma}, \eta(1 - \cdot)) \stackrel{d}{=} (\C, 0, D_{h^\gamma}, \mu_{h^\gamma}, \eta).
    \end{equation}
    Consequently, $N_\ep(\partial_0^L) \stackrel{d}{=} N_\ep(\partial_1^L)$ and $N_\ep(\partial_0^R) \stackrel{d}{=} N_\ep(\partial_1^R)$. Since $\bd \eta[0,1] = \bd_0^L \cup \bd_0^R \cup \bd_1^L \cup \bd_1^R$, we therefore obtain the lemma statement.
\end{proof}

Let us first sketch how to deduce \eqref{eqn:sle-boundary-blowup-covering} from Lemma~\ref{lem:sle-dim-1/2}. First, note that it suffices to show that for some fixed $\zeta \in (0,d_\gamma/2)$, with probability tending to 1 as $\ep \to 0$, we have $N_\ep(\bd_\ep\eta[s,t]) \leq \ep^{-\zeta} N_\ep(\bd\eta[s,t])$. The idea is to sample a collection $\mathcal X_\ep$ of $\lfloor \ep^{-\zeta}N_\ep(\bd\eta[s,t])\rfloor$ i.i.d.\ points in $\mathcal \bd_{3\ep/2}\eta[s,t]$ from the LQG measure $\mu_{h^\gamma}$, and show that $\bigcup_{x\in\mathcal X_\ep} \mathcal B_\ep(x;D_{h^\gamma})$ covers $\bd_\ep\eta[s,t]$.

To this end, we sample another set of i.i.d.\ points $\widetilde W_\ep$ from $\mu_{h^\gamma}$, conditionally independent of $\mathcal X_\ep$ given $h^\gamma$ and $\eta$, such that $\bd_\ep \eta[s,t] \subset \bigcup_{w\in \widetilde{\mathcal{W}}_\ep} \mathcal B_{\ep/2}(w;D_{h^\gamma}) \subset \bd_{2\ep} \eta[s,t]$. (Because the Minkowski dimension of $\gamma$-LQG is finite, we only need polynomially many points in $\ep$ in $\widetilde{\mathcal W}_\ep$ \cite[Theorem~A.3]{gwynne-geodesic-network}.) For each $w\in \widetilde{\mathcal{W}}_\ep$, the conditional probability given $h^\gamma$, $\eta$, and $x \in \mathcal X_\ep$ that $\mathcal B_{\ep/2}(w;D_{h^\gamma}) \subset \mathcal B_{\ep}(x;D_{h^\gamma})$ is given by $\mu_{h^\gamma}(\mathcal B_{\ep/2}(x;D_{h^\gamma}))/\mu_{h^\gamma}(\bd_{2\ep} \eta[s,t])$. Using Theorem~\ref{thm:metric-ball-volume} (the volume estimate for LQG metric balls), we show that this number is no more than $\ep^{-\zeta}/|\mathcal X_\ep|$. Since the points in $\mathcal X_\ep$ are sampled conditionally i.i.d., the conditional probability that $\mathcal B_{\ep/2}(w;D_{h^\gamma}) \subset \bigcup_{x\in \mathcal X_\ep} \mathcal B_\ep(x;D_{h^\gamma})$ is no less than $1-(1-\ep^{-\zeta}/|\mathcal X_\ep|)^{|\mathcal X_\ep|}$, which tends to 1 superpolynomially fast as $\ep \to 0$. Since the cardinality of $\widetilde{\mathcal W}_\ep$ is polynomial in $\ep$, by taking a union bound over points in this set, we conclude that $\bd_\ep \eta[s,t] \subset \bigcup_{w\in \widetilde{\mathcal{W}}_\ep} \mathcal B_{\ep/2}(w;D_{h^\gamma}) \subset \bigcup_{x\in \mathcal X_\ep} \mathcal B_\ep(x;D_{h^\gamma})$ with probability increasing to 1 as $\ep \to 0$.

What we have above are essentially the statement and the proof of Lemma~\ref{lem:epsilon-nbhd-covering}, except that they are for the whole-plane GFF $h$ instead of the $\gamma$-quantum cone field $h^\gamma$. The reason that we prove \eqref{eqn:epsilon-nbhd-covering} for $h$ first is that most results about the Minkowski dimension of $\gamma$-LQG, such as Theorem~\ref{thm:metric-ball-volume}, are stated in terms of $h$ instead of $h^\gamma$. After proving Lemma~\ref{lem:epsilon-nbhd-covering}, we transfer this result to the setting of a $\gamma$-quantum cone using Lemma~\ref{lem:zooming-in} to complete the proof of Lemma~\ref{lem:sle-nbhd-zero-content}. Recall the notation $\mathcal B_\ep(A;D_h) := \bigcup_{z \in A} \mathcal B_\ep(z;D_h)$.

\begin{lem}\label{lem:epsilon-nbhd-covering}
    Let $h$ be a whole-plane field whose law is absolutely continuous with that of the whole-plane GFF normalized to have mean zero on $\bd \ud$. Let $K\subset \C$ be a fixed compact set. Let $A \subset K$ be a random set, not necessarily independent from $h$. Then, for each $\zeta >0$, 
    \begin{equation}\label{eqn:epsilon-nbhd-covering} 
        N_\ep(\mathcal B_\ep(A; D_h); D_h) \leq \ep^{-\zeta}N_\ep(A; D_h)
    \end{equation}
    holds with probability tending to 1 as $\ep \to 0$.
\end{lem}

\begin{proof}
    Let $\mathcal X_\ep$ be a collection of $\lfloor \ep^{-\zeta} N_{\ep}(A;D_h) \rfloor$ points in $\mathcal B_{2\ep}(A;D_h)$ sampled, given $h$ and $A$, conditionally i.i.d.\ from $\mu_h|_{\mathcal B_{2\ep}(A;D_h)}$ normalized to be a probability measure. We claim that $\bigcup_{x\in \mathcal X_\ep} \mathcal B_\ep(x;D_h)$ covers $\mathcal B_\ep(A;D_h)$ with probability increasing to 1 as $\ep \to 0$.
    
    To show that $\bigcup_{x\in \mathcal X_\ep} \mathcal B_\ep(x;D_h)$ covers $\mathcal B_\ep(A;D_h)$, we consider an $\frac{\ep}{2}$-cover of $\mathcal B_\ep(A;D_h)$ and then check that $\bigcup_{x\in \mathcal X_\ep}\mathcal B_\ep(x;D_h)$ contains this $\frac{\ep}{2}$-cover. Let us first compute the conditional probability that an LQG metric ball $\mathcal B_{\ep/2}(w;D_h)$ appearing in the $\frac{\ep}{2}$-cover is contained in $\bigcup_{x\in \mathcal X_\ep} \mathcal B_\ep(x;D_h)$. If $\mathcal B_{\ep/2}(w;D_h) \cap \mathcal B_\ep(A;D_h) \neq \varnothing$, then $w\in \mathcal B_{3\ep/2}(A;D_h)$ and $\mathcal B_{\ep/2}(w;D_h) \subset \mathcal B_{2\ep}(A;D_h)$.
    Conditioned on $h$, $A$, and $w \in \mathcal B_{3\ep/2}(A;D_h)$, each $x \in \mathcal X_\ep$ has
    \begin{equation} 
        \Prob \left(\mathcal B_{\ep/2}(w;D_h) \subset \mathcal B_\ep(x;D_h) \middle| h,A,w\right) \geq  \Prob \left( x \in \mathcal B_{\ep/2}(w;D_h) \middle| h,A, w\right) = \frac{\mu_h(\mathcal B_{\ep/2}(w;D_h))}{\mu_h(\mathcal B_{2\ep}(A;D_h))}. 
    \end{equation}    
    Since the points in $\mathcal X_\ep$ are conditionally i.i.d.\ given $h$ and $A$,
    \begin{equation}\label{eqn:metric-nbd-pt}
        \Prob \left(\mathcal B_{\ep/2}(w;D_h) \subset \bigcup_{x \in \mathcal X_\ep} \mathcal B_\ep(x;D_h) \middle| h,A,w\right) \geq 1 - \left(1 - \frac{\mu_h(\mathcal B_{\ep/2}(w;D_h))}{\mu_h(\mathcal B_{2\ep}(A;D_h))}\right)^{\lfloor\ep^{-\zeta}N_\ep(A;D_h)\rfloor}.
    \end{equation}
    
    We now specify the $\frac{\ep}{2}$-cover of $\mathcal B_\ep(A;D_h)$. Let $U:= B_1(K)$ and $W:= B_2(K)$. Let $\mathcal W_\ep$ be a collection of $\lfloor (\ep/2)^{-d_\gamma - \zeta} \rfloor$ points sampled conditionally i.i.d.\ from $\mu_h|_{W}$ normalized to be a probability measure, where $\mathcal W_\ep$ and $\mathcal X_\ep$ are conditionally independent given $h$. 
    By \cite[Lemma~A.3]{gwynne-geodesic-network}, the event 
    \begin{equation}
        E_1 := \left\{U \subset \bigcup_{w\in \mathcal W_\ep}\mathcal B_{\ep/2}(w;D_h)\right\}
    \end{equation} 
    occurs with probability tending to 1 as $\ep \to 0$. Also, since $D_h$ is almost surely a continuous metric, the event
    \begin{equation}
        E_2:= \{ D_h(K, \bd U)\geq \ep\} \cap \{D_h(\bd U, \bd W) \geq \ep\}
    \end{equation}
    occurs with probability tending to 1 as $\ep \to 0$. Truncating on the event $E_1 \cap E_2$, we have $\mathcal B_\ep(A;D_h) \subset U$ and hence
    \begin{equation}\label{eqn:intermediate-cover}
        \mathcal B_\ep(A;D_h) \subset \bigcup_{w\in \widetilde {\mathcal W}_\ep}\mathcal B_{\ep/2}(w;D_h) \quad \text{where} \quad \widetilde {\mathcal W}_\ep:= \mathcal W_\ep \cap \mathcal B_{3\ep/2}(A;D_h).
    \end{equation}
    This is the $\frac{\ep}{2}$-cover we choose for $\mathcal B_\ep(A;D_h)$, which holds on the event $E_1\cap E_2$. By \cite[Theorem 1.1]{afs-metric-ball}, there exists a random $c>0$ which depends only on $W$ and $\zeta$ such that a.s.\ for every $\ep \in (0,1)$,
    \begin{equation} \label{eqn:metric-nbd-mass}
    \frac{\inf_{z \in W} \mu_h(\mathcal B_{\ep/2}(z;D_h))}{\sup_{z \in W} \mu_h(\mathcal B_{3\ep}(z;D_h))} \geq c\ep^{\zeta/2}. 
    \end{equation} 
    Note that 
    \begin{equation} \label{eqn:double-radius-cover}
        \mu_h(\mathcal B_{2\ep}(A;D_h)) \leq N_\ep(A;D_h) \left( \sup_{z\in \mathcal B_\ep(A;D_h)} \mu_h(\mathcal B_{3\ep}(z;D_h)) \right)
    \end{equation}
    since we can cover $\mathcal B_{2\ep}(A;D_h)$ by first choosing an $\ep$-cover of $A$ and then blowing up the radius of every metric ball in this cover to $3\ep$. Combining \eqref{eqn:metric-nbd-mass} and \eqref{eqn:double-radius-cover}, we almost surely have
    \begin{equation}\label{eqn:volume-ratio-to-covering-num}
        \inf_{w\in \mathcal B_{3\ep/2}(A;D_h)}\frac{\mu_h(\mathcal B_{\ep/2}(w;D_h))}{\mu_h(\mathcal B_{2\ep}(A;D_h))} \geq \frac{1}{N_\ep(A;D_h)} \cdot \frac{\inf_{z \in W} \mu_h(\mathcal B_{\ep/2}(z;D_h))}{\sup_{z \in W} \mu_h(\mathcal B_{3\ep}(z;D_h))} \geq \frac{c\ep^{\zeta/2}}{N_\ep(A;D_h)}
    \end{equation}
    truncated on the event $E_2$ so that $\mathcal B_{2\ep}(A;D_h) \subset W$. Combining \eqref{eqn:metric-nbd-pt} with \eqref{eqn:volume-ratio-to-covering-num}, for each $w\in \mathcal W_\ep$ we have
    \begin{equation}\label{eqn:intermediate-cover-individual}
    \begin{aligned}
        &\Prob\left(\mathcal B_{\ep/2}(w;D_h) \subset \bigcup_{x\in \mathcal X_\ep} \mathcal B_\ep(x;D_h) \text{ or }w \notin \mathcal B_{3\ep/2}(A;D_h)\middle| h,A, E_2 \right) \\ &  \geq 1 - \left(1 - \frac{c\ep^{-\zeta/2}}{\ep^{-\zeta}N_\ep(A;D_h)}\right)^{\ep^{-\zeta}N_\ep(A;D_h)}.
    \end{aligned}
    \end{equation}
    Since $\ep^{-\zeta}N_\ep(A;D_h) \geq \ep^{-\zeta}$, the conditional probability in \eqref{eqn:intermediate-cover-individual} becomes superpolynomially high as $\ep\to 0$ at a rate uniform on $h$ and $A$. Since there are $\ep^{-d_\gamma - \zeta}$ points in $\mathcal W_\ep$, an intersection of \eqref{eqn:intermediate-cover-individual} over the points $w\in \mathcal W_\ep$ implies that truncated on the event $E_2$, we have $\bigcup_{w\in \widetilde{W}_\ep} \mathcal B_{\ep/2}(w;D_h) \subset \bigcup_{x\in \mathcal X_\ep} \mathcal B_\ep(x;D_h)$ with superpolynomially high probability as $\ep \to 0$. Recalling our choice of the $\frac{\ep}{2}$-cover of $\mathcal B_\ep(A;D_h)$ in \eqref{eqn:intermediate-cover}, we conclude the proof of the lemma by observing that $\Prob\left\{ \mathcal B_\ep(A;D_h) \subset \bigcup_{x\in \mathcal X_\ep} \mathcal B_\ep(x;D_h) \right\}$ is bounded below by
    \begin{equation}
         \Prob\left(E_1 \cap E_2 \cap \left\{ \bigcup_{w\in \widetilde{\mathcal W}_\ep} \mathcal B_{\ep/2}(w;D_h) \subset \bigcup_{x\in \mathcal X_\ep} \mathcal B_\ep(x;D_h) \right\}\right),
    \end{equation}
    which tends to 1 as $\ep\to 0$. 
\end{proof}

We are now ready to complete the proof of Lemma~\ref{lem:sle-nbhd-zero-content}. 

\begin{proof}[Proof of Lemma~\ref{lem:sle-nbhd-zero-content}]
    It suffices to show $\lim_{\ep \to 0} \ep^{d_\gamma} N_\ep(\bd_\ep \eta[0,1]) = 0$ almost surely, due to Proposition~\ref{prop:qc-mmspace-invariance}. Let $(z,h)$ be sampled from $\mathbf{1}_{\ud}(z) \mu_h(dz) dh$ normalized to be a probability measure as described in Lemma~\ref{lem:zooming-in}. For each constant $C$, let $\phi_C$ be the random translation and scaling such that $h^C:=(h+C)\circ \phi_C + Q\log|\phi_C'|$ is the field under the circle average embedding of $(\C,h+C,z,\infty)$. Given $\beta>0$, choose a large constant $R>0$ such that the event
    \begin{equation}
        E_3:= \{\eta[0,1]\subset B_R(0) \text{ and } D_{h^\gamma}(\bd B_R(0), \bd B_{2R}(0)) \geq 1\}
    \end{equation}
    has probability at least $1-\beta/3$. Then, choose a large constant $C >0$ such that the following two conditions are satisfied. 
    \begin{enumerate}[(i)]
        \item The total variation distance between the laws of the fields $h^\gamma$ and $h^C$ restricted to $B_{2R}(0)$ is at most $\beta/3$.
        \item The event $E_4:=\{\phi_C(B_R(0))\subset B_2(0)\}$ has probability at least $1-\beta/3$.
    \end{enumerate}
    For sufficiently large $C$, the first item holds by Lemma~\ref{lem:zooming-in} and the second item holds by the conditional law of $h(\cdot-z)$ as described in \cite[Lemma~A.10]{dms-lqg-mating}.

    Fix $\zeta \in (0,d_\gamma/2)$. Couple $h^C$ and $h^\gamma$ so that their restrictions to $B_{2R}(0)$ agree on an event of probability $1-\beta/3$, which we call $E_5$. Substitute $A = \phi_C(\bd\eta[0,1])$ and $K = \overline{B_2(0)}$ in Lemma~\ref{lem:epsilon-nbhd-covering}. Then, truncated on $E_3 \cap E_4 \cap E_5$, the event
    \begin{equation}\label{eqn:zoomed-in-covering-numbers}
        N_{\ep C^\xi}(\mathcal B_{\ep C^\xi}(\phi_C(\bd\eta[0,1]);D_h);D_h)
        \leq (\ep C^\xi)^{-\zeta} N_{\ep C^\xi}(\phi_C(\bd\eta[0,1]);D_h)
    \end{equation}
    occurs with probability increasing to the full probability of $E_3 \cap E_4 \cap E_5$ as $\ep \to 0$. By Lemmas~\ref{lem:translation-metric-invariance} and \ref{lem:scaling-metric-covariance}, we almost surely have
    \begin{equation} \begin{aligned}
        N_{\ep C^\xi}(\mathcal B_{\ep C^\xi}(\phi_C(\bd\eta[0,1]);D_h);D_h) & = N_\ep(\mathcal B_\ep(\phi_C(\bd\eta[0,1]);D_{h+C});D_{h+C}) \\ & = N_\ep(\mathcal B_\ep(\bd\eta[0,1];D_{h^C});D_{h^C}) \end{aligned}
    \end{equation}
    on the left-hand side of \eqref{eqn:zoomed-in-covering-numbers} and 
    \begin{equation} \begin{aligned}
        (\ep C^\xi)^{-\zeta} N_{\ep C^\xi}(\phi_C(\bd\eta[0,1]);D_h) & = (\ep C^\xi)^{-\zeta} N_\ep(\phi_C(\bd\eta[0,1]);D_{h+C}) \\ & = (\ep C^\xi)^{-\zeta}N_\ep(\bd\eta[0,1]);D_{h^C}) \end{aligned}
    \end{equation}
    on the right-hand side. On the event $E_3 \cap E_5$, since $D_{h^C}(\bd\eta[0,1], \bd B_{2R}(0)) \geq 1$, both $N_\ep(\bd\eta[0,1];D_{h^C})$ and $N_\ep(\mathcal B_\ep(\bd\eta[0,1];D_{h^C});D_{h^C})$ for $\ep \in(0, 1/2)$ are almost surely determined by $h^C|_{B_{2R}(0)}$. Hence, \eqref{eqn:zoomed-in-covering-numbers} implies 
    \begin{equation}\label{eqn:boundary-covering-conclusion}
        \ep^{d_\gamma} N_\ep(\bd_\ep \eta[0,1];D_{h^\gamma}) \leq  C^{-\xi\zeta} \ep^{d_\gamma-\zeta} N_\ep(\bd \eta[0,1];D_{h^\gamma}).
    \end{equation}
    Since $d_\gamma -\zeta > d_\gamma/2$, we deduce from Lemma~\ref{lem:sle-dim-1/2} that, truncated on $E_3\cap E_4\cap E_5$, the right-hand side of \eqref{eqn:boundary-covering-conclusion} tends to 0 almost surely as $\ep \to 0$. This proves the lemma since $E_3\cap E_4\cap E_5$ occurs with probability at least $1-\beta$, and our choice of $\beta$ was arbitrary.
\end{proof}

\subsection{Identifying the Minkowski content process}\label{sec:minkowski-content-determinstic}

For the remainder of this section, let $(X_I: I \in \mathcal I_\Q)$ be a weak limit of $(\mathfrak b_{\ep_n}^{-1} N_{\ep_n} (\eta(I)): I \in \mathcal I_\Q)$ for some sequence $\ep_n$ that tends to 0. As described in the proof overview for Proposition~\ref{prop:main-thm-qc-sle}, we define the ``Minkowski content process" $\{Y_t\}_{t\in \Q}$ by
    \begin{equation}\label{eqn:minkowski-content-process} 
    Y_t := \begin{cases} X_{[0,t]} & t>0, \\ 0 & t=0, \\ -X_{[t,0]} & t < 0. \end{cases}  
    \end{equation}
By Proposition~\ref{prop:finite-additivity}, $    X_{[s,t]} = Y_t - Y_s$ a.s.\ for all $[s,t]\in \mathcal I_\Q$. The next step is to show that $\{Y_t\}_{t\in \Q}$ extends to a two-sided Brownian motion with drift.

\begin{prop}\label{prop:minkowski-content-bm}
    The process $\{Y_t\}_{t\in \Q}$ defined in \eqref{eqn:minkowski-content-process} extends to a L\'evy process with almost surely continuous paths: i.e., a two-sided Brownian motion with drift.
\end{prop}

\begin{proof} 
    We first show that $\{Y_t\}_{t\in\Q}$ extends to a continuous process defined on $\R$. Let $p>1$. For every rational $s<t$, by Fatou's lemma, Corollary~\ref{cor:normalization-constant-comparison}, and Proposition~\ref{prop:covering-num-moment-bound},
    \begin{equation} 
        \begin{aligned} \Exp [|Y_t- Y_s|^p]  = \Exp[|X_{[s,t]}|^p] &\leq \liminf_{n\to \infty} \Exp[|\mathfrak b_{\ep_n}^{-1} N_{\ep_n}(\eta[s,t])|^p] \\& \leq C_1 \liminf_{n\to \infty} \Exp[| \ep_n^{d_\gamma} N_{\ep_n}(\eta[s,t])|^p] \leq C_2 |t-s|^p \end{aligned}
    \end{equation}
    for constants $C_1,C_2>0$ depending only on $\gamma$. We deduce from the Kolmogorov continuity criterion that $\{Y_t\}_{t\in \Q}$ can be extended to a process $\{Y_t\}_{t\in \R}$ whose paths are almost surely continuous.

    It now suffices to check that $\{Y_t\}_{t\in\Q}$ has stationary and independent increments since these properties extend to $\{Y_t\}_{t\in\R}$ by continuity.
    
    \begin{itemize}
    \item \textit{Stationary increments.} We found in Corollary~\ref{cor:covering-number-relation} that $N_\ep(\eta[s,t]) \stackrel{d}{=} N_\ep(\eta[0,t-s])$ for every fixed $s<t$ and $\ep>0$. Hence, $Y_t - Y_s = X_{[s,t]}$ and $Y_{t-s} = X_{[0,t-s]}$ agree in law for every rational $s<t$.

    \item \textit{Independent increments.} For each $s\in \R$, denote by $U_{s-}$ and \ $U_{s+}$ the interiors of $\eta(-\infty,s]$ and $\eta[s,\infty)$, respectively. We claim that for any $\ep>0$ and any fixed interval $I\in \mathcal I_\Q$, if $I \subset (-\infty,s]$ (resp.\ $I \subset [s,\infty)$), then the quantity $N_\ep^\circ(\eta(I))$ is almost surely determined by the curve-decorated metric measure space $(U_{s-}, D_h^{U_{s-}}, \mu_h|_{U_{s-}}, \eta|_{(-\infty,s]})$ (resp.\ $(U_{s+}, D_h^{U_{s+}}, \mu_h|_{U_{s+}}, \eta|_{[s,\infty)]})$).
    
    Let us prove the claim. Without loss of generality, assume $I \subset (-\infty,s]$. Recall the definition 
    \begin{equation}N_\ep^\circ(\eta(I)) = N_\ep(\eta(I) \setminus \bd_\ep \eta(I);D_{h^\gamma}) \end{equation} where $\eta(I) \setminus \bd_\ep \eta(I) := \{z\in \eta(I): D_{h^\gamma}(z, \bd \eta(I)) \geq \ep\}$. Suppose $\mathcal B_\ep(w;D_{h^\gamma})$ is a ball counted in $N_\ep^\circ(\eta(I))$. Then $\mathcal B_\ep(w;D_{h^\gamma})$ has a nonempty intersection with $\eta(I)\setminus\bd_\ep \eta(I)$, so $w$ must be in the interior of $\eta(I)$. Now suppose $z$ is any point within the intersection of $\mathcal B_\ep(w;D_{h^\gamma})$ and $\eta(I) \setminus\bd_\ep\eta(I)$. Then, the $D_{h^\gamma}$-geodesic from $z$ to $w$ must be contained in the interior of $\eta(I)$, or otherwise its $D_{h^\gamma}$-length would be at least \begin{equation} D_{h^\gamma}(z,\bd \eta(I)) + D_{h^\gamma}(w,\bd\eta(I)) \geq \ep. \end{equation} Since $\mathrm{int}(\eta(I))\subset U_{s-}$, we have $D_{h^\gamma}(z,w) = D_{h^\gamma}^{U_{s-}}(z,w)$. Since this holds for any $z$ in the intersection of $\mathcal B_\ep(w;D_{h^\gamma})$ and $\eta(I) \setminus\bd_\ep\eta(I)$, we have
    \begin{equation} \label{eqn:internal-metric-covered-points}
        \mathcal B_\ep(w;D_{h^\gamma}^{U_{s-}}) \cap (\eta(I)\setminus \bd_\ep\eta(I)) = \mathcal B_\ep(w;D_{h^\gamma}) \cap (\eta(I)\setminus \bd_\ep\eta(I)).
    \end{equation}    
    Since \eqref{eqn:internal-metric-covered-points} holds for any $w$ in the interior of $\eta(I)$, we obtain
    \begin{equation}\label{eqn:independence-covering-num-locality}
        N_\ep^\circ(\eta(I)) = N_\ep(\eta(I)\setminus\bd_\ep\eta(I); D_{h^\gamma}^{U_{s-}}).
    \end{equation}
    Observe that 
    \begin{equation} \eta(I)\setminus \bd_\ep\eta(I) = \{z \in \eta(I): \mathcal B_\ep(z; D_{h^\gamma}^{U_{s-}} ) \subset \eta(I)\}.
    \end{equation}
    Hence, the right-hand side of \eqref{eqn:independence-covering-num-locality} is a.s.\ determined by $(U_{s-}, D_h^{U_{s-}}, \mu_h|_{U_{s-}}, \eta|_{(-\infty,s]})$.

    Let us now deduce from the claim that $\{Y_t\}_{t\in\Q}$ has independent increments. Let $t_0 < t_1 < \dots<t_m$ be a fixed collection of finitely many rational times. Setting $s = t_k$, we see from the claim that for each $\ep_n$, the random vector $(\mathfrak b_{\ep_n}^{-1} N_{\ep_n}^\circ(\eta[t_{j-1},t_j]): 1\leq j\leq k)$ is a.s.\ determined by $(U_{s-}, D_h^{U_{s-}}, \mu_h|_{U_{s-}}, \eta|_{(-\infty,s]})$, and $(\mathfrak b_{\ep_n}^{-1} N_{\ep_n}^\circ(\eta[t_{j-1},t_j]): k+1\leq j\leq m)$ is a.s.\ determined by $(U_{s+}, D_h^{U_{s+}}, \mu_h|_{U_{s+}}, \eta|_{[s,\infty)})$. These two random vectors are thus independent by Proposition~\ref{prop:left-right-independent-mmspace} combined with the translation invariance of Proposition~\ref{prop:qc-mmspace-invariance}. Recalling \eqref{eqn:covering-num-joint-convergence}, we conclude that $(X_{[t_{j-1},t_j]}: 1\leq j\leq k)$ and $(X_{[t_{j-1},t_j]}: k+1 \leq j\leq m)$ are independent. This works for any $1\leq k\leq m$, so the random variables $X_{[t_0,t_1]},\dots,X_{[t_{m-1},t_m]}$ are independent.
    \end{itemize}\vspace{-20pt}
\end{proof}

We conclude the proof of Proposition~\ref{prop:main-thm-qc-sle} by checking that the Minkowski content process has almost surely positive increments.

\begin{prop}[Positive increments]\label{prop:positive-increments}
    Almost surely, $X_{[s,t]}>0$ for every rational $s<t$.
\end{prop}
\begin{proof} 
By Proposition~\ref{prop:covering-number-lower-bound}, we can find a deterministic constant $c>0$ such that with probability tending to 1 as $\ep \rightarrow 0$, we have $ N_\ep(\eta[0,1]) \geq c \ep^{-d_\gamma}$. By Corollary~\ref{cor:covering-number-relation}, for any fixed rational times $s<t$, it also holds with probability tending to 1 as $\ep\rightarrow 0$ that $ N_\ep(\eta[s,t]) \geq c \ep^{-d_\gamma} (t-s)$. 
Combining this with Corollary~\ref{cor:normalization-constant-comparison}, we deduce that $X_{[s,t]} > 0$ almost surely for every rational $s<t$.
\end{proof}

\begin{proof}[Proof of Proposition~\ref{prop:main-thm-qc-sle}]

We proved that $\{\ep^{d_\gamma} N_\ep(\eta[s,t])\}_{0 < \ep < (t-s)^{1/d_\gamma} }$ is tight for any $s<t$ in Proposition~\ref{prop:covering-num-moment-bound}. By Corollary~\ref{cor:normalization-constant-comparison}, $\{\mathfrak b_\ep^{-1} N_\ep(\eta[s,t])\}_{0 < \ep < (t-s)^{1/d_\gamma}}$ is also tight. Therefore, given any sequence of $\ep$ decreasing to 0, we can find a subsequence $\ep_n$ for which the corresponding sequence of $\R^{\mathcal I_\Q}$-valued random variables $(\mathfrak b_{\ep_n}^{-1} N_{\ep_n}(\eta(I)): I \in \mathcal I_\Q)$ converges in distribution w.r.t.\ the product topology on $\R^{\mathcal I_\Q}$.

Now suppose $\ep_n$ is any sequence decreasing to 0 such that $(\mathfrak b_{\ep_n}^{-1} N_{\ep_n}(\eta(I)): I \in \mathcal I_\Q)$ has a weak limit $(X_I:I\in \mathcal I_\Q)$. Define the process $\{Y_t\}_{t\in\Q}$ as in \eqref{eqn:minkowski-content-process}. By Proposition~\ref{prop:minkowski-content-bm}, $\{Y_t\}_{t\in\Q}$ extends to a Brownian motion with drift. Since $\{Y_t\}_{t\in\Q}$ has positive increments (Proposition~\ref{prop:positive-increments}), the variance of the Brownian motion must be zero. That is, there exists a deterministic constant $a>0$ such that almost surely, $Y_t = at$ for all $t$. In other words, for every rational $s<t$, the subsequence $\mathfrak b_{\ep_n}^{-1} N_{\ep_n}(\eta[s,t])$ converges in distribution to $a(t-s)$ as $n\to \infty$. Since $a(t-s)$ is a deterministic constant, the convergence is in probability. In fact, for every $p\geq 1$, the convergence is also in $L^p$ since $\sup_{0 < \ep < |t-s|^{1/d_\gamma}} \mathbb E[|\mathfrak b_\ep^{-1} N_\ep(\eta[s,t])|^p]<\infty$ (Proposition~\ref{prop:covering-num-moment-bound}). By the definition~\eqref{eqn:scaling-constant} of $\mathfrak b_\ep$, we have $\Exp[\mathfrak b_{\ep}^{-1} N_{\ep}(\eta[0,1])] = 1$ for all $\ep>0$. Hence, $a = 1$ regardless of the choice of the subsequence $\ep_n$. 
 
In summary, we have that given any sequence of $\ep$ decreasing to 0, there exists a subsequence $\ep_n$ such that $\lim_{n\to \infty} \mathfrak b_{\ep_n}^{-1} N_{\ep_n}(\eta[s,t]) = t-s$ in probability for all rational $s<t$. This proves \eqref{eqn:sle-minkowski-content} for rational $s<t$. Now suppose $s<t$ are any fixed real numbers. If $s_1 < s_0 < t_0 < t_1$ are rational numbers such that $[s_0,t_0] \subset [s,t] \subset [s_1,t_1]$, then 
\begin{equation}
    N_\ep(\eta[s_0,t_0]) \leq N_\ep(\eta[s,t]) \leq N_\ep(\eta[s_1,t_1]).
\end{equation}
By taking $s_1,s_0$ close to $s$ and $t_0,t_1$ close to $t$, we conclude that $\lim_{n\to\infty} \mathfrak b_{\ep}^{-1}N_\ep(\eta[s,t]) = t-s$ in probability.
\end{proof}

\section{Generalization to other sets and fields}\label{sec:generalization}

In this section, we show that the LQG measure $\mu_h$ can be recovered from $D_h$ in terms of Minkowski content with respect to $\mathfrak b_\ep$, not only for the $\gamma$-quantum cone but also for other variants of the whole-plane GFF. 

In general, Minkowski content is not a measure. Accordingly, we do not expect that the Minkowski content of an arbitrary bounded Borel set $A \subset \C$ is equal to its LQG measure $\mu_h(A)$. The sufficient condition we impose is that $A$ is a random bounded Borel set such that 
\begin{equation}\label{eqn:sufficient-condition}
    \mu_{h}(\bd A) = 0 \quad \text{a.s.}
\end{equation}
Since $\bd A$ is closed, this condition is equivalent to 
\begin{equation}\label{eqn:sufficient-condition-euclidean} 
    \lim_{\delta \to 0} \mu_h(B_\delta(\bd A)) = 0 \quad \text{a.s.}
\end{equation}
Since the LQG metric induces the Euclidean topology, another equivalent condition is
\begin{equation}\label{eqn:sufficient-condition-lqg} 
    \lim_{\delta \to 0} \mu_h(\mathcal B_\delta(\bd A; D_h)) = 0 \quad \text{a.s.} 
\end{equation}

\begin{rem}\label{rem:small-quantum-dim}
    One sufficient condition for \eqref{eqn:sufficient-condition} is that the LQG Minkowski dimension of $A$ is bounded above by $d_\gamma-\zeta$ for some deterministic constant $\zeta>0$. That is, $\lim_{\ep \to 0} \ep^{d_\gamma-\zeta} N_\ep(\bd A;D_h) = 0$ almost surely. Note that this condition does not reference the LQG measure $\mu_h$. Indeed, suppose that $h$ is a whole-plane GFF with $h_1(0)=0$. Theorem~\ref{thm:metric-ball-volume} states that 
    \begin{equation}
       \sup_{\ep \in (0,1)} \sup_{z \in B_r(0)} \frac{\mu_h(\mathcal B_\ep(z;D_h))}{\ep^{d_\gamma - \frac{\zeta}{2}}} < \infty \quad \text{a.s.}
    \end{equation}
    Then, on the event $A \subset B_r(0)$,
    \begin{equation}
       \mu_h(\bd A) \leq \limsup_{\ep \to 0} \left( N_\ep(\bd A;D_h)\cdot \sup_{z \in B_r(0)} \mu_h(\mathcal B_\ep(z;D_h) )\right) = 0 \quad \text{a.s.}
    \end{equation}  
    Letting $r\to \infty$, we obtain \eqref{eqn:sufficient-condition} since $A$ is almost surely bounded.
    This implication holds even when we replace $h$ with a whole-plane GFF plus continuous function because $\mu_{h+f}$ is almost surely absolutely continuous with respect to $\mu_h$ for any random continuous function $f:\C\to \R$.
\end{rem}

\begin{rem}\label{rem:determinsitic-lebesgue-zero}
    Let $A\subset \C$ be a deterministic bounded Borel set whose boundary has zero Lebesgue measure, and let $h$ be a whole-plane GFF plus continuous function. We claim that $A$ satisfies \eqref{eqn:sufficient-condition}. Indeed, let $r > 0$ such that $\overline A\subset B_r(0)$. Let $h^0$ be a zero-boundary GFF on $B_r(0)$. From \cite[Proposition~1.2]{shef-kpz}, we have $\Exp[\mu_{h^0}(\bd A)] = \int_{\bd A} \mathrm{crad}(z;B_r(0))^{\gamma^2/2}\,d^2z$ where $\mathrm{crad}(z;B_r(0))$ is the conformal radius of $B_r(0)$ viewed from $z$. Since the Lebesgue measure of $\bd A$ is zero, we get $\Exp[\mu_{h^0}(\bd A)] = 0$ and $\mu_{h^0}(\bd A) = 0$ almost surely. By the Markov property of the whole-plane GFF~\cite[Proposition 2.8]{ig4}, we can couple $h^0$ and $h$ so that $(h-h^0)|_{B_r(0)}$ is a continuous function. Since adding a continuous function to the field results in weighting the Liouville measure by a continuous function, we also get the desired statement for $h$.
\end{rem}

\subsection{General sets on the quantum cone}

In Proposition~\ref{prop:main-thm-qc-sle}, we found that whenever $s<t$ are fixed real numbers, $\mathfrak b_\ep^{-1} N_\ep(\eta[s,t];D_{h^\gamma})$ converges in probability to $t-s=\mu_{h^\gamma}(\eta[s,t])$. The goal of this subsection is to replace $\eta[s,t]$ with more general subsets of $\C$; in particular, we consider bounded Borel sets satisfying \eqref{eqn:sufficient-condition}. We can also consider random subsets of $\C$ coupled with the field $h^\gamma$. The idea is to bound such a set from inside and outside by unions of finitely many SLE cells --- i.e., $\bigcup_{I\subset \mathcal I} \eta(I)$ where $\mathcal I\subset \mathcal I_\Q := \{[s,t]: s,t\in\Q, s<t\}$ is a random collection of finitely many closed intervals with dyadic rational endpoints.

\begin{prop}\label{prop:main-thm-qc-general}
    Let $h^\gamma$ be the field of a $\gamma$-quantum cone in the circle average embedding. Let $A\subset\C$ be either a deterministic bounded Borel set or a random compact set coupled with $h^\gamma$ and $\eta$ (as in Theorem~\ref{thm:main-thm}) such that $\mu_{h^\gamma}(\bd A) = 0$ almost surely. Then,
    \begin{equation}
        \lim_{\ep \to 0} \mathfrak b_\ep^{-1} N_\ep(A;D_{h^\gamma}) = \mu_{h^\gamma}(A)
    \end{equation}
    in probability.
\end{prop}
\begin{proof}
    As introduced in the above proof idea, we first consider the case $A = \bigcup_{I\in \mathcal I}\eta(I)$ where $\mathcal I$ is a random collection of finitely many closed bounded intervals with dyadic rational endpoints. Concretely, fix a positive integer $k$ and consider a random collection $\mathcal I \subset \{[j/2^k, (j+1)/2^k]: j \in \Z\}$ coupled with $h^\gamma$ and $\eta$ which a.s.\ has finitely many elements. We work in the same way we proved the finite additivity of Minkowski content of SLE cells (Proposition~\ref{prop:finite-additivity}). Recall that $\bd_\ep \eta(I) = \{z \in \eta(I): D_{h^\gamma}(z,\bd \eta(I)) < \ep\}$. For such $\mathcal I$, for each $\ep>0$, we almost surely have
    \begin{equation}
    	\begin{aligned} &\sum_{I} \boldsymbol{1}_{\{I \in \mathcal I\}} \mathfrak b_\ep^{-1} [N_\ep(\eta(I);D_{h^\gamma}) - N_\ep(\bd_\ep \eta(I);D_{h^\gamma})] \\ & \hspace{1in} \leq \mathfrak b_\ep^{-1} N_\ep\left(\bigcup_{I\in \mathcal I}\eta(I);D_{h^\gamma}\right) \leq \sum_{I} \boldsymbol{1}_{\{I \in \mathcal{I}\}} \mathfrak b_\ep^{-1} N_\ep(\eta(I);D_{h^\gamma})\end{aligned}
    \end{equation}
    where the two sums are over all intervals $I$ of the form $[j/2^k, (j+1)/2^k]$. Both sums converge in probability as $\ep \to 0$ to $\sum_{I \in \mathcal I} \mu_{h^\gamma}(\eta(I)) = \mu_{h^\gamma}(\bigcup_{I\in \mathcal I}\eta(I))$ since each term in these sums converges in probability to $\boldsymbol{1}_{\{I \in \mathcal I\}} \eta(I)$ as seen in Proposition~\ref{prop:main-thm-qc-sle} and Lemma~\ref{lem:sle-nbhd-zero-content}. Therefore,
    \begin{equation} \label{eqn:random-union-minkowski-content}
        \lim_{\ep \to 0} \mathfrak b_{\ep}^{-1} N_{\ep} \left(\bigcup_{I\in \mathcal I}\eta(I);D_{h^\gamma}\right) = \mu_{h^\gamma} \left(\bigcup_{I\in \mathcal I}\eta(I)\right) \quad \text{in probability.} 
    \end{equation}

    Let $A\subset \C$ now be a deterministic bounded Borel set or a random compact set coupled with $h^\gamma$ and $\eta$ such that $\mu_{h^\gamma}(\bd A) = 0$ almost surely. For each integer $k$, let 
    \begin{equation} 
        \mathcal I_k := \left\{\left[j/2^k, (j+1)/2^k\right]: j \in \Z, \eta\left(\left[j/2^k, (j+1)/2^k\right]\right) \subset A\right\}
    \end{equation}
    and
    \begin{equation} 
        \mathcal J_k := \left\{\left[j/2^k, (j+1)/2^k\right]: j \in \Z, \eta\left(\left[j/2^k, (j+1)/2^k\right]\right) \cap A \neq \varnothing\right\}.
    \end{equation}
    That is, $\bigcup_{I\in\mathcal I_k} \eta(I)$ and $\bigcup_{I \in \mathcal J_k} \eta(I)$ are approximations of $A$ from inside and outside, respectively. For each $j$ and $k$, note that $\eta([j/2^k, (j+1)/2^k])$ is a random compact subset of $\C$ that is measurable w.r.t.\ the Borel $\sigma$-algebra generated by the Hausdorff distance on compact subsets of $\C$; hence, $\mathcal I_k$ and $\mathcal J_k$ are measurable. Almost surely, $\mathcal I_k$ and $\mathcal J_k$ contain finitely many intervals since $A$ is bounded and $\lim_{t\to \pm\infty}\eta(t) = \infty$. Since $\bigcup_{I \in \mathcal I_k} \eta(I) \subset A \subset \bigcup_{I\in \mathcal J_k} \eta(I)$, we have 
    \begin{equation}
        \mathfrak b_\ep^{-1} N_\ep \left( \bigcup_{I \in \mathcal I_k} \eta(I);D_{h^\gamma}\right) \leq \mathfrak b_\ep^{-1} N_\ep(A) \leq \mathfrak b_\ep^{-1} N_\ep\left(\bigcup_{I \in \mathcal J_k} \eta(I);D_{h^\gamma}\right)  
    \end{equation}
    almost surely for each $\ep>0$. As in \eqref{eqn:random-union-minkowski-content}, we have
    \begin{equation} \label{eqn:approx-from-inside}
        \lim_{\ep \to 0}\mathfrak b_\ep^{-1} N_\ep \left( \bigcup_{I \in \mathcal I_k} \eta(I);D_{h^\gamma}\right)= \mu_{h^\gamma} \left( \bigcup_{I \in \mathcal I_k} \eta(I) \right) 
    \end{equation} and 
    \begin{equation} \label{eqn:approx-from-outside} 
    	\lim_{\ep \to 0}\mathfrak b_\ep^{-1} N_\ep \left( \bigcup_{I \in \mathcal J_k} \eta(I);D_{h^\gamma}\right) = \mu_{h^\gamma} \left( \bigcup_{I \in \mathcal J_k} \eta(I) \right) 
    \end{equation}
    in probability for each $k$.
    
    We claim 
    \begin{equation}  
        \lim_{k\to \infty} \mu_{h^\gamma} \left( \bigcup_{I \in \mathcal I_k} \eta(I)\right) = \mu_{h^\gamma} (A) =\lim_{k\to \infty} \mu_{h^\gamma}\left(\bigcup_{I \in \mathcal J_k} \eta(I)\right) 
    \end{equation}
    almost surely, which, in combination with \eqref{eqn:approx-from-inside} and \eqref{eqn:approx-from-outside}, completes the proof of the proposition. Note that for each integer $k$,
    \begin{equation}  
        \mu_{h^\gamma} \left( \bigcup_{I \in \mathcal I_k} \eta(I)\right) \leq \mu_{h^\gamma} (A) \leq \mu_{h^\gamma}\left(\bigcup_{I \in \mathcal J_k} \eta(I)\right) . 
    \end{equation}
    Because $\bigcup_{I \in \mathcal I_k} I$ increases and $\bigcup_{I \in \mathcal J_k} I$ decreases as $k$ increases, it suffices to prove 
    \begin{equation} 
        \lim_{k\to \infty} \left[ \mu_{h^\gamma}\left(\bigcup_{I \in \mathcal J_k} \eta(I)\right) - \mu_{h^\gamma}\left(\bigcup_{I \in \mathcal I_k} \eta(I)\right) \right] = \lim_{k\to \infty} \mu_{h^\gamma}\left(\bigcup_{I \in \mathcal J_k \setminus \mathcal I_k} \eta(I)\right) = 0
    \end{equation}
    almost surely. Note that 
    \begin{equation}  
        \mathcal J_k \setminus \mathcal I_k = \left\{\left[j/2^k, (j+1)/2^k\right]: j \in \Z, \eta\left(\left[j/2^k, (j+1)/2^k\right]\right) \cap \bd A \neq \varnothing  \right\}. 
    \end{equation}
    Since $A$ is bounded and $t \mapsto \eta(t)$ is continuous, 
    \begin{equation}  
        \lim_{k\to \infty} \max_{I \in \mathcal J_k} \mathrm{diam}(\eta(I);D_{h^\gamma}) = 0 
    \end{equation}
    with probability one. Hence, 
    \begin{equation}  
        \lim_{k\to \infty} \mu_{h^\gamma}\left( \bigcup_{I \in \mathcal J_k \setminus \mathcal I_k} \eta(I) \right) \leq \lim_{\delta\to 0} \mu_{h^\gamma}( \mathcal B_\delta(\bd A;D_{h^\gamma})).
    \end{equation}
    almost surely. As discussed in the beginning of this section, because $D_{h^\gamma}$ induces the Euclidean topology, $\Prob\{\mu_{h^\gamma}(\bd A) = 0\}=1$ implies $\lim_{\delta\to 0} \mu_{h^\gamma}( \mathcal B_\delta(\bd A;D_{h^\gamma}))=0$ almost surely and thus our claim.
\end{proof}

\subsection{Generalization to other GFF variants}

The final technical detail that allows us to replace $h^\gamma$ in Proposition~\ref{prop:main-thm-qc-general} with any whole-plane GFF plus a continuous function is that $\mathfrak b_\ep$ is regularly varying with index $-d_\gamma$.

\begin{proof}[Proof of Proposition~\ref{prop:scaling-constant-properties}]
    We proved \eqref{eqn:scaling-constant-bounds} in Corollary~\ref{cor:normalization-constant-comparison}. It remains to show \eqref{eqn:scaling-constant-regularly-varying}.
    Recall from Corollary~\ref{cor:covering-number-relation} that $N_{\ep}(\eta[0,r^{-d_\gamma}];D_{h^\gamma}) \stackrel{d}{=} N_{r\ep}(\eta[0,1];D_{h^\gamma})$ for every $\ep>0$. Hence, by Proposition~\ref{prop:main-thm-qc-sle},
    \begin{equation} 
        r^{-d_\gamma} = \frac{\lim_{\ep \to 0} \mathfrak b_\ep^{-1} N_\ep(\eta[0,r^{-d_\gamma}];D_{h^\gamma})}{\lim_{\ep \to 0}\mathfrak b_{r\ep}^{-1}  N_{r\ep}(\eta[0,1];D_{h^\gamma})} = \frac{\lim_{\ep \to 0} \mathfrak b_\ep^{-1} N_{r\ep}(\eta[0,1];D_{h^\gamma})}{\lim_{\ep \to 0}\mathfrak b_{r\ep}^{-1}  N_{r\ep}(\eta[0,1];D_{h^\gamma})} = \lim_{\ep \to 0} \frac{\mathfrak b_{r\ep}}{\mathfrak b_\ep}.
    \end{equation}
\end{proof}

\begin{proof}[Proof of Theorem~\ref{thm:main-thm}]
    The proof proceeds by standard arguments based on the fact that the Minkowski content depends locally on $h$ and behaves nicely under adding a constant to $h$. Specifically, if $f$ is a deterministic smooth function with compact support on $\C$ and $h$ is a whole-plane GFF with circle average normalization $h_r(z)=0$ on a fixed circle $\bd B_r(z)$ disjoint from the support of $f$, then the laws of $h+f$ and $h$ are mutually absolutely continuous \cite[Proposition~2.9]{ig4}. Using this fact, we transfer our results from a $\gamma$-quantum cone to a whole-plane GFF in steps, introducing more generality in our choice of the field $h$ and the deterministic bounded Borel or random compact set $A$. In each stage, we assume $\mu_h(\bd A) = 0$ almost surely.
    
    \begin{enumerate}
    \item \label{general-step1} \textit{The field $h$ is a whole-plane GFF with $h_1(0)=0$, and $A\subset B_{1/4}(1/2)$ almost surely.}

    Note that $\mu_h(A)$, $\mu_h(\bd A)$, and $\lim_{\ep \to 0} N_\ep(A;D_h)$ are a.s.\ determined by $h|_{B_{1/3}(1/2)}$. Since $h - \gamma \log |\cdot|$ and $h^\gamma$ agree in law when restricted to $\ud$ (Definition~\ref{def:quantum-cone}), the laws of $h^\gamma$ and $h$ restricted to $B_{1/3}(1/2)$ are mutually absolutely continuous. We deduce from Proposition~\ref{prop:main-thm-qc-general} that $\lim_{\ep\to 0} \mathfrak b_\ep^{-1}N_\ep(A;D_h) = \mu_h(A)$ in probability.

    \item \label{general-step2} \textit{The field $h$ is a whole-plane GFF with normalization $h_{4r}(-2r)=0$, and $A \subset B_r(0)$ a.s.\ for some fixed $r>0$.} 
    
    Denote $\tilde h:= h(4r\cdot-2r)$. By \eqref{eqn:gff-translate}, $\tilde h$ is a whole-plane GFF with $\tilde h_1(0)=0$. Since $A \subset B_r(0)$, we have $(4r)^{-1}A + 1/2 \subset B_{1/4}(1/2)$. We saw in Step~\ref{general-step1} that if $\mu_{\tilde h}((4r)^{-1}A + 1/2) = 0$ almost surely, then 
    \begin{equation}\label{eqn:scaled-translated-mc-convergence}
        \lim_{\ep \to 0} \mathfrak b_\ep^{-1} N_\ep((4r)^{-1}A + 1/2;D_{\tilde h}) = \mu_{\tilde h}((4r)^{-1}A + 1/2)
    \end{equation}
    in probability. From the coordinate change axiom \eqref{eqn:coordinate-change-axiom} for deterministic translation and scaling, we almost surely have
    \begin{equation}
        D_h(4ru-2r, 4rv-2r) = D_{\tilde h + Q\log (4r)}(u,v) = (4r)^{\xi Q} D_{\tilde h}(u,v) \quad \forall u,v\in \mathbb C.
    \end{equation}
    Then, almost surely,
        \begin{equation}\label{eqn:covering-number-wpgff-scaled}
        N_{\ep(4r)^{\xi Q}}(A;D_h) = N_\ep((4r)^{-1}A +1/2;D_{\tilde h})\quad \forall \ep>0.
    \end{equation}
    On the other hand, from the coordinate change rule \eqref{eqn:lqg-measure-coordinate-change} for the LQG measure, almost surely,
    \begin{equation}\label{eqn:lqg-measure-wpgff-scaled}
        \mu_h(A) = \mu_{\tilde h + Q\log (4r) }((4r)^{-1}(A+2r)) = (4r)^{\gamma Q} \mu_{\tilde h}((4r)^{-1}A+1/2).
    \end{equation}
    One consequence of \eqref{eqn:lqg-measure-wpgff-scaled} is that $\mu_h(\bd A) = 0$ implies $\mu_{\tilde h}((4r)^{-1}(\bd A) + 1/2) = 0$. Hence, we deduce from \eqref{eqn:scaled-translated-mc-convergence} combined with \eqref{eqn:covering-number-wpgff-scaled} and \eqref{eqn:lqg-measure-wpgff-scaled} that
    \begin{equation}
        \lim_{\ep \to 0} \mathfrak b_\ep^{-1} N_{\ep (4r)^{\xi Q}} (A;D_h) = (4r)^{-\gamma Q} \mu_h(A)
    \end{equation}
    in probability. Since we have 
    \begin{equation}
        \lim_{\ep \to 0} \frac{\mathfrak b_{\ep (4r)^{\xi Q}}}{\mathfrak b_\ep} = (4r)^{-\gamma Q}
    \end{equation}
    from Proposition~\ref{prop:scaling-constant-properties}, we conclude that $\lim_{\ep\to0} \mathfrak b_\ep^{-1} N_\ep(A;D_h) = \mu_h(A)$ in probability.

    \item \textit{The field $h$ is any whole-plane GFF plus a continuous function, but $A\subset B_r(0)$ for some fixed $r>0$.}

    We can write $h = \hat h + f$ where $\hat h$ is a whole-plane GFF with normalization $\hat h_{4r}(-2r) = 0$ and $f$ is a random continuous function (which is not necessarily independent from $\hat h$). Again, the goal is to show that for any Borel $A\subset B_r(0)$ with $\mu_{\hat h+f}(\bd A) = 0$ almost surely,
    \begin{equation} \label{eqn:main-thm-gffplus} 
        \lim_{\ep \to 0} \mathfrak b_\ep^{-1}N_\ep(A;D_{\hat h+f}) = \mu_{\hat h+f}(A) 
    \end{equation} 
    in probability. Since $f$ is a.s.\ bounded on $\overline{B_r(0)}$, we have $\mu_{\hat h}(\bd A) = 0$ almost surely. The idea is to use that $f$ is a.s.\ locally uniformly continuous. To this end, consider the collection of dyadic squares $\mathcal S_k = \{ [\frac{m}{2^k}, \frac{m+1}{2^k}] + [\frac{n}{2^k}, \frac{n+1}{2^k}]i \subset \overline{B_r(0)}: m,n\in \Z \}$ and define 
    \begin{equation}
    \mathcal A_k = \{ S \in \mathcal S_k: S \subset A \} \quad \text{and} \quad \mathcal A^k = \{S \in \mathcal S_k: S \cap A \neq \varnothing\}
    \end{equation}
    to be the subcollections whose unions approximate $A$ from inside and outside, respectively. For each dyadic square $S = [\frac{m}{2^k}, \frac{m+1}{2^k}] + [\frac{n}{2^k}, \frac{n+1}{2^k}]i \in \mathcal S_k$, denote its union with all adjacent dyadic squares by $\hat S = [\frac{m-1}{2^k}, \frac{m+2}{2^k}] + [\frac{n-1}{2^k}, \frac{n+2}{k}]i$. Denote 
    \begin{equation}
        m_{S} = \min_{\hat S} f \quad \text{and} \quad M_{S} = \max_{\hat S} f . 
    \end{equation}
    In Step~\ref{general-step2}, we established \eqref{eqn:main-thm-gffplus} when $f=0$. By Remark~\ref{rem:determinsitic-lebesgue-zero},  $\mu_{\hat h}(\bd S) = 0$ a.s.\ for all $S\in \bigcup_{k\geq 1} \mathcal S_k$. Also note $\lim_{\ep\to 0} \mathfrak b_{\ep \exp(-\xi M_S)}/\mathfrak b_\ep = e^{\gamma M_S}$ from \eqref{eqn:scaling-constant-regularly-varying}. Given any sequence of $\ep$ decreasing to 0, we can choose a subsequence $\ep_n$ such that almost surely,
    \begin{equation}\label{eqn:convergence-from-above}
        e^{\gamma M_S} \mu_{\hat h}(S) = \lim_{n\to \infty} \mathfrak b_{\ep_n}^{-1} N_{\ep_n \exp(-\xi M_S)}(S;D_{\hat h}) \quad \text{for all } S \in \bigcup_{k\geq 1} \mathcal S_k.
    \end{equation}
    For $S = [\frac{m}{2^k}, \frac{m+1}{2^k}] + [\frac{n}{2^k}, \frac{n+1}{2^k}]i\in \mathcal S_k$ and $j\in \N$, denote \begin{equation}S_j := \left[\frac{m}{2^k} + \frac{1}{2^{k+j}}, \frac{m+1}{2^k}-\frac{1}{2^{k+j}}\right] + \left[\frac{n}{2^k}+\frac{1}{2^{k+j}}, \frac{n+1}{2^k}-\frac{1}{2^{k+j}}\right]i. \end{equation} We again have $\mu_{\hat h}(\bd S_j) = 0$ by Remark~\ref{rem:determinsitic-lebesgue-zero}. Take a further subsequence of $\ep_n$ such that, almost surely, 
    \begin{equation}\label{eqn:convergence-from-below}
        e^{\gamma m_S} \mu_{\hat h}(S_j) = \lim_{n\to \infty} \mathfrak b_{\ep_n}^{-1} N_{\ep_n \exp(-\xi M_S)}(S_j;D_{\hat h}) \quad \text{for all } S \in \bigcup_{k\geq 1} \mathcal S_k,\; j\in \N.
    \end{equation}
    By the Weyl scaling axiom \eqref{eqn:weyl-scaling}, for each $S\in \mathcal S_k$ and $j\in \N$, we almost surely have 
    \begin{equation} \label{eqn:gff-plus-min-max-covering}
    N_{\ep \exp(-\xi m_S)}(S_j;D_{\hat h}) \leq N_\ep(S_j;D_h) \leq N_\ep(S;D_h) \leq N_{\ep \exp(-\xi M_S)}(S;D_{\hat h}) 
    \end{equation}
    for all sufficiently small $\ep>0$. (The threshold is random; it depends on $\hat h$ and $f$.) 
    On one hand, from \eqref{eqn:convergence-from-above} and \eqref{eqn:gff-plus-min-max-covering}, we almost surely have
    \begin{equation}
    \begin{aligned}\label{eqn:covering-number-above-summed}
        \sum_{S \in \mathcal A^k} e^{\gamma M_S} \mu_{\hat h}(S) &= \sum_{S \in \mathcal A^k} \lim_{n\to\infty} \mathfrak b_{\ep_n}^{-1} N_{\ep_n \exp(-\xi M_S)}(S;D_{\hat h}) \geq \sum_{S \in \mathcal A^k} \lim_{n\to\infty} \mathfrak b_{\ep_n}^{-1} N_{\ep_n}(S;D_h) \\&= \lim_{n\to\infty} \mathfrak b_{\ep_n}^{-1} \sum_{S \in \mathcal A^k} N_{\ep_n}(S;D_h) \geq \limsup_{n\to\infty} \mathfrak b_{\ep_n}^{-1} N_{\ep_n}(A;D_h).
    \end{aligned}
    \end{equation}
    On the other hand, for each $k$ and $j$, we almost surely have 
    \begin{equation}
        \sum_{S \in \mathcal A_k} N_\ep(S_j;D_h) \leq N_\ep(A;D_h)
    \end{equation} 
    for all sufficiently small $\ep>0$, since $\inf_{S, \tilde S \in \mathcal S_k, S\neq \tilde S} D_h(S_j, \tilde S_j) > 0$. Combining this with \eqref{eqn:convergence-from-below} and \eqref{eqn:gff-plus-min-max-covering}, we almost surely have
    \begin{equation}\label{eqn:covering-number-below-summed}
    \begin{aligned}
        \sum_{S \in \mathcal A_k} e^{\gamma m_S} \mu_{\hat h}(S_j) &= \sum_{S \in \mathcal A_k} \lim_{n\to\infty} \mathfrak b_{\ep_n}^{-1} N_{\ep_n \exp(-\xi m_S)}(S_j;D_{\hat h}) \leq \sum_{S \in \mathcal A_k} \lim_{n\to\infty} \mathfrak b_{\ep_n}^{-1} N_{\ep_n}(S_j;D_h) \\&= \lim_{n\to\infty} \mathfrak b_{\ep_n}^{-1} \sum_{S \in \mathcal A_k} N_{\ep_n}(S_j;D_h) \leq \liminf_{n\to\infty} \mathfrak b_{\ep_n}^{-1} N_{\ep_n}(A;D_h).
    \end{aligned}
    \end{equation}
    For each $S\in \mathcal S_k$, since $\mu_{\hat h}(\bd S) = 0$, we have $\lim_{j\to \infty} \mu_{\hat h}(S_j) = \mu_{\hat h}(S)$ almost surely.
    Combining \eqref{eqn:covering-number-above-summed} and \eqref{eqn:covering-number-below-summed} and letting $j\to \infty$, we almost surely have that 
    \begin{equation}\label{eqn:minkowski-dyadic-sum}
        \begin{aligned} \sum_{S\in \mathcal A_k} e^{\gamma m_S} \mu_{\hat h}(S) &\leq \liminf_{n\to\infty} \mathfrak b_{\ep_n}^{-1} N_{\ep_n}(A;D_h)\\&\leq \limsup_{n\to \infty} \mathfrak b_{\ep_n}^{-1} N_{\ep_n}(A;D_h) \leq \sum_{S\in \mathcal A^k} e^{\gamma M_S} \mu_{\hat h}(S).
        \end{aligned}
    \end{equation}
    Since $\mathcal A^k \setminus \mathcal A_k = \{S \in \mathcal S_k: S \cap \bd A \neq \varnothing\}$, it follows from the  equivalence between \eqref{eqn:sufficient-condition} and \eqref{eqn:sufficient-condition-euclidean} that $\lim_{k\to \infty}\sum_{S \in \mathcal A^k \setminus \mathcal A_k} \mu_{\hat h}(S) = 0$. In addition, 
    \begin{equation}
    M:= \sup_{k \in \mathbb N} \max_{S \in \mathcal S_k} M_S \leq \sup_{B_{r+1}(0)} f < \infty \quad \text{and} \quad \lim_{k\to \infty} \max_{S \in \mathcal S_k}(\exp(\gamma M_S) - \exp(\gamma m_S)) = 0
    \end{equation}
    almost surely because $f$ is and uniformly continuous on $B_R(0)$. Hence, the right-hand side of the inequality
    \begin{equation} 
        \sum_{S \in \mathcal A^k} e^{\gamma M_S} \mu_{\hat h}(S) - \sum_{S\in \mathcal A_k} e^{\gamma m_S} \mu_{\hat h}(S) \leq e^{\gamma M} \sum_{S \in \mathcal A^k \setminus \mathcal A_k} \mu_{\hat h}(S) + \sum_{S \in \mathcal A_k} (e^{\gamma M_S} - e^{\gamma m_S}) \mu_{\hat h}(S) .
    \end{equation}
    converges almost surely to 0 as $k\to \infty$. By this together with~\eqref{eqn:minkowski-dyadic-sum}, since $\ep_n$ is a subsequence of an arbitrary sequence of $\ep$ decreasing to 0,
    \begin{equation}\label{eqn:dyadic-sum-penultimate} \lim_{k\to \infty} \sum_{S \in \mathcal A_k} e^{\gamma m_S} \mu_{\hat h}(S) = \lim_{\ep \to 0} \mathfrak b_\ep^{-1}N_\ep(A; D_h) = \lim_{k\to \infty} \sum_{S \in \mathcal A^k} e^{\gamma M_S} \mu_{\hat h}(S) \quad \text{in probability.} \end{equation}
    On the other hand, since $\mu_h(\bd S)=0$ and thus $\mu_h = \mu_{\hat h+f}$ is finitely additive for $S\in \mathcal S_k$,
    \begin{equation}\label{eqn:dyadic-sum-final} \sum_{S \in \mathcal A_k} e^{\gamma m_S} \mu_{\hat h}(S) \leq \sum_{S \in \mathcal A_k} \mu_h(S) \leq \mu_h(A) \leq \sum_{S \in \mathcal A^k} \mu_h(S) \leq \sum_{S \in \mathcal A^k} e^{\gamma M_S} \mu_{\hat h}(S)\end{equation}
    almost surely. We found in \eqref{eqn:dyadic-sum-penultimate} the leftmost and the rightmost terms of this inequality converge to the same limit in probability as $k\to \infty$, which must be equal to $\mu_h(A)$ almost surely. Therefore, we conclude that
    \begin{equation}
        \mu_h(A) = \lim_{\ep \to 0} \mathfrak b_\ep^{-1}N_\ep(A; D_h)
    \end{equation}
    in probability.
    
    \item \textit{The field $h$ is any whole-plane GFF plus a continuous function and $A \subset \mathbb C$ is any random bounded Borel set.} 
    
    Let $E_k$ be the event $\{A \subset B_k(0)\}$. Given any sequence of $\ep$ decreasing to 0, using a diagonal argument we can find a subsequence $\ep_n$ such that $\lim_{n\to \infty} \mathfrak b_{\ep_n}^{-1} N_{\ep_n}(A;D_h) = \mu_h(A)$ a.s.\ on $E_k$ for all integer $k$. Since $\Prob(\bigcup_{k\in \mathbb N}E_k) = 1$, we conclude that the limit $\lim_{n\to \infty} \mathfrak b_\ep^{-1}N_\ep(A;D_h) = \mu_h(A)$ holds in probability. 
    \end{enumerate}
    
    This completes the proof of \eqref{eqn:main-theorem} for general Gaussian field $h$ and set $A$.
\end{proof}

We finally show that the pointed metric space $(\C, 0, D_h)$ almost surely determines the marked quantum surface $(\C,h,0)$.

\begin{proof}[Proof of Corollary~\ref{cor:forgetting-parameterization}] 

    Let $\mathbb L$ (resp.\ $\mathbb M$) be the space of isometry classes of complete, locally compact length spaces with a marked point (resp.\ a marked point and a locally finite Borel measure), endowed with the local Gromov--Hausdorff (resp.\ local Gromov--Hausdorff--Prokhorov) topology. We consider $\mathbb{L}$ (resp.\ $\mathbb{M}$) as a \textit{complete} probability space endowed with the probability measure $\Prob_h^{\mathbb{L}}$ (resp.\ $\Prob_h^{\mathbb{M}}$) corresponding to $(\C, 0, D_h)$ (resp.\ $(\C, 0, D_h, \mu_h))$. Let $f: \mathbb{M} \to \mathbb{L}$ be the natural projection. We claim that $f$ is injective on a set $E\subset \mathbb{M}$ with $\Prob_h^{\mathbb{M}}(E) = 1$. If so, since $\mathbb{M}$ is a Polish space \cite{adh-ghp-distance}, by the measurable selection theorem (see, e.g., \cite[Theorem~6.9.1]{bogachev-measure-theory}), there is a measurable function $g:\mathbb{L} \to \mathbb{M}$ such that $g \circ f$ is the identity map on $E$. This means that $(\C, 0, D_h)$ almost surely determines $(\C, 0, D_h, \mu_h)$. 

    To this end, consider a sequence $A_1, A_2, \dots$ of measurable functions from continuous metrics on $\C$ to compact subsets of $\C$. Let $\Pi(D):= \{A_1(D), A_2(D),\dots\}$ be the collection of these sets without ordering: i.e., $\Pi$ is a function from the set of continuous metrics on $\C$ to the power set of compact subsets of $\C$. For now, we prescribe the following set of properties that these functions should satisfy; we shall construct an explicit sequence of such maps at the end of the proof.
    \begin{itemize}
    	\item Almost surely, $\mu_h(\bd A_j(D_h)) = 0$ for every $j$.
    	\item Almost surely, $\Pi(D_h)$ is a $\pi$-system which generates the Borel $\sigma$-algebra on $\C$.
    	\item If $D$ and $D'$ are continuous metrics on $\C$ such that $\phi:(\C,D)\to(\C,D')$ is an isometry preserving 0, then $\phi(\Pi(D)) = \Pi(D')$.
    \end{itemize}

    By Theorem~\ref{thm:main-thm} and a standard diagonalization argument, the first property implies that there is a sequence $\ep_n$ decreasing to 0 such that, almost surely, 
    \begin{equation}\label{eqn:pi-system-convergence}
    	\lim_{n\to\infty} \mathfrak b_{\ep_n}^{-1} N_{\ep_n}(A_j(D_h);D_h)  = \mu_h(A_j(D_h)) \quad \text{for every } j.
    \end{equation}
    Hence, there exists a Borel subset $\widetilde E$ of the product space of continuous metrics on $\C$ and Borel measures on $\C$ with $\Prob\{(D_h,\mu_h)\in \widetilde E\}=1$ on which the first two bulleted properties as well as \eqref{eqn:pi-system-convergence} hold. By the $\pi$-$\lambda$ theorem, if two Borel measures agree on a $\pi$-system, then they must be identical. Hence, the projection $(D,\mu) \mapsto D$, where $D$ is a continuous metric on $\C$ and $\mu$ is a Borel measure on $\C$, is injective on $\widetilde E$.

    Let $E$ be the image of $\widetilde E$ under the ``forget the embedding'' map: i.e., $(D,\mu) \mapsto (\C,0,D,\mu)$ where the latter is a pointed metric measure space.\footnote{It is straightforward to check that the natural embedding $(D,\mu) \mapsto (\C,0,D,\mu)$ is a continuous map whose domain is a Polish space, so $E$ is an analytic set. Since $\mathbb P_h^{\mathbb M}$ is a complete probability measure, $E$ is $\mathbb P_h^{\mathbb M}$-measurable (see, e.g., \cite[Theorem~21.10]{descriptive-set-theory}). }
    Let $(D,\mu), (D',\mu') \in \widetilde E$ be any two elements which are mapped into the same pointed metric space $(\C,0,D) = (\C,0,D')$ under the natural embedding $(D,\mu) \mapsto (\C,0,D)$. That is, there is an isometry $\phi:(\C,D) \to (\C,D')$ fixing 0. By the last bulleted property, $\Pi(D') = \phi(\Pi(D))$ and 
    \begin{equation}
    	\mu'(\phi(A)) = \lim_{n\to\infty} \mathfrak b_{\ep_n}^{-1} N_{\ep_n}(\phi(A);D') = \lim_{n\to\infty} \mathfrak b_{\ep_n}^{-1} N_{\ep_n}(A;D) = \mu(A)
    \end{equation}
    for every $A \in \Pi(D)$. Since $\Pi(D')$ is a $\pi$-system, we conclude that $\mu' = \phi_*\mu$. That is, the pointed metric measure spaces $(\C,0,D,\mu)$ and $(\C,0,D',\mu')$ are identical. Therefore, the natural projection $f:\mathbb M \to \mathbb L$, $(\C,0,D,\mu)\mapsto (\C,0,D)$ is injective on $E$.
    
    It remains to show the existence of measurable maps $A_1, A_2, \dots$ with the bulleted properties. For each continuous metric $D$ on $\C$, define
    \begin{equation}
    	\mathcal Z(D) = \{z \in \C: \text{there are exactly three distinct } D\text{-geodesics from } 0 \text{ to } z\}.
    \end{equation}
    Let $\mathcal K(D)$ be the collection of all compact $K \subset \C$ that are of the following form: 
    \begin{enumerate}
    	\item Let $z_1,z_2,\dots,z_n$ be any finite subset of $\mathcal Z(D)$;
    	\item For $k=1,\dots,n$, let $\eta_k$ be any $D$-geodesic between $z_{k-1}$ and $z_k$ (where $z_0 = z_n$);
    	\item Let $\eta = \eta_1 \cup \eta_2 \cup \dots \eta_n$ be the closed curve formed by concatenating the geodesics;
    	\item Let $K$ be the union of $\eta$ and all bounded components of $\C \setminus \eta$.
    \end{enumerate}
    Finally, let $\Pi(D)$ consist of any finite intersections of sets in $\mathcal K(D)$. These sets are defined only using the pointed metric structure $(\C,0,D)$ so that if $\phi:(\C,D) \to (\C,D')$ is an isometry fixing 0, then $\phi(\mathcal Z(D)) = \mathcal Z(D')$, $\phi(\mathcal K(D)) = \mathcal K(D')$, and $\phi(\Pi(D))=\Pi(D')$.
    
    For the $\gamma$-LQG metric $D_h$, almost surely, the set $\mathcal Z(D_h)$ is countable and dense \cite[Theorem~1.2]{gwynne-geodesic-network}, and any two points in $\mathcal Z(D_h)$ are joined by finitely many $D_h$-geodesics \cite[Theorem~1.7]{gwynne-geodesic-network}. Moreover, since $D_h$ is a.s.\ a continuous metric, we can a.s.\ find for every $z\in \Q^2$ and $r\in \Q_{>0}$ a set $K \in \mathcal K(D_h)$ such that $K$ contains the Euclidean ball $B_r(z)$ and is contained in the Euclidean ball $B_{2r}(z)$. Hence, $\Pi(D_h)$ is a.s.\ a countable $\pi$-system generating the Borel $\sigma$-algebra on $\C$. Using the measurable selecction theorem inductively, we can find a sequence of measurable functions $A_1,A_2,\dots$ such that $\Pi(D) = \{A_1(D), A_2(D),\dots\}$ almost surely.
    
    Finally, we claim that the quantum Minkowski dimension of $\bd A_j(D_h)$ is a.s.\ no more than 1 for every $j$. Since $D_h$ is a continuous metric on $\C$, the lengths of $D_h$-geodesics between all points in $\mathcal Z(D_h)$ (which are equal to the $D_h$-distances between these points) are finite. Because $(\C,D_h)$ is a length space, a curve $\eta$ with $D_h$-length $L<\infty$ satisfies $N_\ep(\eta;D_h) \leq L/\ep $. This is since if $\eta$ is parameterized by $D_h$-length, then $\{\mathcal B_\ep(P(\ep k);D_h): k = 1,2,\dots,\lfloor L/\ep \rfloor\}$ covers $\eta$. For every $j$, since $\bd A_k(D_h)$ is a finite union of $D_h$-geodesics between points in $\mathcal Z(D_h)$, its LQG Minkowski dimension is at most 1 as claimed. Therefore, by Remark~\ref{rem:small-quantum-dim}, we have that $\mu_h(\bd A_j(D_h)) = 0$ a.s.\ for every $j$.
    
    As these functions $A_1,A_2,\dots$ satisfy all three bulleted properties stated above, we conclude that $(\C,0,D_h)$ a.s.\ determines $(\C,0,D_h,\mu_h)$. The rest of the corollary follows since the pointed metric measure space $(\C,0,D_h,\mu_h)$ a.s.\ determines the field $h$ up to rotation and scaling of the complex plane centered at the origin \cite[Theorem~1.3]{afs-metric-ball}. 
\end{proof}

\bibliographystyle{alphaurl}
\bibliography{lqg-minkowski-content}

\end{document}